\documentclass[reqno,11pt,letterpaper]{amsart}
\usepackage{amsmath,amssymb,amsthm,graphicx,mathrsfs,url}
\usepackage[usenames,dvipsnames]{color}
\usepackage[colorlinks=true,linkcolor=Red,citecolor=Green]{hyperref}

\def\asItem#1{\item[\hypertarget{a:#1}{(#1)}]}
\def\as#1{({\hyperlink{a:#1}{#1}})}

\def\smallsection#1{\smallskip\noindent\textbf{#1}.}

\newtheorem{theo}{Theorem}
\newtheorem{prop}{Proposition}[section]
\newtheorem{defi}[prop]{Definition}

\newtheorem{lemm}[prop]{Lemma}
\newtheorem{corr}[prop]{Corollary}

\numberwithin{equation}{section}

\DeclareMathOperator{\comp}{comp}
\DeclareMathOperator{\even}{even}
\DeclareMathOperator{\Ell}{ell}
\DeclareMathOperator{\HS}{HS}
\DeclareMathOperator{\Imag}{Im}
\DeclareMathOperator{\loc}{loc}
\DeclareMathOperator{\Op}{Op}
\DeclareMathOperator{\PSL}{PSL}
\DeclareMathOperator{\Real}{Re}

\DeclareMathOperator{\SL}{SL}
\DeclareMathOperator{\supp}{supp}
\DeclareMathOperator{\Vol}{dvol}

\DeclareMathOperator{\WF}{WF}
\DeclareMathOperator{\Tr}{Tr}
\def\WFh{\WF_h}
\def\plM{\partial\overline M}
\newcommand{\Sc}{S^{\comp}}

\usepackage{color}

\newcommand{\mc}{\mathcal}
\newcommand{\rr}{\mathbb{R}}
\newcommand{\nn}{\mathbb{N}}
\newcommand{\cc}{\mathbb{C}}
\newcommand{\hh}{\mathbb{H}}
\newcommand{\zz}{\mathbb{Z}}
\newcommand{\la}{\lambda}
\newcommand{\eps}{\varepsilon}

\newcommand{\pl}{\partial}
\newcommand{\x}{\times}

\newcommand{\til}{\widetilde}
\newcommand{\bbar}{\overline}

\newcommand{\cjd}{\rangle}
\newcommand{\cjg}{\langle}
\newcommand{\demi}{\frac{1}{2}}
\newcommand{\ndemi}{\frac{n}{2}}
\newcommand{\indic}{\operatorname{1\hskip-2.75pt\relax l}}

\title[Microlocal limits of plane waves and Eisenstein functions]%
{Microlocal limits of plane waves\\
and Eisenstein functions}
\author{Semyon Dyatlov}
\email{dyatlov@math.berkeley.edu}
\address{Department of Mathematics, Evans Hall, University of California,
Berkeley, CA 94720, USA}
\author{Colin Guillarmou}
\email{cguillar@dma.ens.fr}
\address{DMA, U.M.R. 8553 CNRS, \'Ecole Normale Superieure, 45 rue d'Ulm,
75230 Paris cedex 05, France}

\begin{document}

\begin{abstract}
We study microlocal limits of plane waves on noncompact Riemannian
manifolds $(M,g)$ which are either Euclidean or asymptotically
hyperbolic with curvature $-1$ near infinity. The plane waves
$E(z,\xi)$ are functions on $M$ parametrized by the square root of
energy $z$ and the direction of the wave, $\xi$, interpreted as a
point at infinity.  If the trapped set $K$ for the geodesic flow has
Liouville measure zero, we show that, as $z\to +\infty$, $E(z,\xi)$
microlocally converges to a measure $\mu_\xi$, in average on energy
intervals of fixed size, $[z,z+1]$, and in $\xi$. We express the rate
of convergence to the limit in terms of the classical escape rate of
the geodesic flow and its maximal expansion rate~--- when the flow is
Axiom~A on the trapped set, this yields a negative power of $z$. As an
application, we obtain Weyl type asymptotic expansions for local
traces of spectral projectors with a remainder controlled in terms of
the classical escape rate.
\end{abstract}



\maketitle

\addtocounter{section}{1}
\addcontentsline{toc}{section}{1. Introduction}

For a compact Riemannian manifold $(M,g)$ of dimension $d$ whose
geodesic flow is ergodic with respect to the Liouville measure
$\mu_L$, \emph{quantum ergodicity} (QE) of eigenfunctions~\cite{sch,z1,cdv} states
that any orthonormal basis $(e_j)_{j\in \mathbb N}$ of eigenfunctions of the Laplacian
with eigenvalues $z_j^2$, has a density one subsequence $(e_{j_k})$
that converges microlocally to $\mu_L$ in the following
sense: for each symbol $a\in C^\infty(T^*M)$ of order zero,
\begin{equation}
  \label{e:qe}
\langle \Op_{h_{j_k}}(a) e_{j_k},e_{j_k}\rangle_{L^2(M)}\to {1\over\mu_L(S^*M)}\int_{S^*M}a\,d\mu_L.
\end{equation}
Here $S^*M$ stands for the unit cotangent bundle, $\Op_h(a)$ denotes
the pseudodifferential operator obtained by quantizing $a$ (see
Section~\ref{s:prelim.notation}), and we put $h_j=z_j^{-1}$.
The proof uses the following integrated form of quantum ergodicity~\cite{h-m-r}:
\begin{equation}
  \label{e:qe-integrated}
h^{d-1}\sum_{h^{-1}\leq z_j\leq h^{-1}+1}\bigg| \cjg \Op_h(a) e_j,e_j\cjd_{L^2}-
\frac{1}{\mu_L(S^*M)}\int_{S^*M}a\,d\mu_L \bigg| \to 0\quad \textrm{ as }
h\to 0.
\end{equation}
See Appendix~\ref{s:qe} for a short self-contained proof of this result
using the methods of this paper.

In the present paper, we consider a non-compact complete Riemannian
manifold $(M,g)$ and show that generalized eigenfunctions of the
Laplacian on $M$ known in scattering theory as~\emph{distorted plane
waves} or~\emph{Eisenstein functions}, converge microlocally on
average, similarly to~\eqref{e:qe-integrated}, with the limiting
measure $\mu_\xi$ depending on the direction of the plane wave
$\xi$~-- see Theorem~\ref{t:convergence}. We also give estimates on
the rate of convergence in terms of classical quantities defined from
the geodesic flow on $M$~-- see Theorem~\ref{t:remainder}.

Our microlocal convergence of plane waves is similar in spirit to the
QE results~\eqref{e:qe} and~\eqref{e:qe-integrated}.  However, unlike
the case of QE where \emph{ergodicity} of the geodesic flow is
essential, our result is based on a different phenomenon, roughly
described as \emph{dispersion} of plane waves. This difference
manifests itself in the proofs as follows: instead of averaging an
observable along the geodesic flow as in the standard proof of quantum
ergodicity, we propagate it. See Section~\ref{s:outline} for an
outline of the proofs of Theorems~\ref{t:convergence}
and~\ref{t:remainder}.

\smallsection{Geometric assumptions near infinity}
The manifold $M$ has dimension $d=n+1$. For our results to hold, we
need to make several assumptions on the geometry of $(M,g)$ near
infinity and on the spectral decomposition of its Laplacian
$\Delta$. They are listed in Section~\ref{s:general} and we check in
Sections~\ref{s:euclidean} and~\ref{s:ah} that they are satisfied in
each of the following two cases:
%
%
\begin{enumerate}
  \item
there exists a compact set $K_0\subset M$ such that $(M\setminus
K_0,g)$ is isometric to $\mathbb R^{n+1}\setminus B(0,R_0)$ with the
Euclidean metric $g_0$ for some $R_0>0$; here $B(0,R_0)$ denotes the ball
centered at $0$ of radius $R_0$,
  \item
$(M,g)$ is an \emph{asymptotically hyperbolic} manifold in the sense
that it admits a smooth compactification $\overline M$ and there
exists a smooth boundary defining function $x$ such that in a collar
neighbourhood of the boundary $\plM$, the metric has the form
\begin{equation}
  \label{e:k-1}
g={dx^2+h(x)\over x^2}.
\end{equation}
where $h(x)$ is a smooth 1-parameter family of metrics on
$\plM$ for $x\in[0,\eps)$. We further assume that $g$ has
sectional curvature $-1$ in a neighbourhood of $\plM$.
\end{enumerate}
%
%
In case~(1), we call $(M,g)$ \emph{Euclidean near infinity}, while in
case~(2), we call it \emph{hyperbolic near infinity}. Case~(2) in
particular includes convex co-compact hyperbolic quotients
$\Gamma\backslash \mathbb H^{n+1}$~-- see Appendix~\ref{s:k-1}.  Other
possible geometries are discussed in Section~\ref{s:other-geometry}.

\smallsection{Distorted plane waves/Eisenstein functions}
Let $\Delta$ be the (nonnegative) Laplace--Beltrami operator on $M$.
In the study of the relation between classical dynamics and high
energy behavior it is natural to use the semiclassically rescaled
operator $h^2\Delta$, with $h>0$ small parameter tending to zero.

The operator $h^2\Delta$ has continuous
spectrum on a half-line $[c_0h^2,\infty)$ (here $c_0$ is $0$ for the Euclidean
and $n^2/4$ for the hyperbolic case) and possibly a finite number of eigenvalues in
$(0,c_0h^2)$. The continuous spectrum is parametrized by
\emph{distorted plane waves} (or \emph{Eisenstein functions} in the hyperbolic case)
$E_h(\lambda,\xi)\in C^\infty(M)$, satisfying for $\lambda\in \mathbb R$,
\begin{equation}
  \label{e:e-eq-h}
(h^2(\Delta-c_0)-\lambda^2)E_h(\lambda,\xi)=0.
\end{equation}
Because of the $h$-rescaling, the effective spectral parameter is $\la/h$.
Here $\xi$ lies on the boundary $\plM$ of a compactification
$\overline M$ of $M$. We can think of an element of $\plM$ as the
direction of escape to infinity for a non-trapped geodesic; then $\xi$
is the direction of the outgoing part of the plane wave
$E_h(\lambda,\xi)$ at infinity.

For instance, in the case of manifolds Euclidean near infinity, $c_0=0$,
$\plM=\mathbb S^n$ is the sphere, and for $m\in M\setminus K_0\simeq \rr^{n+1}\setminus B(0,R_0)$,
$$
E_h(\lambda,\xi;m)=e^{{i\lambda\over h}\xi\cdot m}+E_{\mathrm{inc}},
$$
where $E_{\mathrm{inc}}$  
is incoming in the sense that there exists $f\in C^\infty(S^n)$ such that 
$[E_{\mathrm{inc}}(\lambda,\xi;m)
-|m|^{-\ndemi}e^{- i\frac{\lambda}{h}|m|}f(\frac{m}{|m|})]|_{M\setminus K_0}\in L^2$, or equivalently $E_{\mathrm{inc}}$ lies in the
image of $C_0^\infty(\mathbb R^{n+1})$ under the free incoming
resolvent $(h^2\Delta-(\la-i0)^2)^{-1}$ of the Laplacian on 
$\mathbb R^{n+1}$). These conditions provide a unique characterization
of $ E_h ( \lambda, \xi ) $. We can also write
$E_h(\lambda,\xi)=E(\lambda/h,\xi)$, where $E(z,\xi)$ is the
nonsemiclassical plane wave, and rewrite the results below in terms of
the parameter $z$, as in the abstract.

We will freely use the notions of semiclassical analysis as found for
example in~\cite{e-z}, and reviewed in Section~\ref{s:prelim}.  We
denote elements of the cotangent bundle $T^*M$ by $(m,\nu)$, where
$m\in M$ and $\nu\in T^*_mM$.  The semiclassical principal symbol of
$h^2\Delta$ is equal to $p(m,\nu)=|\nu|_g^2$, where $|\nu|_g$ is the
length of $\nu\in T^*_mM$ with respect to the metric $g$. Therefore,
the plane wave $E_h$ should be concentrated on the unit cotangent
bundle (see~\cite[Theorem~5.3]{e-z})
$$
S^*M:=\{(m,\nu)\in T^*M\mid |\nu|_g=1\}.
$$
If $g^t:T^*M\to T^*M$ denotes the geodesic flow, then the Hamiltonian
flow of $p$ is $e^{tH_p}=g^{2t}$. 

\smallsection{Semiclassical limits of $E_h$ when the trapped set has measure zero}
In scattering theory trajectories which never escape to infinity play
a special role as they can be observed only indirectly in asymptotics
of plane waves. The \emph{incoming tail} (resp. \emph{outgoing tail})
$\Gamma_-\subset S^*M$ (resp. $\Gamma_+\subset S^*M$) of the flow is
defined as follows: a point $(m,\nu)$ lies in $\Gamma_-$
(resp. $\Gamma_+$) if and only if the geodesic $g^t(m,\nu)$ stays in
some compact set for $t\geq 0$ (resp. $t\leq 0$).  The \emph{trapped
set} $K:=\Gamma_+\cap \Gamma_-$ is the set of points $(m,\nu)$ such
that the geodesic $g^t(m,\nu)$ lies entirely in some compact subset of
$S^*M$.

Our first result states that if $\mu_L(K)=0$, then plane waves
$E_h(\la,\xi)$ converge on average to some measures supported on the
closure of the set of trajectories converging to $\xi$ in $\bbar{M}$:
%
%
\begin{theo}
  \label{t:convergence}
Let $(M,g)$ be a Riemannian manifold satisfying the 
assumptions of Section~\ref{s:general} and suppose that the trapped set
has Liouville measure $\mu_L(K)=0$. For  Lebesgue almost every $\xi\in \plM$, there exists a Radon measure
$\mu_\xi$ on $S^*M$ such that for each compactly supported
$h$-semiclassical pseudodifferential operator $A\in\Psi^0(M)$, we have as $h\to 0$,
\begin{equation}
  \label{e:convergence-1}
h^{-1}\bigg\|\langle AE_h(\lambda,\xi),E_h(\lambda,\xi)\rangle_{L^2(M)}
-\int_{S^*M}\sigma(A)\,d\mu_\xi\bigg\|_{L^1_{\xi,\lambda}(\plM\times[1,1+h])}\to 0.
\end{equation}
where $\sigma(A)$ is the semiclassical principal symbol of $A$ as defined in~\cite[Theorem~14.1]{e-z}.
The measure $\mu_\xi$ has support  
\begin{equation}
  \label{e:supp-mu-xi}
\supp(\mu_{\xi})\subset \bbar{\{(m,\nu)\in S^*M\mid \lim_{t\to +\infty}g^{t}(m,\nu)=\xi\}},
\end{equation}
and disintegrates the Liouville measure in the sense that there exists a smooth measure $d\xi$ on $\plM$ such that,
if $\mu_L$ is the Liouville measure generated by $\sqrt{p}=|\nu|_g$ on $S^*M$, then
\begin{equation}
  \label{e:total-liouville}
\int_{\plM}\mu_\xi\,d\xi=\mu_L.
\end{equation}
\end{theo}
%
%
The limiting measure $\mu_\xi$ is defined in
Section~\ref{s:general.limiting}. Implicit
in~\eqref{e:total-liouville} is the statement that for any bounded
Borel $S\subset S^*M$, we have $\mu_\xi(S)\in L^1_\xi(\plM)$. In Lemma \ref{cccgroup}, we show that for hyperbolic manifolds $\mu_{\xi}$ is well defined for all $\xi\in\plM$ and it is likely that the same is true when the curvature of $g$ is negative near the trapped set, but we believe that this does not hold in the general setting of Theorem~\ref{t:convergence}.

In the case when $\WFh(A)\cap\Gamma_-=\emptyset$ (in particular when
$g$ is non-trapping), we actually have a full expansion of $\langle
AE_h,E_h\rangle$ in powers of $h$, with remainders bounded in
$L^1_{\xi,\lambda}(\plM\times [1,1+h])$~--
see~\eqref{e:full-expansion}. It is likely that for $K=\emptyset$,
this can be strengthened to uniform convergence in $\xi,\lambda$,
using nontrapping estimates on the resolvent.

The now standard argument of Colin de Verdi\`ere and Zelditch (see for
example the proof of~\cite[Theorem~15.5]{e-z}) shows that there exists
a family of Borel sets $\mathcal A(h)\subset \plM\times[1,1+h]$ such
that the ratio of the measure of $\mathcal A(h)$ to the measure of the
whole $\plM\times[1,1+h]$ goes to 1 as $h\to 0$, and for each
$A\in\Psi^0(M)$ as in Theorem~\ref{t:convergence} with $\sigma(A)$
independent of $h$,
\begin{equation}
  \label{e:sh-convergence}
\langle AE_h(\lambda,\xi),E_h(\lambda,\xi)\rangle_{L^2(M)}\to\int_{S^*M}
\sigma(A)\,d\mu_\xi\text{ uniformly in }(\lambda,\xi)\in \mathcal A(h).
\end{equation}
This statement can be viewed as an analogue of the quantum ergodicity fact~\eqref{e:qe},
though as explained above, it is produced by a different phenomenon.

\smallsection{Estimates for the remainder}
We next provide a quantitative version of Theorem~\ref{t:convergence},
namely an estimate of the left-hand side of~\eqref{e:convergence-1}.
We define the set $\mathcal T(t)$ of geodesics trapped for time $t>0$ as follows: 
let $K_0$ be a compact geodesically convex subset of $M$ 
(in the sense of~\eqref{geodconvex}) containing a neighborhood of the
trapped set $K$, then (see also Section~\ref{s:proofs-2})
\begin{equation}
  \label{e:T-t}
\mathcal T(t):=\{(m,\nu)\in S^*M\mid m\in K_0,\ \pi( g^t(m,\nu))\in K_0\}.
\end{equation}
where $\pi: T^*M\to M$ is the projection map. 
A quantity which will appear frequently with
some parameter $\Lambda>0$ is the following interpolated measure
\begin{equation}
  \label{defofrh}
r(h,\Lambda):= \sup_{0\leq\theta\leq 1}h^{1-\theta}\mu_L\big(\mathcal T\big(\theta\Lambda^{-1}|\log h|\big)\big),
\end{equation}
where $h>0$ is small. This converges to $0$ as $h\to 0$ when
$\mu_{L}(K)=0$ and it interpolates between $h$ (when $\theta=0$) and
the Liouville measure of the set of geodesics that remain trapped for
time $\Lambda^{-1}|\log h|$ (when $\theta=1$). When the measure
$\mu_L(\mathcal T(t))$ decays exponentially in $t$, as
in~\eqref{e:pressure-estimate}, $r(h,\Lambda)$ can be replaced by
simply $\mathcal O(h)+\mu_L(\mathcal T(\Lambda^{-1}|\log h|))$. 
The~$\mathcal O(h)$ term here is natural since one can
add an operator in $h\Psi^0(M)$ to $A$, which will change
$\langle AE_h,E_h\rangle$ by $\mathcal O(h)$, but will not change
$\sigma(A)$ (which is only defined invariantly modulo $\mathcal O(h)$).

We next define the \emph{maximal expansion rate} as follows (see also~\eqref{e:lambda-max}):
\begin{equation}
  \label{lamax}
\Lambda_{\max}:=\limsup_{|t|\to +\infty}{1\over |t|}\log \sup_{(m,\nu)\in \mathcal T(t)}
\|dg^t(m,\nu)\|.
\end{equation}
We can estimate the left-hand side of~\eqref{e:convergence-1} in terms
of the (interpolated) measure of the set of all trajectories trapped
for the Ehrenfest time, see~\eqref{e:ehrenfest-time} for the definition of
the latter. If we pair with a test function in $\xi$
instead of taking the $L^1_\xi$ norm, then the estimate becomes stronger,
corresponding to the set of all trajectories trapped for \emph{twice}
the Ehrenfest time:
%
%
\begin{theo}
  \label{t:remainder}
Let $(M,g)$ be as in Theorem~\ref{t:convergence}.
Take $\Lambda_0>\Lambda_{\max}$. Then for each compactly
supported $h$-semiclassical pseudodifferential operator $A\in \Psi^0(M)$ 
and for each $f\in
C^\infty(\plM)$, 
\begin{align}
  \label{e:convergence-2}
h^{-1}\bigg\|\langle AE_h,E_h\rangle
-\int_{S^*M} \sigma(A)\,d\mu_\xi\bigg\|_{L^1_{\xi,\lambda}(\plM\times
[1,1+h])}
&=\mc{O}(r(h,2\Lambda_0)),\\
  \label{e:convergence-3}
h^{-1}\bigg\|\int_{\plM}f(\xi)\bigg(\langle A E_h,E_h\rangle
-\int_{S^*M}\sigma(A)\,d\mu_\xi\bigg)\,d\xi\bigg\|_{L^1_\lambda([1,1+h])}
&= \mc{O}(r(h,\Lambda_0)).
\end{align}
\end{theo}
%
%
The proof of Theorem~\ref{t:remainder} actually gives an expansion
of $\langle AE_h,E_h\rangle$ in powers of $h$, with remainder
$\mu_L(\mathcal T(\Lambda^{-1}|\log h|))$ instead of $r(h,\Lambda)$~--
see~\eqref{e:analysis2-main} and the proofs of
Propositions~\ref{l:analysis2-4} and~\ref{l:trace-2}. This
full expansion is cumbersome to write down, therefore we only do it
for the trace estimates~\eqref{TraPI} below.

\smallsection{Remainder in terms of pressure} 
When the trapped set $K$ has Liouville measure $0$ and is uniformly partially hyperbolic in the
sense of Appendix~\ref{ap:escape}, we estimate using \cite{Yo}
\begin{equation}
  \label{e:pressure-estimate}
\mu_L(\mathcal T(t))=\mc{O}(e^{t(P(J^u)+\varepsilon)}),
\end{equation}
for each $\varepsilon>0$,
where $P(J^u)\leq 0$ is the topological pressure of the unstable
Jacobian~-- see Appendix~\ref{s:escaperate-1}. When $K$ is a hyperbolic
basic set (Axiom~A flow), then $P(J^u)<0$ by \cite{BoRu}, and the
remainders in~\eqref{e:convergence-2} and~\eqref{e:convergence-3} are
then polynomial in $h$:
$$
r(h,\Lambda)=\mc{O}(h+h^{-(P(J^u)+\varepsilon)/\Lambda}),
$$
and one can get rid of $\varepsilon$ here by slightly changing
$\Lambda_0$.  In the special case where $g$ has constant sectional
curvature $-1$ near $K$, the bounds in~\eqref{e:convergence-2} and
\eqref{e:convergence-3} become respectively,  $\mathcal O(h+h^{(n-\delta)/2-})$ and
$\mathcal O(h+h^{n-\delta-})$, where $K$ has Hausdorff
dimension $\dim_{H}(K)=2\delta+1$.%
\footnote{The notation $\mathcal O(h^{-\alpha-})$ means $\mathcal O(h^{-\alpha-\varepsilon})$ for any fixed $\varepsilon>0$.}
See Appendix~\ref{s:escaperate-2} for details.

In all cases, if $K$ is nonempty, then it has Minkowski dimension
at least 1; since $g^{t/2}(\mathcal T(t))$ contains an $e^{-\Lambda_0
t/2}$ sized neighborhood of $K$, there exists $c>0$ such that for all small $h>0$
\begin{equation}
  \label{e:lower}
\mu_L\big(\mathcal T((2\Lambda_0)^{-1}|\log h|)\big)\geq c h^{n/2},\
\mu_L\big(\mathcal T(\Lambda_0^{-1}|\log h|)\big)\geq c h^n.
\end{equation}

\smallsection{Local Weyl asymptotics for spectral projectors}
It is possible to express the spectral measure of $h^2\Delta$ in terms
of the distorted plane waves (see ~\eqref{e:spectrum-eis-2}), and
using~\eqref{e:convergence-3}, we obtain an expansion in powers of $h$
for local traces of spectral projectors up to an explicit
remainder. We only write it here for the case where the flow is
partially uniformly hyperbolic with $P(J^u)<0$, but a more general
result with the Liouville measure of $\mathcal T(\Lambda_0^{-1}|\log
h|)$ holds~-- see Theorem~\ref{asympofs_A} in
Section~\ref{s:proofs-3}. In the theorem below, we use a quantization procedure
$\Op_h$ on $M$ mapping compactly supported symbols to compactly
supported operators; alternatively, one can consider operators of the form
$\chi \Op_h(a)\chi$, where $\Op_h$ is any quantization procedure,
$a(x,\xi)$ is compactly supported in the $x$ variable, and $\chi\in C_0^\infty(M)$
is equal to 1 near $\pi(\supp a)$.
%
%
\begin{theo}
\label{cor:trace_est}
Let $(M,g)$ be as in Theorem~\ref{t:convergence}, let
$\Lambda_0>\Lambda_{\max}$ and assume that the trapped set $K$ is
uniformly partially hyperbolic with $\mu_L(K)=0$ and that the
topological pressure $P(J^u)$ of the unstable Jacobian on $K$ is negative.   Then there
exist differential operators%
\footnote{In this paper, the symbols $L_j$ will denote different
operators in different propositions.}
$L_j$ of order $2j$ on $T^*M$, with
$L_0=1$, such that for each compactly supported zeroth order classical symbol $a$, we have for each $s>0$ and $N\in \nn$
\begin{equation}
  \label{TraPI} 
\Tr({\rm Op}_h(a)\indic_{[0,s]}(h^2\Delta))=(2\pi h)^{-n-1}\sum_{j=0}^{N}
h^j \int\limits_{|\nu|_g^2\leq s}L_ja\, d\mu_\omega
+h^{-n}\mc{O}\big(h^{-\frac{P(J^u)}{\Lambda_0}}+h^{N}\big)
\end{equation}
where $\mu_\omega$ is the standard volume form
on $T^*M$ and $\indic_{[0,s]}(h^2\Delta)$ denotes the spectral projector of
$h^2\Delta$ onto the frequency window $[0,s]$. The remainder is
uniform in $s$ when $s$ varies in a compact subset of $(0,\infty)$.
\end{theo}
%
%
In particular, if $g$ has constant sectional curvature $-1$ near $K$
and the Hausdorff dimension of $K$ is given by $2\delta+1$,
then the remainder in~\eqref{TraPI} becomes $\mathcal O(h^{-\delta-})$,
for $N$ large enough. 

\smallsection{Applications}
In a separate paper~\cite{DyGu}, we show that
Theorem~\ref{cor:trace_est} implies new asymptotics for the spectral
shift function (or scattering phase) with remainders in terms of
$P(J^u)$ when the trapped set has Liouville measure $0$ and the
manifold is Euclidean near infinity with uniformly partially
hyperbolic geodesic flow near $K$.

\smallsection{Previous works}
Let us briefly discuss the history of Quantum Ergodicity (QE) and
explain its relation to the present paper. The original QE statement
was proved by Shnirelman~\cite{sch}, Zelditch~\cite{z1}, and Colin de
Verdi\`ere~\cite{cdv} in the microlocal case, by
Helffer--Martinez--Robert~\cite{h-m-r} in the semiclassical case (with
the integrated estimate using an $\mathcal O(h)$ spectral window like
in the present paper, rather than the $\mathcal O(1)$ window used in
the microlocal case), and by G\'erard--Leichtnam~\cite{g-l} and
Zelditch--Zworski~\cite{z-z} for manifolds with boundary (ergodic
billiards). Quantum ergodicity for boundary values and restrictions of
eigenfunctions to hypersurfaces was studied by
Hassell--Zelditch~\cite{HaZe}, Burq~\cite{Burq},
Toth--Zelditch~\cite{t-z10,t-z11}, and by
Dyatlov--Zworski~\cite{DyZw}.

The first result on noncompact manifolds, namely for embedded
eigenvalues and Eisenstein functions on surfaces with cusps, was
proved by Zelditch~\cite{z2}. For the special case of arithmetic
hyperbolic surfaces, a stronger statement of Quantum Unique Ergodicity
(QUE), saying that the whole sequence of eigenstates microlocally
converges to the Liouville measure, was proved by
Lindenstrauss~\cite{Lin} and Soundararajan~\cite{sound} for
Hecke--Maass forms and by Luo--Sarnak~\cite{LuoSarnak} and
Jakobson~\cite{Jakobson} for Eisenstein functions. For further
information on the topic, the reader is directed to the recent
reviews~\cite{n,sar,z3}.

As remarked above, our result differs from the above works in that it
uses dispersion of plane waves instead of the ergodicity of the
geodesic flow. This dispersion phenomenon was used to study microlocal
limits of plane waves on convex co-compact hyperbolic quotients
satisfying $\delta<n/2$ by Guillarmou--Naud in~\cite{g-n}, and on
surfaces with cusps at complex energies by Dyatlov~\cite{eiscusp}.
Both~\cite{g-n} and~\cite{eiscusp} guarantee microlocal convergence of
the Eisenstein functions that is uniform in $\lambda$ and $\xi$,
rather than the (weaker) $L^1_{\lambda,\xi}$ estimates of the current
paper; these statements are formally similar to QUE, while our
statement is formally similar to QE. In~\cite{g-n}, uniform in
$\lambda$ and $\xi$ estimates are possible because Lagrangian states,
when propagated by the Schr\"odinger group $U(t)=e^{ith\Delta/2}$, would disperse
faster than they fail to be approximated semiclassically, a phenomenon
similar to the one studied by Nonnenmacher--Zworski~\cite{n-z}. In
fact, it is plausible that the result of~\cite{g-n} is true when the
condition $\delta<n/2$ is replaced by the negative pressure condition
of~\cite{n-z}. As for~\cite{eiscusp}, the energy being away from the
real line makes the measure corresponding to $E_h$ exponentially
increasing, rather than invariant, along the flow; 
in a way, this paper relies on damping of plane waves
rather than dispersion.

We see that the uniform convergence in~\cite{g-n} and~\cite{eiscusp}
is possible because one has better control on the propagated
Lagrangian states. Such better control is directly related to having a
polynomial bound on the scattering resolvent. In the less restricted
situation of our paper, however, it is not clear if such a bound would
hold; therefore, we need to average in $\lambda$ and $\xi$ to pass to
trace (or, strictly speaking, Hilbert--Schmidt norm) estimates, just
as in the proof of Quantum Ergodicity.

The expansions for local traces of the spectral measure as in
Theorem~\ref{cor:trace_est} were studied by Robert--Tamura~\cite{RoTa}
for nontrapping perturbations of the Euclidean space, yielding a full
expansion in powers of $h$ in that setting.

\section{Outline of the proofs}
  \label{s:outline}

In this section, we explain the ideas of the proofs of
Theorems~\ref{t:convergence} and~\ref{t:remainder}, in the case of
manifolds Euclidean near infinity. We also describe the structure of
the paper.

We start with Theorem~\ref{t:convergence}. Take $t>0$; we will use
$\lim_{t\to+\infty}\lim_{h\to 0}$ limits, therefore remainders that
decay in $h$ with constants depending on $t$ will be negligible.
Since $E_h$ is a generalized eigenfunction of the
Laplacian~\eqref{e:e-eq-h}, we have for $\la>0$
\begin{equation}
  \label{e:e-propagation}
E_h(\lambda,\xi)=e^{-it\lambda^2/(2h)}U(t)E_h(\lambda,\xi).
\end{equation}
Here $U(t)=e^{ith\Delta/2}$ is the semiclassical Schr\"odinger
propagator, quantizing the geodesic flow $g^t$. Since $E_h$ does not
lie in $L^2(M)$, we cannot apply the operator $U(t)$ to it; however,
\eqref{e:e-propagation} can be made rigorous, with an $\mathcal
O(h^\infty)$ error, by using appropriate cutoffs~-- see
Lemma~\ref{l:key}. We will not write these cutoffs here for the
sake of brevity.

Take a compactly supported and compactly microlocalized
semiclassical pseudodifferential operator $A$ on $M$; then
by~\eqref{e:e-propagation},
\begin{equation}
  \label{e:a-dec}
\langle AE_h,E_h\rangle=\langle AU(t)E_h,U(t)E_h\rangle=
\langle A^{-t} E_h,E_h\rangle,
\end{equation}
where $A^{-t}:=U(-t)AU(t)$ is a pseudodifferential operator with
principal symbol $\sigma(A)\circ g^{-t}$. (It is not compactly
supported, but we ignore this issue here.) We now use the following
decomposition of plane waves (see~\eqref{paramEheucl}): for a fixed $\la>0$,
\begin{equation}
  \label{decomp_of_E_h}
E_h=\chi_0 E_h^0+E_h^1,\quad 
E_h^0(m)=e^{{i\lambda\over h}\xi\cdot m},\
E_h^1=-R_h(\lambda)F_h,\
F_h:=(h^2\Delta-\lambda^2)\chi_0E_h^0.
\end{equation}
Here $E_h^0$ is the outgoing part of the plane wave, defined in a certain neighborhood
of infinity and solving~\eqref{e:e-eq-h} there, while $\chi_0$ is a cutoff
function equal to $1$ near infinity and supported inside the domain of $E_h^0$;
then
$$
F_h=[h^2\Delta,\chi_0]E_h^0
$$
is compactly supported and we can apply to it the semiclassical
incoming resolvent $R_h(\lambda)$ defined by $R_h(\la):=\lim_{\eps\to 0+}(h^2\Delta-(\la-i\eps)^2)^{-1}$ 
when acting on compactly supported functions, where $(h^2\Delta-z)^{-1}$ is bounded on $L^2$ for
$z\notin [0,\infty)$~-- see~\eqref{e:r-h-eucl} for the definition
in the Euclidean case and~\eqref{e:r-h-hyp} for a similar description in the hyperbolic case.
For $\lambda=1+\mathcal O(h)$, the function $F_h$ is microlocalized inside
the set
$$
W_\xi:=\{(m,\nu)\mid m\in\supp(d\chi_0),\ \nu=\xi\}\subset S^*M.
$$
In general, we cannot expect the resolvent $R_h(\lambda)$ to be
polynomially bounded in $h$, and thus cannot determine the wavefront
set of $E_h^1$. However, we will show the following weaker propagation
of singularities statement: the function
$$
\widetilde E_h^1(\lambda,\xi):={E_h^1(\lambda,\xi)\over 1+\|E_h(\lambda,\xi)\|_{L^2(K_0)}},
$$
where $K_0\subset M$ is a sufficiently large compact set, is
polynomially bounded in $h$ and for each $(m,\nu)\in \WFh(\widetilde
E_h^1)$, the geodesic $g^t(m,\nu)$ is either trapped as $t\to +\infty$
or passes through $W_\xi$ for some $t\geq 0$. For the case of
manifolds Euclidean near infinity, this follows directly
from the explicit formula for the scattering resolvent on the free
Euclidean space; for manifolds hyperbolic near infinity, we use the
microlocal properties of the resolvent established in~\cite{v}. See
assumption~\as{A6} in Section~\ref{s:general.analytic},
Section~\ref{s:euclidean.analysis}, and Proposition~\ref{WFE} for
details.

If $A$ and $1-\chi_0$ are both supported in the ball of radius $R$,
let $\varphi\in C_0^\infty(M)$ be independent of $t$ and equal to 1 in the
ball of radius $R+1$. Then we write
$$
A^{-t}=A_0^{-t}+A_1^{-t},\
A_0^{-t}:=A^{-t}\varphi,\
A_1^{-t}:=A^{-t}(1-\varphi).
$$
Now, each $(m,\nu)\in \WFh(A_1^{-t})$ has the following properties:
$|m|\geq R+1$, and for $(m',\nu')=g^{-t}(m,\nu)$, $|m'|\leq R$.
(See Figure~\ref{f:general-explanation}.)
Therefore, the geodesic $g^s(m,\nu)$ escapes to infinity
for $s\geq 0$ and never passes through $W_\xi$; it follows
from the discussion of the wavefront set of $\widetilde E_h^1$
in the previous paragraph that
$$
\|A^{-t}_1E^1_h\|_{L^2}=\mathcal O(h^\infty(1+\|E_h\|_{L^2(K_0)})).
$$ 
Therefore, we can write
\begin{equation}
  \label{e:a-dec2}
\langle AE_h,E_h\rangle=\langle A_1^{-t}\chi_0E_h^0,\chi_0E_h^0\rangle
+\langle A^{-t}_0E_h,E_h\rangle+\mathcal O(h^\infty(1+\|E_h\|_{L^2(K_0)}^2)).
\end{equation}
The first term on the right-hand side is explicit, as we have a formula
for $E_h^0$; we can calculate for Lebesgue almost every $\xi$ and $\lambda=1+\mathcal O(h)$,
\begin{equation}
  \label{e:e-0-limit}
\lim_{t\to +\infty}\lim_{h\to 0}\langle A_1^{-t}\chi_0 E_h^0(\lambda,\xi),\chi_0 E_h^0(\lambda,\xi)\rangle
=\int_{S^*M} a\,d\mu_\xi.
\end{equation}
%
%
\begin{figure}
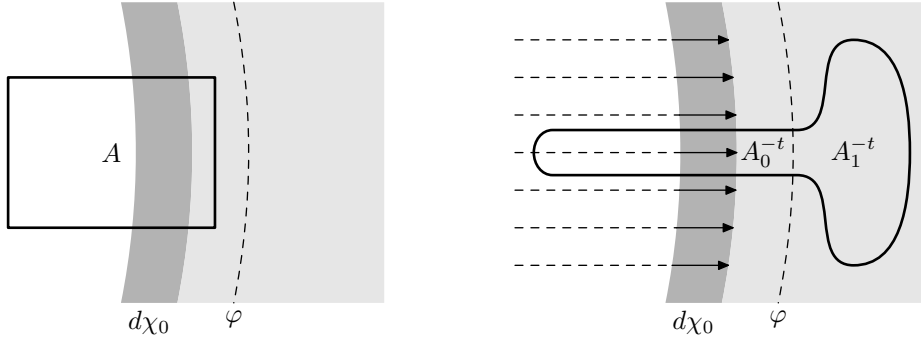

\includegraphics{qeefun.7}
\qquad\qquad
\includegraphics{qeefun.8}
\caption{A phase space picture of the main argument.
The right side of each picture represents infinity;
$\chi_0=1$ in the lighter shaded region and $d\chi_0$
is supported in the darker shaded region, while $\varphi=1$
to the left of the vertical dashed line. The horizontal dashed
lines on the right represent the wavefront set of $\widetilde E^1_h$;
they terminate at the solid arrows, which denote the set $W_\xi$.}
\label{f:general-explanation}
\end{figure}
%
%
It then remains to estimate the second and third terms on average
in $\lambda$ and $\xi$. For this, we use the relation~\eqref{e:spectrum-eis-2}
of distorted plane waves to the spectral measure of the Laplacian
to get for any bounded compactly supported pseudodifferential operator $B$,
\begin{equation}
  \label{e:tr-bd}
h^{-1}\|BE_h(\lambda,\xi)\|^2_{L^2_{m,\xi,\lambda}(M\times \plM\times [1,1+h])}
\leq Ch^n\|B\indic_{[1,(1+h)^2]}(h^2\Delta)\|^2_{\HS}.
\end{equation}
Here $\HS$ denotes the Hilbert--Schmidt norm.  One can estimate the
right-hand side of~\eqref{e:tr-bd} uniformly in $h$~-- see
Lemma~\ref{l:h-s-estimate} and the proof of
Proposition~\ref{l:avg-bdd}. Then $h^{-1}\|E_h\|^2_{L^2(K_0)}$, when
integrated over $\lambda\in [1,1+h]$ and $\xi$, is bounded uniformly
in $h$; this removes the third term on the right-hand side
of~\eqref{e:a-dec2}.

Finally, the average in $\lambda,\xi$ of the second term on the
right-hand side of~\eqref{e:a-dec2} can be bounded, modulo an
$\mathcal O_t(h)$ remainder, by the $L^2$ norm
$\|\sigma(A^{-t}_0)\|_{L^2(S^*M)}$ of the restriction of the principal
symbol of $A^{-t}_0$ to the energy surface $S^*M$, with respect to the
Liouville measure. Now, $\sigma(A^{-t}_0)=(\sigma(A)\circ
g^{-t})\varphi$ converges to zero as $t\to +\infty$ at any point which
is not trapped in the backwards direction.  Since we assumed
$\mu_L(K)=0$, by the dominated convergence theorem
$\|\sigma(A^{-t}_0)\|_{L^2(S^*M)}$ converges to zero as $t\to
+\infty$; this finishes the proof of Theorem~\ref{t:convergence}.

For the estimate~\eqref{e:convergence-2} in Theorem~\ref{t:remainder},
we need to take $t$ up to the Ehrenfest time:
\begin{equation}
  \label{e:ehrenfest-time}
t=t_e:=\Lambda_0^{-1}\log(1/h)/2,
\end{equation}
 replacing the $\lim_{t\to
+\infty} \lim_{h\to 0}$ limit in the argument of
Theorem~\ref{t:convergence} by just the $\lim_{h\to 0}$ limit, but
with $t$ depending on $h$.  The operator $A^{-t}$ is then still
pseudodifferential, though in a mildly exotic class. To avoid a
quantization procedure uniform at infinity, we give an iterative
argument, propagating $A$ for a fixed time for $\sim\log(1/h)$ steps,
applying $t$-independent cutoffs and removing the microlocally
negligible terms at each step. The proof then works as before, with
the term~$\langle A^{-t}_0 E_h,E_h\rangle$ bounded by the Liouville
measure of the microsupport of $A^{-t}_0$ (see Definition~\ref{d:microlocal-vanishing}), which depends
on $h$ and is contained in $g^t(\mathcal T(t))$, where $\mathcal T(t)$
is defined in~\eqref{e:T-t}; this proves~\eqref{e:convergence-2}. The
interpolated quantity $r(h,\Lambda)$ from~\eqref{defofrh} appears
because of the subprincipal terms in~\eqref{e:e-0-limit}.

For~\eqref{e:convergence-3}, we have to propagate to twice the
Ehrenfest time: $t=2t_e$.  The operator $A^{-t}$ is not
pseudodifferential, but we can use~\eqref{e:e-propagation} to write
\begin{equation}
  \label{e:a-t-0}
\langle A^{-t}_0 E_h,E_h\rangle=\langle U(-t/2)AU(t/2)\cdot U(t/2)\varphi U(-t/2)E_h,E_h\rangle.
\end{equation}
The operators $U(-t/2)AU(t/2)$ and $U(t/2)\varphi U(-t/2)$ are both
pseudodifferential in a mildly exotic class; multiplying them, we get
a pseudodifferential operator whose full symbol is supported
inside~$g^{t/2}(\mathcal T(t))$, and thus~\eqref{e:a-t-0} can be
estimated by the Liouville measure of this set, giving the
remainder~\eqref{e:convergence-3}.

A problem arises when trying to get a rate of convergence
in~\eqref{e:e-0-limit} for $t$ up to twice the Ehrenfest time. We are
unable to propagate the Lagrangian state $E_h^0(\lambda,\xi)$
pointwise in $\xi$ and $\lambda$ for time $t$, therefore we do not get
an $L^1_\xi$ estimate in~\eqref{e:convergence-3}.  However, for $f\in
C^\infty(\plM)$ we can still approximate the integral
\begin{equation}
  \label{e:e-0-int}
\int_{\plM} f(\xi)\langle A^{-t}_1\chi_0 E_h^0, \chi_0 E_h^0\rangle\,d\xi
\end{equation}
as follows. Define the operator
$$
\Pi^0_f(\lambda):=\int_{\plM} f(\xi) (\chi_0 E_h^0(\lambda,\xi))\otimes (\chi_0 E_h^0(\lambda,\xi))\,d\xi.
$$
Here $\otimes$ denotes the Hilbert tensor product; that is, if $u,v\in C^\infty(M)$, then
$u\otimes v$ is the operator with the Schwartz kernel
\begin{equation}
  \label{e:tensor-product}
K_{u\otimes v}(m,m')=u(m)\overline{v(m')}.
\end{equation}
We can show that if $\widetilde X$ is a pseudodifferential operator with compactly supported Schwartz kernel and
microlocalized in a compact subset of $T^*M$, satisfying certain conditions, then
$\widetilde X \Pi^0_f\widetilde X^*$ is a Fourier integral operator associated
to the canonical relation
$$
\{(m,\nu;m',\nu')\mid (m,\nu)\in S^*M,\
(m',\nu')=g^s(m,\nu)\text{ for some }s\in (-T_0,T_0)\},
$$
for a fixed $T_0>0$ depending on $\widetilde X$. (For comparison, for the spectral measure of
$h^2\Delta$ we would have to formally take all possible values of $s$,
which would destroy any hope on microlocally approximating it when the
geodesic flow is chaotic. The use of $E_h^0$ instead of $E_h$ here puts us in a `nontrapping' situation
$M=\mathbb R^n$,
where the cutoff $\widetilde X$ restricts the range of times $s$ we have to consider.) We can then write
$$
\widetilde X \Pi^0_f(\lambda)\widetilde X^*=(2\pi h)^n\int_{-T_0}^{T_0} e^{-i\lambda^2 s/(2h)}U(s)B_s\,ds,
$$
where $B_s$ is a smooth family of pseudodifferential operators, compactly
supported in $s\in (-T_0,T_0)$~-- see Lemma~\ref{l:trace-1}. We then write the integral~\eqref{e:e-0-int} as
$$
\begin{gathered}
\Tr(U(-t)AU(t)(1-\varphi) \Pi^0_f(\lambda))=\Tr\int_{-T_0}^{T_0} e^{-i\lambda^2 s/(2h)} U(-t)AU(t)(1-\varphi)U(s)B_s\,ds\\
=\Tr\int_{-T_0}^{T_0}e^{-i\lambda^2 s/(2h)} U(-t/2)AU(t/2)\cdot U(t/2)(1-\varphi)U(s)B_sU(-s-t/2)\cdot U(s)\,ds.
\end{gathered}
$$
The operators $U(-t/2)AU(t/2)$ and $U(t/2)(1-\varphi)U(s)B_sU(-s-t/2)$ are
pseudodifferential in a mildly exotic class; thus their product is
also pseudodifferential and (bearing in mind that $s$ varies in a
bounded set) one gets a microlocal expansion for~\eqref{e:e-0-int} through
a local trace formula for Schr\"odinger propagators~-- see Lemma~\ref{l:robert-trace}
and Proposition~\ref{l:trace-2}.

\subsection{Other possible geometric assumptions}
\label{s:other-geometry}

Our results should be true for asymptotically hyperbolic manifolds
without the constant curvature assumption near infinity. The main
difficulty here is constructing a good semiclassical parametrix for
the Eisenstein function $E_h(\la,\xi)$ near $\xi\in \plM$; this can be
done by WKB approximation, and the phase is a Busemann function
$\phi_\xi(m)$ near $\xi$, however one would need a good understanding
of the regularity of $\phi_\xi(m)$ as $m\to \xi$. This is in a way
related to the high-frequency parametrix of \cite{MeSBVa} in the
non-trapping setting.  For asymptotically Euclidean or
asymptotically conic ends, this might be more complicated as we would
need a parametrix of $E_h(\la,\xi)$ in a large neighbourhood of
$\xi\in \plM$, essentially in a region with closure containing a ball
of radius $\pi/2$ in $\plM$. In particular, the Lagrangian supporting
the semiclassical parametrix of $E_h(\la,\xi)$ would not a priori be
projectable far from $\xi$, which would make the construction more
technical. We leave these questions for future research.

The convergence result in Theorem~\ref{t:convergence} should be true
in the case where $M$ has a boundary, for instance $M=\rr^{n+1}\setminus
\Omega$ with $\Omega$ a piecewise smooth obstacle. In fact, it should
be straightforward to check that the method of proof applies when
combined with the idea of \cite{z-z}, based on the fact that the
region in phase space near the boundary where the dynamics is
complicated is of Liouville measure $0$ (since we assume
$\mu_L(K)=0$). To get a good remainder in that setting would be more
involved since one would need to care about the amount of mass of
plane waves staying in the regions near the boundary where the
dynamics is complicated, as we propagate up to Ehrenfest time. A
reasonable case to start with is that of strictly convex obstacles.

\subsection{Structure of the paper}

In Section~\ref{s:prelim}, we review certain notions of semiclassical
analysis and derive several technical lemmata; in particular, in
Section~\ref{s:prelim.lagrangian}, we review the local theory of
semiclassical Lagrangian distributions and Fourier integral operators
and in Section~\ref{s:prelim.propagator} we study microlocal
properties of Schr\"odinger propagators, including the
Hilbert--Schmidt norm bound (Lemma~\ref{l:h-s-estimate}).  In
Section~\ref{s:general}, we formulate the general assumptions on the
studied manifolds and derive some immediate corollaries;
Section~\ref{s:general.geometry} contains the geometric assumptions
and the definition of the trapped set and
Section~\ref{s:general.analytic} contains the analytic assumptions on
distorted plane waves. In Section~\ref{s:general.limiting} we
construct the limiting measures $\mu_\xi$ and in
Section~\ref{s:general.averaged} we prove averaged estimates on
Eisenstein functions.

In Section~\ref{s:proofs}, we give the proofs of our main theorems.
Section~\ref{s:proofs-1} contains the proof of
Theorem~\ref{t:convergence}, Section~\ref{s:proofs-2} contains the
proof of the estimate~\eqref{e:convergence-2} in
Theorem~\ref{t:remainder}, while Section~\ref{s:proofs-3} contains the
proof of the estimate~\eqref{e:convergence-3} in
Theorem~\ref{t:remainder}. Section~\ref{s:proofs-3} also contains the
Tauberian argument proving an expansion of the local trace of a
spectral projector (Theorem~\ref{asympofs_A}).
Sections~\ref{s:euclidean} and~\ref{s:ah} study the Euclidean and
hyperbolic near infinity manifolds, respectively, and show that the
general assumptions of Section~\ref{s:general} are satisfied in these
cases.

Appendix~\ref{s:k-1} provides a formula for the limiting measures in
the case of a convex co-compact hyperbolic quotient, which generalizes
the limiting measure of~\cite{g-n} to the case $\delta\geq n/2$.
Appendix~\ref{s:escaperate} discusses the classical escape rate, in
particular explaining~\eqref{e:pressure-estimate}.
Appendix~\ref{s:ehrenfest} gives a self-contained proof of Egorov's
theorem up to the Ehrenfest time (Proposition~\ref{l:ehrenfest}).
Finally, Appendix~\ref{s:qe} contains a short proof of (a special case
of) quantum ergodicity in the semiclassical setting, which is simpler
than that of~\cite{h-m-r} because it does not rely on~\cite{d-g,PeRo}.

\section{Semiclassical preliminaries}
  \label{s:prelim}

In this section, we review the methods of semiclassical analysis
needed for our argument. Most of the constructions listed below are
standard: pseudodifferential operators, wavefront sets, local theory
of Fourier integral operators, and Egorov's theorem. However,
Section~\ref{s:prelim.propagator} contains the propagation result
for generalized eigenfunctions (Lemma~\ref{l:key})
and a Hilbert--Schmidt norm estimate in an $\mathcal O(h)$ spectral
window (Lemma~\ref{l:h-s-estimate}), which the authors were
unable to find in previous literature.


\subsection{Notation}
  \label{s:prelim.notation}

In this subsection, we briefly review certain notation used in
semiclassical analysis. The reader is referred to~\cite{e-z}
(especially Chapter~14 on semiclassical calculus on manifolds)
or~\cite{d-s} for a detailed introduction to the subject.

\smallsection{The phase space}
Let $M$ be a $d$-dimensional manifold without boundary. We denote
elements of the cotangent bundle
$T^*M$ by $(m,\nu)$, where $\nu\in T^*_mM$.
Following~\cite[Section~2]{v}, consider the fiber-radial
compactification $\overline T^*M$ of $T^*M$,
with the boundary definining function given by $\langle\nu\rangle^{-1}$
for any smooth inner product on the fibers of $T^*M$.
 The boundary
$\partial \overline T^*M$, called the \emph{fiber infinity}, is diffeomorphic to the cosphere bundle $S^*M$
over $M$.%
\footnote{Unlike~\cite{v}, we do not use the notation $S^*M$ for fiber
infinity~--- we reserve it for the unit cotangent bundle $\{|\nu|_g=1\}\subset
T^*M$.}

Except in Propositions~\ref{l:elliptic} and~\ref{WFE}, we will use compactly microlocalized operators, for which the
fiber-radial compactification is not necessary. 

\smallsection{Symbol classes}
For $k\in \mathbb R$ and $\rho\in [0,1/2)$, consider the
symbol class $S^k_\rho(M)$ defined as follows: a smooth function
$a(m,\nu;h)$ on $T^*M\times [0,h_0)$ lies in $S^k_\rho(M)$ if and only
if for each compact set $K\subset M$ and each multiindices $\alpha,\beta$,
there exists a constant $C_{\alpha\beta K}$ such that
\begin{equation}\label{e:sym-rho}
\sup_{m\in K,\ \nu\in T^*_m M}|\partial^\alpha_m \partial^\beta_\nu a(m,\nu;h)|
\leq C_{\alpha\beta K} h^{-\rho(|\alpha|+|\beta|)}\langle\nu\rangle^{k-|\beta|}.
\end{equation}
These classes are independent of the choice of coordinates on $M$.
Note that we do not fix the
behaviour of the symbols as $m\to\infty$. The important special case
is $\rho=0$, which includes the classical symbols studied
in~\cite{v}. The class $S^k_0(M)$, denoted simply by $S^k(M)$, would
be sufficient for the convergence Theorem~\ref{t:convergence}. The classes
$S^k_\rho$ with $\rho>0$ will be important for obtaining the remainder
estimate of Theorem~\ref{t:remainder}; these classes arise when propagating symbols in
$S^k_0$ for short logarithmic times, as in Proposition~\ref{l:ehrenfest}.

Since plane waves are microlocalized on the cosphere bundle, away from
the fiber infinity, we will most often work with the classes
$\Sc_\rho$, consisting of compactly supported functions
satisfying~\eqref{e:sym-rho}; we have $S^{\comp}_\rho\subset S^k_\rho$
for all $k$.

\smallsection{Pseudodifferential operators}
Following~\cite[Section~14.2]{e-z}, we can define the algebra
$\Psi^k_\rho(M)$ of pseudodifferential operators with symbols
in~$S^k_\rho(M)$. (The properties of the symbol classes $S^k_\rho$
required for the construction of~\cite[Section~14.2]{e-z} are derived
as in~\cite[Section~4.4]{e-z}; see also~\cite[Chapter~18]{ho3} or~\cite[Chapter~3]{gr-s}.) As
before, denote $\Psi^k=\Psi^k_0$. Since our symbols can grow
arbitrarily fast as $m\to\infty$, we do not make any a priori
assumptions on the behavior of elements of $\Psi^k_\rho$ near the
infinity in $M$. However, we require that all operators
$A\in\Psi^k(M)$ be \emph{properly supported}; namely, the
restriction of each of the projection maps $\pi_m,\pi_{m'}:M\times
M\to M$ to the support of the Schwartz kernel $K_A(m,m')$ of $A$ is a
proper map. See for example~\cite[Proposition~18.1.22]{ho3}
for how to obtain properly supported quantizations on noncompact
manifolds. Then each element of $\Psi^k(M)$ acts $H^s_{h,\loc}(M)\to
H^{s-k}_{h,\loc}(M)$, where $H^s_{h,\loc}(M)$ denotes the space of
distributions locally in the semiclassical Sobolev space $H_h^s$
(see for example~\cite[Section~7.1]{e-z} for the definition
of semiclassical Sobolev spaces).
We also include properly supported operators that are $\mathcal O(h^\infty)_{\Psi^{-\infty}}$
into all considered pseudodifferential classes, see for example~\cite[Definition~18.1.20]{ho3}.

We have the semiclassical principal symbol map
$$
\sigma:\Psi^k_\rho(M)\to S^k_\rho(M)/h^{1-2\rho}S^{k-1}_\rho(M)
$$
and its right inverse, a non-canonical quantization map
$$
\Op_h:S^k_\rho(M)\to\Psi^k_\rho(M).
$$
For $A\in\Psi^k_\rho(M)$, we often use $\sigma(A)$ to denote
any representative of the corresponding equivalence class, hence
the remainder terms below.
The standard operations of pseudodifferential calculus with symbols
in $S^k_\rho$ have an $\mathcal O(h^{1-2\rho})$ remainder instead
of the $\mathcal O(h)$ remainder valid for the class
$S^k_0$. More precisely, we have for $A\in\Psi^k_\rho(M)$ and
$B\in\Psi^{k'}_\rho(M)$,
$$
\begin{gathered}
\sigma(A^*)=\overline{\sigma(A)}+\mathcal O(h^{1-2\rho})_{S^{k-1}_\rho(M)},\\
\sigma(AB)=\sigma(A)\sigma(B)+\mathcal O(h^{1-2\rho})_{S^{k+k'-1}_\rho(M)},\\
\sigma([A,B])=-ih \{\sigma(A),\sigma(B)\}+\mathcal O(h^{2(1-2\rho)})_{S^{k+k'-2}_{\rho}(M)}.
\end{gathered}
$$
Here $\{\cdot,\cdot\}$ stands for the Poisson bracket and the adjoint is with respect to $L^2(M)$. 
The $\mathcal O(\cdot)$ notation is used in the present paper in the following way:
we write $u=\mathcal O_z(F)_{\mathcal X}$ if the norm of the function,
or the operator, $u$ in the functional space $\mathcal X$ is bounded
by the expression $F$ times a constant depending on the parameter $z$.

\smallsection{Wavefront sets}
If $A:C^\infty(M)\to C^\infty(M)$ is a properly supported operator, we
say that $A=\mathcal O(h^\infty)_{\Psi^{-\infty}}$ if $A$ is smoothing
and each of the $C^\infty(M\times M)$ seminorms of its Schwartz kernel
is $\mathcal O(h^\infty)$.  For each $A\in\Psi^k_\rho(M)$, we have
$A=\Op_h(a)+\mathcal O(h^\infty)_{\Psi^{-\infty}}$ for some $a\in
S^k_\rho(M)$. Define the semiclassical wavefront set $\WFh(A)\subset
\overline T^*M$ of $A$ as follows: a point $(m,\nu)\in \overline T^*M$
does not lie in $\WFh(A)$, if there exists a neighborhood $U$ of
$(m,\nu)$ in $\overline T^*M$ such that each $(m,\nu)$-derivative of
$a$ is $\mathcal O(h^\infty\langle\nu\rangle^{-\infty})$ in $U\cap
T^*M$.

Operators with compact wavefront sets in $T^*M$ are called \emph{compactly
microlocalized}; those are exactly operators of the form
$\Op_h(a)+\mathcal O(h^\infty)_{\Psi^{-\infty}}$ for some $a\in
S^{\comp}_\rho$.  We denote by $\Psi^{\comp}_\rho(M)$ the class of all
compactly microlocalized elements of $\Psi^k_\rho(M)$; as before, we
put $\Psi^{\comp}(M)=\Psi^{\comp}_0(M)$. Compactly microlocalized
operators should not be confused with \emph{compactly supported}
operators (operators whose Schwartz kernels are compactly
supported).

We will need a finer notion of microsupport on $h$-dependent sets, used in
the proofs in Sections~\ref{s:proofs-2} and~\ref{s:proofs-3},
for example in Proposition~\ref{l:analysis2-3}:
%
%
\begin{defi}\label{d:microlocal-vanishing}
An operator $A\in\Psi^{\comp}_\rho(M)$ is said to be
\emph{microsupported} on an $h$-dependent family of sets
$V(h)\subset T^*M$, if we can write $A=\Op_h(a)+\mathcal
O(h^\infty)_{\Psi^{-\infty}}$, where for each compact set $K\subset
T^*M$, each differential operator $\partial^\alpha$ on $T^*M$, and
each $N$, there exists a constant $C_{\alpha N K}$ such that for $h$
small enough,
$$
\sup_{(m,\nu)\in K\setminus V(h)}|\partial^\alpha a(m,\nu;h)|\leq C_{\alpha N K}h^N.
$$
\end{defi}
%
%
Since the change of variables formula for the full symbol of a
pseudodifferential operator~\cite[Theorem~9.10]{e-z} contains an asymptotic expansion in powers
of $h$, consisting of derivatives of the original symbol,
Definition~\ref{d:microlocal-vanishing} does not depend on the choice
of the quantization procedure $\Op_h$.  Moreover, if
$A\in\Psi^{\comp}_\rho$ is microsupported inside some $V(h)$ and
$B\in\Psi^k_\rho$, then $AB$, $BA$, and $A^*$ are also microsupported
inside $V(h)$. It follows from the definition of the wavefront set
that $(m,\nu)\in T^*M$ does not lie in $\WFh(A)$ for some
$A\in\Psi^{\comp}_\rho$, if and only if there exists an
$h$-independent neighborhood $U$ of $(m,\nu)$ such that $A$ is
microsupported on the complement of $U$. Note however that $A$
need not be microsupported on $\WFh(A)$, though it will
be microsupported on any $h$-independent neighborhood of $\WFh(A)$.
Finally, it can be seen by
Taylor's formula that if $A\in\Psi^{\comp}_\rho(M)$ is microsupported
in $V(h)$ and $\rho'>\rho$, then $A$ is also microsupported on the set
of all points in $V(h)$ which are at least $h^{\rho'}$ away from the
complement of $V(h)$.

\smallsection{Ellipticity}
For $A\in\Psi^k_\rho(M)$, define its \emph{elliptic set}
$\Ell(A)\subset \overline T^*M$ as follows: $(m,\nu)\in\Ell(A)$ if and
only if there exists a neighborhood $U$ of $(m,\nu)$ in $\overline
T^*M$ and a constant $C$ such that $|\sigma(A)|\geq
C^{-1}\langle\nu\rangle^k$ in $U\cap T^*M$.  The following statement
is the standard semiclassical elliptic estimate;
see~\cite[Theorem~18.1.24']{ho3} for the closely related microlocal
case and for example~\cite[Section~2.2]{zeeman} for the semiclassical
case.
%
%
\begin{prop}\label{l:elliptic}
Assume that $P\in \Psi_\rho^k(M)$, $A\in \Psi_\rho^{k'}(M)$,
and $\WFh(A)\subset\Ell(P)$. Assume moreover that $A$ is compactly
supported.
Then there exists a constant $C$ and a function $\chi\in C_0^\infty(M)$
such that for each $s\in \mathbb R$, each $u\in H^{s+k'}_{h,\loc}(M)$ and each $N$, we have
\[
\|Au\|_{H^s_h}\leq C\|\chi Pu\|_{H^{s+k'-k}_h}+\mathcal O(h^\infty)\|\chi u\|_{H^{-N}}.
\]
Moreover, if $P$ is a differential operator, then we can take any
$\chi$ such that the Schwartz kernel of $A$ is supported in
$\{\chi\neq 0\}\times\{\chi\neq 0\}$.
\end{prop}
%
%
\smallsection{Semi-classical wave-front sets of distributions}
An $h$-dependent
family $u(h)\in \mathcal D'(M)$ is called \emph{h-tempered}, if
for each open $U$ compactly contained in $M$, there exist constants
$C$ and $N$ such that 
\begin{equation}
  \label{tempered}
\|u(h)\|_{H^{-N}_{h}(U)}\leq Ch^{-N}.
\end{equation}
For a tempered distribution $u$, we say that $(m_0,\nu_0)\in \overline
T^*M$ does not lie in the wavefront set $\WFh(u)$, if there exists a
neighborhood $V(m_0,\nu_0)$ in $\overline T^*M$ such that for each
$A\in\Psi^0(M)$ with $\WFh(A)\subset V$, we have $Au=\mathcal
O(h^\infty)_{C^\infty}$.  By Proposition~\ref{l:elliptic},
$(m_0,\nu_0)\not\in\WFh(u)$ if and only if there exists compactly
supported $A\in\Psi^0(M)$ elliptic at $(m_0,\nu_0)$ such that
$Au=\mathcal O(h^\infty)_{C^\infty}$.  The wavefront set of $u$ is a
closed subset of $\overline T^*M$; it is empty if and only if
$u=\mathcal O(h^\infty)_{C^\infty(M)}$. We can also verify that for
$u$ tempered and $A\in\Psi^k_\rho(M)$,
$\WFh(Au)\subset\WFh(A)\cap\WFh(u)$.

\subsection{Semiclassical Lagrangian distributions}
  \label{s:prelim.lagrangian}

In this subsection, we review some facts from the theory of
semiclassical Lagrangian distributions.  See~\cite[Chapter~6]{g-s}
or~\cite[Section~2.3]{svn} for a detailed account,
and~\cite[Section~25.1]{ho4} or~\cite[Chapter~11]{gr-s} for the
closely related microlocal case.  However, note that we do not attempt
to define the principal symbols as global invariant geometric objects;
this makes the resulting local theory considerably simpler.

\smallsection{Phase functions}
Let $M$ be a manifold without boundary.  We denote its dimension
by $d$; in the convention used in the present paper, $d=n+1$. As
before, we denote elements of $T^*M$ by $(m,\nu)$, $m\in M$, $\nu\in
T^*_m M$.  Let $\varphi(m,\theta)$ be a smooth real-valued function on
some open subset $U_\varphi$ of $M\times \mathbb R^L$, for some $L$;
we call $m$ \emph{base variables} and $\theta$ \emph{oscillatory
variables}. We say that $\varphi$ is a (nondegenerate) \emph{phase
function}, if the differentials $d
(\partial_{\theta_1}\varphi),\dots,d(\partial_{\theta_L}\varphi)$ are
linearly independent on the \emph{critical set}
\begin{equation}
  \label{e:c-varphi}
C_\varphi:=\{(m,\theta)\mid \partial_\theta \varphi=0\}\subset U_\varphi.
\end{equation}
In this case
\[
\Lambda_\varphi:=\{(m,\partial_m\varphi(m,\theta))\mid (m,\theta)\in C_\varphi\}\subset T^*M
\]
is an (immersed, and we will shrink the domain of $\varphi$ to make it
embedded) Lagrangian submanifold. We say that $\varphi$
\emph{generates} $\Lambda_\varphi$.

\smallsection{Symbols}
Let $\rho\in [0,1/2)$. A smooth function $a(m,\theta;h)$ is called a
compactly supported symbol of type $\rho$ on $U_\varphi$, if it is supported
in some compact $h$-independent subset of $U_\varphi$, and for each
differential operator $\partial^\alpha$ on $M\times \mathbb R^L$,
there exists a constant $C_\alpha$ such that
$$
\sup_{U_\varphi}|\partial^\alpha a|\leq C_\alpha h^{-\rho|\alpha|}.
$$
Similarly to Section~\ref{s:prelim.notation}, we write $a\in
S^{\comp}_\rho(U_\varphi)$ and denote $S^{\comp}:=S^{\comp}_0$. 

\smallsection{Lagrangian distributions}
Given a phase function $\varphi$ and a symbol $a\in
S^{\comp}_\rho(U_\varphi)$, consider the $h$-dependent family of
functions
\begin{equation}
  \label{e:lagrangian-basic}
u(m;h)=h^{-L/2}\int_{\mathbb R^L} e^{i\varphi(m,\theta)/h}a(m,\theta;h)\,d\theta.
\end{equation}
We call $u$ a \emph{Lagrangian distribution} of type $\rho$ generated by $\varphi$.
By the method of non-stationary phase, if
$\supp a$ is contained in some $h$-independent compact
set $K\subset U_\varphi$, then
\begin{equation}
  \label{e:lag-wf}
\WFh(u)\subset\{(m,\partial_m\varphi(m,\theta))\mid (m,\theta)\in C_\varphi\cap K\}\subset\Lambda_\varphi.
\end{equation}
The \emph{principal symbol} $\sigma_\varphi(u)\in S^{\comp}_\rho(\Lambda_\varphi)$ of $u$ 
is defined modulo $\mc{O}(h^{1-2\rho})$ by the expression
\begin{equation}
  \label{e:lagrangian-symbol}
\sigma_\varphi(u)(m,\partial_m\varphi(m,\theta);h)= a(m,\theta;h),\
(m,\theta)\in C_\varphi.
\end{equation}
That $\sigma_\varphi(u)$ does not depend (modulo $\mathcal O(h^{1-2\rho})$) on the choice of $a$
producing $u$ will
follow from Proposition~\ref{l:lagrangian-basic}
and~\eqref{e:lagrangian-key2}.

Following~\cite[Chapter~11]{gr-s}, we introduce a certain (local)
canonical form for Lagrangian distributions. Fix some local system of
coordinates on $M$ (shrinking $M$ to the domain of this coordinate
system and identifying it with a subset of $\mathbb R^d$) and consider
\begin{equation}
  \label{e:lambda-can}
\Lambda_F=\{(m,\nu)\mid m=-\partial_\nu F(\nu),\ \nu\in U_F\}\subset T^*M,
\end{equation}
where $F$ is a smooth real-valued function on some open set
$U_F\subset \mathbb R^d$, such that the image of $-\partial_\nu F$ is
contained in $M$. Then $\Lambda_F$ is Lagrangian; in fact, it is
generated by the phase function $m\cdot\nu+F(\nu)$, with $\nu$ the
oscillatory variable.  One can also prove that any Lagrangian
submanifold not intersecting the zero section (which is
always the case for the Lagrangians considered in this paper) can be locally brought under the
form~\eqref{e:lambda-can} for an appropriate choice of the coordinate
system on $M$~-- see for example~\cite[Lemma~9.5]{gr-s}.

If $b(\nu;h)\in S^{\comp}_\rho(U_F)$ and $\chi\in C_0^\infty(M)$ is
equal to $1$ near $-\partial_\nu F(\supp b)$, then we can define a
Lagrangian distribution by the following special case
of~\eqref{e:lagrangian-basic}:
\begin{equation}
  \label{e:lagrangian-key}
v(m;h)=\chi(m)h^{-d/2}\int_{U_F} e^{i(m\cdot\nu+F(\nu))/h}b(\nu;h)\,d\nu.
\end{equation}
We need $\chi$ to make $v\in C_0^\infty(M)$; however,
by~\eqref{e:lag-wf} (or directly by the method of nonstationary
phase), if we choose $\chi$ differently, then $v$ will change by
$\mathcal O(h^\infty)_{C_0^\infty}$.

If $v$ is given by~\eqref{e:lagrangian-key}, then we can recover
the symbol $b$ by the Fourier inversion formula:
\begin{equation}
  \label{e:lagrangian-key2}
e^{i F(\nu)/h}b(\nu;h)=(2\pi)^{-d}h^{-d/2}\int_M
e^{-im\cdot\nu/h} v(m;h)\,dm+\mathcal O(h^\infty)_{\mathscr S(\mathbb R^d)}.
\end{equation}
Note that
if $v=\mathcal O(h^\infty)_{C_0^\infty}$, then $b(\nu;h)=\mathcal
O(h^\infty)_{C_0^\infty}$.  Moreover, if $v\in C_0^\infty(M)$
satisfies~\eqref{e:lagrangian-key2} for some $b\in
S^{\comp}_\rho(U_F)$, then $v$ is given by~\eqref{e:lagrangian-key}
modulo $\mathcal O(h^\infty)_{C_0^\infty}$.
Following~\cite[Chapter~11]{gr-s}, by the method of stationary phase
each Lagrangian distribution can be brought locally
into the form~\eqref{e:lagrangian-key}:
%
%
\begin{prop}\label{l:lagrangian-basic}
Assume that $\varphi$ is a phase function, and the corresponding
Lagrangian $\Lambda=\Lambda_\varphi$ can be written in the
form~\eqref{e:lambda-can}.  For $a(m,\theta;h)\in
S^{\comp}_\rho(U_\varphi)$ and $b(\nu;h)\in S^{\comp}_\rho(U_F)$,
denote by $u_a$ and $v_b$ the functions given
by~\eqref{e:lagrangian-basic} and~\eqref{e:lagrangian-key},
respectively. Then:

1. For each $a\in S^{\comp}_\rho(U_\varphi)$, there exists $b\in S^{\comp}_\rho(U_F)$
such that $u_a=v_b+\mathcal O(h^\infty)_{C_0^\infty}$. Moreover, we have
the following asymptotic decomposition for $b$:
\begin{equation}
  \label{e:lagrangian-pp-exp}
b(\nu;h)=\sum_{0\leq j<N} h^j L_j a(m,\theta;h)+\mathcal O(h^{N(1-2\rho)})_{S^{\comp}_\rho(U_F)},
\end{equation}
where each $L_j$ is a differential operator of order $2j$ on
$U_\varphi$, and $(m,\theta)\in C_\varphi$ is the solution to
the equation $(m,\partial_m\varphi(m,\theta))=(-\partial_\nu F(\nu),\nu)$.
In particular, if $\sigma_\varphi(u)$ is given by~\eqref{e:lagrangian-symbol},
then
\begin{equation}
  \label{e:lagrangian-pp}
\sigma_\varphi(u)(-\partial_\nu F(\nu),\nu;h)=f_{\varphi F} b(\nu;h)+\mathcal O(h^{1-2\rho})_{S^{\comp}_\rho(U_F)},
\end{equation}
where $f_{\varphi F}$ is some nonvanishing function depending on $\varphi$ and the
coordinate system on $M$. Adding a certain constant to the function $F$,
we can make $f_{\varphi F}$ independent of $h$.

2. For each $b\in S^{\comp}_\rho(U_F)$, there exists $a\in S^{\comp}_\rho(U_\varphi)$
such that $v_b=u_a+\mathcal O(h^\infty)_{C_0^\infty}$.
\end{prop}
%
%
%
%
\begin{defi}
  \label{d:lagrangian}
Let $\Lambda\subset T^*M$ be an embedded Lagrangian submanifold.
We say that an $h$-dependent family of functions
$u(m;h)\in C_0^\infty(M)$ is a (compactly supported
and compactly microlocalized) 
Lagrangian distribution of type $\rho$ associated to $\Lambda$, if
it can be written as a sum of finitely many functions
of the form~\eqref{e:lagrangian-basic}, for different phase functions
$\varphi$ parametrizing open subsets of $\Lambda$, plus an
$\mathcal O(h^\infty)_{C_0^\infty}$ remainder. 
Denote by $I^{\comp}_\rho(\Lambda)$ the space of all such distributions,
and put $I^{\comp}(\Lambda):=I^{\comp}_0(\Lambda)$.
\end{defi}
%
%
By Proposition~\ref{l:lagrangian-basic}, if $\varphi$ is a phase function and $u\in
I^{\comp}_\rho(\Lambda_\varphi)$, then $u$ can be written in the
form~\eqref{e:lagrangian-basic} for some symbol $a$, plus an $\mathcal
O(h^\infty)_{C_0^\infty}$ remainder.  The symbol~$\sigma_\varphi(u)$,
given by~\eqref{e:lagrangian-symbol}, is well-defined modulo $\mathcal
O(h^{1-2\rho})$.

The action of a pseudodifferential operator on a Lagrangian
distribution is given by the following proposition, following from
Proposition~\ref{l:lagrangian-basic} and the method of stationary
phase:
%
%
\begin{prop}\label{l:lagrangian-mul}
Let $u\in I^{\comp}_\rho(\Lambda)$ and $P\in \Psi^k_\rho(M)$.
Then $Pu\in I^{\comp}_\rho(\Lambda)$. Moreover,

1. If $\Lambda=\Lambda_\varphi$ for some phase function $\varphi$,
then
$$
\sigma_\varphi(Pu)=\sigma(P)|_{\Lambda_\varphi}\cdot\sigma_\varphi(u)
+\mathcal O(h^{1-2\rho})_{S^{\comp}_\rho(\Lambda)}.
$$

2. Assume that $\Lambda=\Lambda_F$ is given by~\eqref{e:lambda-can} in
some coordinate system on $M$. Let $b(\nu;h)$ and $b^P(\nu;h)$ be the
symbols corresponding to $u$ and $Pu$, respectively,
via~\eqref{e:lagrangian-key}. Let also $P=\Op_h(p)$ for some
quantization procedure $\Op_h$. Then
$$
b^P(\nu;h)=\sum_{0\leq j<N}h^jL_j (p(m,\nu';h)b(\nu;h))|_{\nu'=\nu,\,m=-\partial_\nu F(\nu)}
+\mathcal O(h^{N(1-2\rho)})_{S^{\comp}_\rho(U_F)},
$$
where each $L_j$ is a differential operator of order $2j$ on $M\times
U_F\times U_F$.
\end{prop}
%
%
Finally, we give the following estimate of the $L^2$ norm of a Lagrangian
distribution, following from the boundedness of the Fourier transform on $L^2$:
%
%
\begin{prop}\label{l:lagrangian-norm}
Assume that $u\in I^{\comp}_\rho(\Lambda_F)$, where $\Lambda_F$ is
given by~\eqref{e:lambda-can}. Let $u$ be given
by~\eqref{e:lagrangian-key}, with $b(\nu;h)$ the corresponding
symbol. Then for some constant $C$ independent of $h$,
\begin{equation}
  \label{e:lagrangian-norm}
\|u(m;h)\|_{L^2}\leq C\|b(\nu;h)\|_{L^2(U_F)}.
\end{equation}
\end{prop}
%
%
\smallsection{Fourier integral operators}
A special case of Lagrangian distributions are Fourier integral
operators associated to canonical transformations.  Let $M,M'$ be two
manifolds of the same dimension $d$, and let $\kappa$ be a
symplectomorphism from an open subset of $T^*M$ to an open subset of
$T^*M'$. Consider the Lagrangian
\[
\Lambda_\kappa=\{(m,\nu;m',-\nu') \mid \kappa(m,\nu)=(m',\nu')\}\subset T^*M\times T^*M'=T^*(M\times M').
\]
A compactly supported operator $U:\mathcal D'(M')\to C_0^\infty(M)$ is
called a (semiclassical) \emph{Fourier integral operator} of type
$\rho$ associated to $\kappa$, if its Schwartz kernel $K_U(m,m')$ lies
in $h^{-d/2}I^{\comp}_\rho(\Lambda_\kappa)$. We write $U\in
I^{\comp}_\rho(\kappa)$.  Note that we quantize a canonical
transformation $T^*M\to T^*M'$ as an operator $\mathcal D'(M')\to
C_0^\infty(M)$, in contrast with the standard convention, which would
quantize it as an operator $\mathcal D'(M)\to C_0^\infty(M')$.  The
$h^{-d/2}$ factor is explained as follows: the normalization for
Lagrangian distributions is chosen so that $\|u\|_{L^2}\sim 1$, while
the normalization for Fourier integral operators is chosen so that
$\|U\|_{L^2(M')\to L^2(M)}\sim 1$.

After sufficiently shrinking the domain of $\kappa$ and choosing an
appropriate coordinate system on $M'$ (which is possible for all $\kappa$
whose graph does not intersect the zero section of $T^*M'$, see the remark following~\eqref{e:lambda-can}),
we can find a generating function $S(m,\nu')$ for $\kappa$;
that is,
\begin{equation}
  \label{e:generating-function}
\kappa(m,\nu)=(m',\nu')\iff \partial_m S(m,\nu')=\nu,\
\partial_{\nu'} S(m,\nu')=m'.
\end{equation}
Here $(m,\nu')$ vary in some open set $U_S\subset M\times \mathbb
R^d$. The phase function $S(m,\nu')-m'\cdot\nu'$, with $\nu'$ the
oscillatory variable, parametrizes $\Lambda_\kappa$ and for $U\in
I^{\comp}_\rho(\kappa)$, we write similarly
to~\eqref{e:lagrangian-key},
\begin{equation}
  \label{e:fourior-basic}
K_U(m,m')=h^{-d}\chi(m')\int_{\mathbb R^d} e^{{i\over h}(S(m,\nu')-m'\cdot\nu')}b(m,\nu';h)
\,d\nu'+\mathcal O(h^\infty)_{C_0^\infty},
\end{equation}
for some symbol $b\in S^{\comp}_\rho(U_S)$ and any $\chi\in
C_0^\infty(M')$ such that $\chi=1$ near the set
$\partial_{\nu'}S(\supp b)$.  The function $b$ is determined uniquely
by $U$ modulo $\mathcal O(h^\infty)_{S^{\comp}_\rho(U_S)}$, similarly
to~\eqref{e:lagrangian-key2}.  Note that if $\kappa$ is the identity
map, then $S(m,\nu')=m\cdot \nu'$ and we arrive to the quantization
formula for a semiclassical pseudodifferential operator.

Similarly to Proposition~\ref{l:lagrangian-mul}, we have
%
%
\begin{prop}
  \label{l:fourior-mul}
Assume that $U\in I^{\comp}_\rho(\kappa)$ and $P\in\Psi^k_\rho(M')$.
Then $UP\in I^{\comp}_\rho(\kappa)$.
If moreover $\kappa$ is given by~\eqref{e:generating-function},
$b(m,\nu';h)$ and $b^P(m,\nu';h)$ are the symbols corresponding
to $U$ and $UP$, respectively, via~\eqref{e:fourior-basic},
and $P=\Op_h(p)$ for some quantization procedure $\Op_h$,
then we have the following asymptotic decomposition
for $b^P$:
\[
b^P(m,\nu';h)=\sum_{0\leq j<N}h^j L_j(p(m',\tilde \nu)b(m,\nu'))|_{\tilde \nu=\nu',\,
m'=\partial_{\nu'} S(m,\nu')}+\mathcal O(h^{N(1-2\rho)})_{S^{\comp}_\rho(U_S)}.
\]
Here each $L_j$ is a differential operator of order $2j$ on $M'\times \mathbb R^d\times U_S$.
In particular,
$$
b^P(m,\nu';h)=p(\partial_{\nu'} S(m,\nu'),\nu';h)b(m,\nu';h)+\mathcal O(h^{1-2\rho})_{S^{\comp}_\rho(U_S)}.
$$
A similar statement is true for an operator of the form $PU$, where $P\in\Psi^k_\rho(M)$ and
the terms of the asymptotic decomposition have the form~$h^j L_j(p(\widetilde m,\nu)b(m,\nu'))|_{\widetilde m=m,\,
\nu=\partial_m S(m,\nu')}$. 
\end{prop}
%
%

\subsection{Schr\"odinger propagators}
  \label{s:prelim.propagator}

Let $(M,g)$ be a complete Riemannian
manifold, $\Delta=\Delta_g$ the corresponding (nonnegative)
Laplace--Beltrami operator, and $p(m,\nu)=|\nu|_g^2$ the semiclassical principal
symbol of $h^2\Delta\in\Psi^2(M)$.
We use the notation
$S^*M=p^{-1}(1)\subset T^*M$
for the unit cotangent bundle. The geodesic flow $g^t$ on $T^*M$ is
related to the Hamiltonian flow $e^{tH_p}$ of $p$ by the formula
$g^t=e^{tH_p/2}$. The
operator $\Delta$ is essentially self-adjoint on $L^2(M)$ by~\cite{Ch}
and its domain is given by the Friedrichs extension. Let
$$
U(t)=e^{ith\Delta/2}=e^{{it\over h}(h^2\Delta/2)}
$$
be the semiclassical Schr\"odinger propagator; it is a unitary
operator on $L^2(M)$. The basic microlocal properties of $U(t)$
are given by the following
%
%
\begin{prop}
  \label{l:egorov}
For each $t\in \mathbb R$,

1. (Egorov's Theorem) For each compactly supported
$A\in\Psi^{\comp}_\rho(M)$, there exists compactly supported
$A^t\in\Psi^{\comp}_\rho(M)$ such that
\begin{equation}
  \label{e:egorov}
U(t)AU(-t)=A^t+\mathcal O(h^\infty)_{L^2\to L^2}.
\end{equation}
Moreover, $\WFh(A^t)\subset g^{-t}(\WFh(A))$ and
$\sigma(A^t)=\sigma(A)\circ g^t+\mathcal O(h^{1-2\rho})$.

2. (Microlocalization) $U(t)$ is microlocalized
on the graph of $g^{-t}$, namely if
$A,B\in\Psi^k_\rho(M)$ are compactly supported and at least
one of them is compactly microlocalized, then
\begin{equation}
  \label{e:mic-req}
g^t(\WFh(A))\cap\WFh(B)=\emptyset\ \Longrightarrow\ 
AU(t)B=\mathcal O(h^\infty)_{L^2\to L^2}.
\end{equation}

3. (Parametrix) If $A\in\Psi^{\comp}(M)$ is compactly supported, then
$U(t)A$ is the sum of a compactly microlocalized Fourier integral
operator (of type 0) associated to $g^t$, as defined in
Section~\ref{s:prelim.lagrangian}, and an $\mathcal
O(h^\infty)_{L^2\to L^2}$ remainder. 
\end{prop}
%
%
The proofs are standard; part~1 can be found
in~\cite[Theorem~11.1]{e-z} (with the mildly exotic classes
$\Psi^{\comp}_\rho$ handled as in Appendix~\ref{s:ehrenfest}), part~2
follows directly from part~1, and part~3 is proved similarly
to~\cite[Theorem~10.4]{e-z}.  The operator $U(t)A$ quantizes $g^t$,
not $g^{-t}$, because of the convention adopted in
Section~\ref{s:prelim.lagrangian} that a canonical transformation
$T^*M\to T^*M'$ is quantized as an operator $\mathcal D'(M')\to
C_0^\infty(M)$.

\smallsection{Egorov's theorem until the Ehrenfest time}
Proposition~\ref{l:egorov} is valid for bounded times $t$; as $t\to \infty$,
the constants in the estimates for the corresponding symbols will blow up.
However, it is still possible to prove Egorov's Theorem for $t$ bounded
by a certain multiple of $\log(1/h)$, called the Ehrenfest time.
To define this time, we fix an open bounded 
set $U\subset M$ with geodesically convex closure in the sense of~\eqref{geodconvex}
and define the \emph{maximal expansion rate}
\begin{equation}
  \label{e:lambda-max}
\Lambda_{\max}:=\limsup_{|t|\to \infty}{1\over |t|}\log\sup_{m\in U,\,|\nu|_g=1,\atop g^t(m,\nu)\in U}
\|dg^t(m,\nu)\|.
\end{equation}
Here $\|dg^t(m,\nu)\|$ is the operator norm of the differential
$
dg^t(m,\nu):T_{(m,\nu)}T^*M\to T_{g^t(m,\nu)}T^*M
$
with respect to any given smooth norm on the fibers of $T(T^*M)$.
Since we will work on a noncompact manifold, we introduce cutoffs
into the corresponding propagators:
%
%
\begin{prop}\label{l:ehrenfest}
Assume that $X_1,X_2\in \Psi^0(M)$ satisfy $\|X_j\|_{L^2\to L^2}\leq 1+\mathcal O(h)$
and are compactly supported inside $U$.
Let $\varepsilon_e>0$ and take $\Lambda_0,\Lambda'_0>0$ such that
$\Lambda_0>\Lambda'_0>(1+2\varepsilon_e)\Lambda_{\max}$.
Fix $t_0\in \mathbb R$. Then for each integer
\begin{equation}\label{e:ehrenfest-l}
l\in [0, \log(1/h)/(2|t_0|\Lambda_0)],
\end{equation}
and each compactly supported $A\in \Psi^{\comp}(M)$
with
$
\WFh(A)\subset \mathcal E_{\varepsilon_e}:=\{1-\varepsilon_e\leq |\nu|_g\leq 1+\varepsilon_e\}
$,
the compactly supported operator
$
A^{(l)}:=(X_2 U(t_0))^l A (U(-t_0)X_1)^l
$
lies in $\Psi^{\comp}_{\rho_l}(M)$, modulo an $\mathcal
O(h^\infty)_{L^2\to L^2}$ remainder, with
\begin{equation}
  \label{e:rho-j}
\rho_l=l|t_0|\Lambda'_0/\log(1/h)<1/2.
\end{equation}
Moreover, the $S^{\comp}_{\rho_l}$ seminorms of the full symbol of
$A^{(l)}$ are bounded uniformly in $l$, in the following sense: the
order $k$ derivatives of this symbol are bounded by $Ch^{-k\rho_l}$,
where $C$ is a constant independent of $h$ and $l$. The principal symbol
of $A^{(l)}$ is
$$
\sigma(A^{(l)})=(\sigma(A)\circ g^{lt_0})\prod_{j=0}^{l-1}(\sigma(X_1)\sigma(X_2))\circ g^{jt_0}
+\mathcal O(h^{1-2\rho_l}). 
$$
The wavefront set of $A^{(l)}$, for $l>0$, is contained
in $\WFh(X_1)\cap\WFh(X_2)\cap \mathcal E_{\varepsilon_e}$.
Finally, if $U_A$ and $U_X$ are open sets such that
$\WFh(A)\subset U_A$ and $\WFh(X_1)\cap\WFh(X_2)\subset U_X$,
then $A^{(l)}$ is microsupported, in the sense of Definition~\ref{d:microlocal-vanishing},
inside the set
$$
V^{(l)}:=g^{-lt_0}(U_A)\cap \bigcap_{j=0}^{l-1} g^{-jt_0}(U_X).
$$
The set $V^{(l)}$ does not depend on $h$ directly, however it depends on $l$,
which is allowed to depend on $h$, and our microlocal vanishing statement
is uniform in $l$. 
\end{prop}
%
%
Proposition~\ref{l:ehrenfest} is the main technical tool of obtaining
the polynomial remainder bound of Theorem~\ref{t:remainder}; it is
also the reason why the classes $\Psi^{\comp}_\rho$ appear. Its proof,
following the methods of~\cite[Section~5.2]{AnNon}
and~\cite[Theorem~11.12]{e-z}, is given in Appendix~\ref{s:ehrenfest}.
See also~\cite[Theorem~7.1]{gabriel} for a more refined estimate in the
more restrictive setting of two-dimensional manifolds with hyperbolic
geodesic flows.
We do not impose any restrictions on the set $U$ at this point, however
in our actual argument it will have to contain a neighborhood of the
trapped set~-- see the beginning of Section~\ref{s:proofs-2}.

\smallsection{Propagating generalized eigenfunctions}
The following fact, similar to~\cite[Proposition~3.3]{eiscusp},
will be used to propagate the Eisenstein functions by the group $U(t)$:
%
%
\begin{lemm}\label{l:key}
Assume that $u\in C^\infty(M)$ solves the equation
\[
(h^2\Delta-z)u=0,\
|1-z|\leq Ch.
\]
Let $\chi\in C_0^\infty(M)$; take $t\in \mathbb R$ and assume
that $\chi_t\in C_0^\infty(M)$ is supported in the interior
of a compact set $K_t\subset M$ and satisfies
\begin{equation}
  \label{e:key-cond}
d_g(\supp\chi,\supp(1-\chi_t))>|t|.
\end{equation}
Here $d_g$ denotes Riemannian distance on $M$. Then 
\[
\chi u=\chi e^{-itz/(2h)}U(t)\chi_t u
  +\mathcal O(h^\infty\|u\|_{L^2(K_t)})_{L^2(M)}.
\]
\end{lemm}
\begin{proof}
Without loss of generality, we assume that $t\geq 0$.
For $0\leq s\leq t$, define
$
u_s=\chi(u-e^{-isz/(2h)}U(s)\chi_t u)
$.
We need to prove that
\begin{equation}\label{e:key-int}
\|u_t\|_{L^2}=\mathcal O(h^\infty)\|u\|_{L^2(K_t)}.
\end{equation}
Since $\chi=\chi\chi_t$, we have $u_0=0$; next,
$$
\begin{gathered}
2hD_s u_s
  = - \chi e^{-isz/(2h)}U(s)(h^2\Delta-z)\chi_t u
  = - e^{-isz/(2h)}\chi U(s) [h^2\Delta,\chi_t]u.
\end{gathered}
$$
Let $B\in\Psi^{\comp}$ be compactly supported inside $K_t\times K_t$,
equal to the identity microlocally near $\supp\chi_t\cap S^*M$, but
microlocalized in a small neighborhood of $S^*M$ so that
by~\eqref{e:key-cond},
\[
g^s(\supp\chi)\cap \WFh(B)\cap \supp(1-\chi_t)=\emptyset.
\]
Note that $\WFh([h^2\Delta,\chi_t])\subset\supp (1-\chi_t)$.
Then by part~2 of Proposition~\ref{l:egorov},
\begin{equation}
  \label{e:key-1}
\|\chi U(s)[h^2\Delta,\chi_t]Bu\|_{L^2}=\mathcal O(h^\infty)\|u\|_{L^2(K_t)},\
0\leq s\leq t.
\end{equation}
Moreover, by Proposition~\ref{l:elliptic}
\begin{equation}
  \label{e:key-2}
\|\chi U(s)[h^2\Delta,\chi_t](1-B)u\|_{L^2}=\mathcal O(h^\infty)\|u\|_{L^2(K_t)}.
\end{equation}
Combining~\eqref{e:key-1} and~\eqref{e:key-2}, we get
$\|\pl_s u_s\|_{L^2}=\mathcal O(h^\infty)\|u\|_{L^2(K_t)}$; it remains
to integrate in $s$ to get~\eqref{e:key-int}.
\end{proof}
%
%
\smallsection{Hilbert--Schmidt norm estimates}
We now prove Hilbert--Schmidt norm estimates for the product of a
pseudodifferential operator with a spectral
projector. (See~\cite[Section~19.1]{ho3} for the properties of
Hilbert--Schmidt and trace class operators.)  To simplify notation, we
consider a spectral interval of size $h$ centered at
$\lambda=1$; similar statement is true for the interval
$[\lambda+c_1h,\lambda+c_2h]$ with $\lambda>0$, replacing $S^*M$ by $\lambda S^*M$. 
%
%
\begin{lemm}\label{l:h-s-estimate}
Fix $c_1,c_2\in \mathbb R$ and
let $\indic_{[1+c_1h,1+c_2h]}(h^2\Delta)$ be defined by means of
spectral theory. Assume that $A\in\Psi^{\comp}_\rho(M)$ is compactly
supported. Then 
\begin{equation}
  \label{e:h-s-estimate-1}
h^{(d-1)/2}\|\indic_{[1+c_1h,1+c_2h]}(h^2\Delta)A\|_{\HS}\leq C\|\sigma(A)\|_{L^2(S^*M)}+\mathcal O(h^{1-2\rho}).
\end{equation}
Here $C$ is a constant independent of $A$ (if $\WFh(A)$ is
contained in a fixed compact set), however the constant
in $\mathcal O(h^{1-2\rho})$ depends on $A$.  We take the $L^2$ norm
of $\sigma(A)$ on the energy surface $S^*M$ with respect to the
Liouville measure $\mu_L$.

Moreover, if $\WFh(A)$ is microsupported, in the sense of
Definition~\ref{d:microlocal-vanishing}, in some $h$-dependent family
of sets $V(h)\subset T^*M$, then
\begin{equation}
  \label{e:h-s-estimate-2}
h^{(d-1)/2}\|\indic_{[1+c_1h,1+c_2h]}(h^2\Delta)A\|_{\HS}\leq C \mu_L(V(h)\cap S^*M)^{1/2}+\mathcal O(h^\infty).
\end{equation}
Here $\mu_L(V(h)\cap S^*M)$ denotes the volume of $V(h)\cap S^*M$ with
respect to the Liouville measure on $S^*M$ and the constant $C$
depends on a certain $S^{\comp}_\rho$-seminorm of the full symbol of
$A$.
\end{lemm}
\begin{proof}
Take a function $\chi\in \mathscr S(\mathbb R)$ such that $\hat\chi$
is compactly supported in some interval $(-T,T)$ and $\chi$ does not
vanish on $[c_1,c_2]$ (for example, take nonzero $\psi\in C_0^\infty(\rr)$ with
$\psi\geq 0$, then $|\hat{\psi}|>0$ in an interval $[c_1\eps,c_2\eps]$; set
$\chi(x):=\hat\psi(\eps x)$). Then
\[
\indic_{[1+c_1h,1+c_2h]}(h^2\Delta)=Z\chi((h^2\Delta-1)/h),\
\]
where $Z$ is a certain function of $h^2\Delta$
and it is bounded on $L^2(M)$ uniformly in $h$. 
It then suffices to estimate the Hilbert--Schmidt norm of
\[
B = h^{(d-1)/2}\chi((h^2\Delta-1)/h)A
  = (2\pi)^{-1}h^{(d-1)/2}\int_{-T}^T\hat\chi(t) e^{-it/h}U(2t)A\,dt. 
\]
Let $A_0\in\Psi^{\comp}_0(M)$ be compactly supported and equal to the
identity microlocally near $\WFh(A)$.  By part~3 of
Proposition~\ref{l:egorov}, for each $t$ we have
$
U(2t)A_0=U_{2t}+R_{2t}
$,
where $U_{2t}\in I^{\comp}(g^{2t})$ and $R_{2t}=\mathcal O(h^{\infty})_{L^2\to L^2}$.
Then
\begin{equation}
  \label{e:hs-approx}
(U(2t)-U_{2t})A=\mathcal O(h^\infty)_{\HS}.
\end{equation}
Indeed, we can write the left-hand side of~\eqref{e:hs-approx}
as the sum of $R_{2t}A$ and $U(2t)(1-A_0)A$; it remains
to note that $R_{2t}=\mathcal O(h^\infty)_{L^2\to L^2}$,
$\|A\|_{\HS}$ is polynomially bounded in $h$,
and $\|(1-A_0)A\|_{\HS}=\mathcal O(h^\infty)$.

By~\eqref{e:hs-approx}, we can replace $U(2t)$ by $U_{2t}$ in the
definition of $B$. Now, $\|B\|_{\HS}$ is equal to
the $L^2(M\times M)$ norm of the Schwartz kernel $K_B$. Using the
local normal form~\eqref{e:fourior-basic}, we can write $K_B$, up to an $\mathcal
O(h^\infty)_{C_0^\infty}$ remainder and an appropriate cutoff in the
$m'$ variable, as a finite sum of expressions of the form (in a fixed
coordinate system on $M$)
\begin{equation}
  \label{e:parametrix}
h^{-(d+1)/2}\int_{-T}^T \int_{\mathbb R^d}
  e^{i(S(m,\nu';2t)-m'\cdot \nu'-t)/h}b(m,\nu',t;h)\,d\nu' dt.
\end{equation}
Here $S(m,\nu';2t)$ is a generating function for $g^{2t}$ and $b$ is a
certain symbol in $S^{\comp}_\rho$. Moreover, $b$ admits an asymptotic
expansion in terms of the full symbol of $A$, by
Proposition~\ref{l:fourior-mul}. The fact that $S$ and $b$ can be
choosen to depend smoothly on $t$ follows from the proof of part~3 of
Proposition~\ref{l:egorov}. By~\cite[Lemma~10.5]{e-z},
$S$ satisfies the Hamilton--Jakobi equation
\begin{equation}
  \label{e:g-t-eq}
g^{2t}(m,\nu)=(m',\nu')\Longrightarrow
\partial_t(S(m,\nu';2t))=p(m,\pl_m S(m,\nu';2t)).
\end{equation}
It follows that
$
\Phi(m,m',\nu',t)=S(m,\nu';2t)-m'\cdot \nu'-t
$
is a nondegenerate phase function (with $m,m'$ as base variables and
$\nu',t$ as the oscillatory variables) and generates the (immersed)
Lagrangian
\[
\Lambda=\{(m,\nu;m',-\nu')\mid p(m,\nu)=1,\ \exists t\in (-T,T):
g^{2t}(m,\nu)=(m',\nu')\}.
\]
Then~\eqref{e:parametrix} lies in $I^{\comp}_\rho(\Lambda)$. By the
local normal form~\eqref{e:lagrangian-key} of a Lagrangian
distribution, we can write~\eqref{e:parametrix}, up to an $\mathcal
O(h^\infty)_{C_0^\infty}$ remainder and an appropriate cutoff in the
$(m,m')$ variables, as the sum of finitely many expressions of the
form
\begin{equation}
  \label{e:parametrix2}
h^{-d}\int_{\mathbb R^{2d}} e^{i(m\cdot\nu+m'\cdot\nu'+F(\nu,\nu'))/h}
\tilde b(\nu,\nu';h)\,d\nu d\nu',
\end{equation}
where $F$ parametrizes some open subset of $\Lambda$ by~\eqref{e:lambda-can}
and $\tilde b$ is a symbol in $S^{\comp}_\rho$. By Proposition~\ref{l:fourior-mul}
and Proposition~\ref{l:lagrangian-basic}, we see that the symbol
$\tilde b$ has the following asymptotic expansion in terms of the full symbol
$a$ of $A$:
\begin{equation}
  \label{e:b-expansion}
\tilde b(\nu,\nu';h)=\sum_{0\leq j<N} h^jL_ja(m',\nu';h)+\mathcal O(h^{N(1-2\rho)})_{S^{\comp}_\rho},
\end{equation}
where each $L_j$ is a differential operator of order $2j$ and $m,m'$
are given by the relation $(m,\nu,m',-\nu')\in \Lambda$; in
particular, $(m',\nu')\in S^*M$.

We now use Proposition~\ref{l:lagrangian-norm} to estimate the $L^2$
norm of~\eqref{e:parametrix2}; as $B$ is, modulo $\mathcal
O(h^\infty)_{\HS}$, a sum of operators with Schwartz kernels of the
form~\eqref{e:parametrix2}, this would give an estimate on the
Hilbert--Schmidt norm of $B$.  For~\eqref{e:h-s-estimate-1}, we can
write $\tilde b(\nu,\nu';h)$ as a multiple of $a(m',\nu')$ plus an
$\mathcal O(h^{1-2\rho})$ remainder and note that $(m',\nu')$ always
lies in $S^*M$.  For~\eqref{e:h-s-estimate-2}, we use that $\tilde
b=\mathcal O(h^\infty)$ outside of the preimage of $V(h)$ under the
map $(\nu,\nu')\mapsto (m',\nu')$, and that $\sup |\tilde b|$ can be
estimated by a certain $S^{\comp}_\rho$-seminorm of $a$.
\end{proof}
%
%
\smallsection{Local traces of integrated Schr\"odinger propagators}
We give the following version of the Schr\"odinger propagator
trace formula when there are no contributions from closed geodesics:
%
%
\begin{lemm}
  \label{l:robert-trace}
Assume that $M$ is a $d$-dimensional complete Riemannian manifold
and $B_s$ is a family of compactly supported
pseudodifferential operators in~$\Psi^{\comp}_\rho(M)$, smooth
and compactly supported in $s\in (-T_0,T_0)$, where $T_0>0$ is fixed.
Assume also that all $B_s$ are microsupported, in the sense of Definition~\ref{d:microlocal-vanishing},
in some $h$-dependent family of bounded sets $V(h)\subset T^*M$, and the following nonreturning
condition holds:
\begin{equation}
  \label{e:robert-trace-condition}
(m,\nu)\in V(h),\ |s|< T_0\Longrightarrow d((m,\nu),g^s(m,\nu))\geq C^{-1}|s|h^\rho.
\end{equation}
Here $C$ is some constant and $d$ denotes some smooth distance function on $T^*M$.
Let $B_s=\Op_h(b(s))+\mathcal O(h^\infty)_{\Psi^{-\infty}}$ for some
family of symbols $b(s,m,\nu)$ and some quantization procedure $\Op_h$.
Then for each $N$ and each $\lambda>0$, we have the trace expansion
\begin{equation}
  \label{e:robert-trace}
\begin{gathered}
(2\pi h)^{d-1}\int_{-T_0}^{T_0}e^{-i\lambda^2s/(2h)}\Tr(U(s)B_s)\,ds
\\=\sum_{0\leq j<N} h^j\int_{S^*M} L_jb(0,m,\lambda\nu)\,d\mu_L(m,\nu)
+\mathcal O(h^{N(1-2\rho)})_{C^\infty_\lambda},
\end{gathered}
\end{equation}
where $\mu_L$ is the Liouville measure and each $L_j$ is a differential operator of order $2j$
on $T^*M_{(m,\nu)}\times (-T_0,T_0)_s$, independent of $B_s$ and smooth in $\lambda$.
In particular, $L_0=\lambda^{d-2}$.
\end{lemm}
\begin{proof}
As in the proof of Lemma~\ref{l:h-s-estimate}, we can reduce to computing
the trace of the operator with the Schwartz kernel (in some fixed local coordinates)
$$
K(m,m')=(2\pi h)^{-1}\int_{-T_0}^{T_0}\int_{\mathbb R^d}
e^{{i\over h}(S(m,\nu';s)-m'\cdot\nu'-\lambda^2s/2)}\tilde b(m,\nu',s;h)\,d\nu'ds,
$$
where $S(m,\nu';s)$ is a local generating function for $g^s$ in the
sense of~\eqref{e:generating-function} and $\tilde b(m,\nu',s;h)$ is a
certain symbol in $S_\rho$ having an asymptotic expansion in terms of
the jet of $b_s$ at the point $(\partial_{\nu'} S(m,\nu';s),\nu')$.  The
trace of the corresponding operator is
$$
\int\limits_M K(m,m)\,dm=(2\pi h)^{-1}\int\limits_{-T_0}^{T_0}
\int\limits_{M\times \mathbb R^d} e^{{i\over h}(S(m,\nu';s)-m\cdot\nu'-\lambda^2s/2)}\tilde b(m,\nu',s;h)\,dmd\nu'ds.
$$
We now use the method of stationary phase. The stationary points of
the phase are solutions to the equations $g^s(m,\nu')=(m,\nu')$ and
$|\nu'|_g=\lambda$; they occur at $s=0$ and may also occur for
$\lambda |s|\geq r_i$, where $r_i>0$ is the injectivity radius of $M$. For
$\lambda |s|\geq r_i/2$, we see by~\eqref{e:robert-trace-condition} that the
expression under the integral can be split into two pieces, on one of
which the symbol is $\mathcal O(h^\infty)$ and on the other, the
differential of the phase function has length at least
$C^{-1}h^{\rho}$; by repeated integration by parts, the latter integral is $\mathcal
O(h^\infty)$.

It remains to evaluate the contribution of the
stationary set $\{s=0\}\cap \lambda S^*M$. The phase function is degenerate on
these points; however, one can pass to polar coordinates
$\nu'=r\omega$, with $|\omega|_g=1$ and $r>0$, and apply the method of
stationary phase in the $(r,s)$ variables, resulting in the
expansion~\eqref{e:robert-trace}. See for example the proofs
of~\cite[Th\'eor\`eme~V-7 and Proposition~V-8]{Ro}
or~\cite[Lemma~3.1]{RoTa} for details of the computation.
\end{proof}
%
%

\section{General assumptions}
  \label{s:general}

In this section, we list general geometric assumptions on the
manifold $M$ and analytic assumptions on its Laplacian required for
our results to hold. As noted in the introduction, they are satisfied in
particular if outside of a compact set, $M$ is isometric to either the
Euclidean space (studied in Section~\ref{s:euclidean}) or an
asymptotically hyperbolic space of constant curvature (studied in
Section~\ref{s:ah}).  We also derive some direct consequences of the
general assumptions, including averaged estimates on plane waves and
the existence of limiting measures $\mu_\xi$.

\subsection{Geometric assumptions}
  \label{s:general.geometry}

In this subsection, we specify the geometry of the manifold $M$ at infinity.

Let us introduce some notation and terminology first. On a complete Riemannian manifold 
$(M,g)$ we denote by $g^t$ the geodesic flow of the metric $g$, considered as
a map on the cotangent bundle $T^*M$. Any smooth function $f$ on $M$
can be lifted to a function on $T^*M$; denote by $\dot f,\ddot f\in
C^\infty(T^*M)$ the derivatives of $f$ with respect to the geodesic
flow:
$$
\dot f(m,\nu):=d_tf(g^t(m,\nu))|_{t=0},\
\ddot f(m,\nu):=d_t^2f(g^t(m,\nu))|_{t=0}.
$$
We denote by $S^*M$ the unit cotangent bundle $\{(m,\nu)\mid
|\nu|_g=1\}\subset T^*M$.

A \emph{boundary defining function} on a smooth compact manifold $\bbar{M}$ with boundary is
a smooth function $x:\overline M\to [0,\infty)$ such that $x>0$ on $M$
and $x$ vanishes to first order on $\plM$.

We make the following assumptions:
%
%
\begin{enumerate}
\asItem{G1}
$(M,g)$ is a complete Riemannian manifold of dimension
$d=n+1$. Moreover, there exists a \emph{compactification} of $M$,
namely a compact manifold with boundary $\overline M$ such that $M$ is
diffeomorphic to the interior of $\overline M$.
The boundary $\plM$ is called the \emph{boundary at infinity};

\asItem{G2}
There exists a boundary defining function $x$ on $M$ and a
constant $\varepsilon_0>0$ such that for any point $(m,\nu)\in S^*M$,
\begin{equation}
  \label{e:kinda-convex}
\text{if }x(m,\nu)\leq\varepsilon_0\text{ and }
\dot x(m,\nu)=0,\text{ then }
\ddot x(m,\nu)<0;
\end{equation}

\asItem{G3} For each $(m,\nu)\in S^*M$ such that $x(m)\leq\eps_0$ and $\dot{x}(m,\nu)\leq 0$,
the geodesic $g^t(m,\nu)$ (projected onto the base space $M$)
converges as $t\to+\infty$, in the topology of $\overline M$, to some
point $\xi_{+\infty}(m,\nu)\in \plM$. The function
$\xi_{+\infty}$ depends smoothly on $(m,\nu)$, and we extend it
naturally (as the limit of the corresponding geodesic) to a smooth
function on $S^*M\setminus\Gamma_-$, with $\Gamma_-$ given in
Definition~\ref{d:trapped} below;

  \asItem{G4}
There exists an open set $U_\infty\subset
M\times\plM$ such that 
$\bbar{U}_\infty$ contains a neighbourhood of $\{(\xi,\xi)\in\bbar{M}\times \plM\mid\xi\in\plM\}$
and a smooth real-valued function
$\phi(m,\xi)=\phi_\xi(m)$ on $U_\infty$ such that
$|\partial_m\phi_\xi(m)|_g=1$ everywhere and the function
\begin{equation}
  \label{e:tau}
\tau(m,\xi):=(m,\partial_m\phi_\xi(m))\in S^*M,\
(m,\xi)\in U_\infty,
\end{equation}
is a diffeomorphism from $U^+_\infty$ onto $V^+_\infty$ with inverse given by
$$
\tau^{-1}(m,\nu)=(m,\xi_{+\infty}(m,\nu)),\ (m,\nu)\in V^+_\infty
$$
where the sets $U^+_\infty$ and $V^+_\infty$ are defined by
$$
\begin{gathered}
U^+_\infty:=\{(m,\xi)\in U_\infty\mid x(m)\leq \eps_0,\,\, \dot{x}(\tau(m,\xi))\leq 0\},\\
V^+_\infty:=\{(m,\nu)\in S^*M\mid x(m)\leq \eps_0,\, \dot{x}(m,\nu)\leq 0,\,\, (m,\xi_{+\infty}(m,\nu))\in U_\infty\};
\end{gathered}
$$

  \asItem{G5} if $(m,\nu)\in V^+_\infty$, then $g^t(m,\nu)\in V^+_\infty$ for
all $t\geq 0$;

  \asItem{G6} if $\xi\in\plM$ and $m,m'\in M$ are such that $(m,\xi),(m',\xi)\in
U_\infty^+$, then $\partial_\xi \phi_\xi(m)=\partial_\xi\phi_\xi(m')$
if and only if $\tau(m,\xi)$ and $\tau(m',\xi)$ lie on the same
geodesic. Moreover, the matrix $\partial_m \partial_\xi\phi_\xi(m)$
has rank $n$. 
\end{enumerate}
%
%

%
%
\begin{figure}
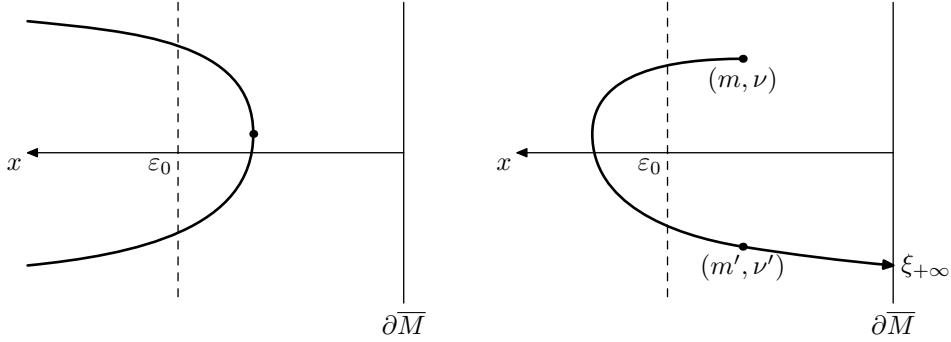

\includegraphics{qeefun.9}
\qquad
\includegraphics{qeefun.10}
\caption{Illustrations for (G2) and (G3). Left: (G2) is not satisfied. Right: (G2) is satisfied. 
The point $(m,\nu)$ does not escape directly in the forward direction, 
but the point $(m',\nu')$ does, illustrating (G3).}
\label{f:g2-g3}
\end{figure}
%
%

\smallsection{Escaping trajectories and the trapped set}.
We now define the incoming/outgoing tails $\Gamma_\pm$ and the trapped set $K$:
%
%
\begin{defi}
  \label{d:trapped}
Let $\gamma(t)$ be a unit speed geodesic. We say that it
\emph{escapes} in the forward, respectively backward, direction, if
$\gamma(t)$ goes to infinity in $M$ as $t\to +\infty$, respectively
$t\to -\infty$. If $\gamma(t)$ does not escape in some direction, we
call it \emph{trapped} in this direction. Denote by $\Gamma_+\subset
S^*M$ the union of all geodesics trapped in the backward direction, by
$\Gamma_-$ the union of all geodesics trapped in the forward
direction, and put $K=\Gamma_+\cap\Gamma_-$; we call $K$ the
\emph{trapped set}.
\end{defi}
%
%
An escaping geodesic could potentially spend a long time in the
compact part of the manifold. It is helpful to consider
geodesics that escape in a straightforward way (with the boundary
defining function $x$ decreasing along them); they appeared in assumption~\as{G3} for instance. 
%
%
\begin{defi}
  \label{d:directly-escape}
We say that $(m,\nu)\in S^*M$ \emph{directly escapes} in the forward,
respectively backward, direction, if $x(m)\leq\varepsilon_0$ and $\dot
x(m,\nu)\leq 0$, respectively $\dot x(m,\nu)\geq 0$. Here
$\varepsilon_0$ is the constant from~\as{G2}.  Denote by
$\mathcal{DE}_+$, respectively $\mathcal{DE}_-$, the set of all points
directly escaping in the forward, respectively backward, direction.
\end{defi}
%
%
One can verify that $\Gamma_\pm$ are closed sets and the trapped
set $K$ is compact (see~\cite[Appendix]{ge-sj}); in fact, since
$S^*M\cap\{x\leq\varepsilon_0\}\subset \mathcal{DE}_+\cup
\mathcal{DE}_-$, we have $K\subset\{x>\varepsilon_0\}$.

For the example of $M=\mathbb R^{n+1}$ discussed below, we have $\Gamma_\pm =\emptyset$. The point
$(m,\nu)$ lies in $\mathcal{DE}_+$ if and only if $x(m)\leq\varepsilon_0$
and $m\cdot\nu\geq 0$.

\smallsection{Comments on the geometric assumptions}

A basic example to have in mind for a manifold satisfying our assumptions is
$M=\mathbb R^{n+1}$ with the radial compactification $\overline M$
being a closed ball and the boundary at infinity $\plM$ equal to the
sphere $\mathbb S^n$. We will often use this example to illustrate the
somewhat abstract assumptions of this section. (A more general version
will be considered in Section~\ref{s:euclidean}.)

An important corollary of the assumption~\as{G2} is that for
$\varepsilon\leq\varepsilon_0$, the compact set
$\{x\geq\varepsilon\}\subset M$ is \emph{geodesically convex} in the sense of~\eqref{geodconvex}.
For the example of $M=\mathbb R^{n+1}$, we can take
$x=(1+|m|^{-2})^{-1/2}$, where $|m|$ is the Euclidean length of $m\in
\mathbb R^d$; the corresponding sets $\{x\geq\varepsilon\}$ are balls
centered at zero.

It also follows from~\as{G2} that for $(m,\nu)\in \mathcal{DE}_+$, the
function~$x(g^t(m,\nu))$ is decreasing for $t\geq 0$. One can show
that $x(g^t(m,\nu))\to 0$ as $t\to +\infty$ and thus $g^t(m,\nu)$
escapes in the forward direction; we do not give a proof of this fact
as it follows from the more restrictive assumption~\as{G3}.  Also, if
a geodesic $\gamma(t)$ escapes in the forward direction, then for $t$
large enough we have $\gamma(t)\in \mathcal{DE}_+$. For $M=\rr^{n+1}$,
we have $\xi_{+\infty}(m,\nu)=\nu\in \mathbb S^n$.

Assumption \as{G4} means that for $m$ sufficiently close to the infinity,
the covectors $\nu$ such that $(m,\nu)\in \mathcal{DE}_+$ are in
one-to-one correspondence with the limit points
$\xi_{+\infty}(m,\nu)$, and the inverse correspondence can
be described by a phase function.
It follows in particular from~\as{G4} that for a fixed $\xi\in \plM$,
the set of directly escaping points $(m,\nu)$ such that
$\xi_{+\infty}(m,\nu)=\xi$ and $(m,\xi)\in U_\infty$ is the
intersection of $\mathcal{DE}_+$ with the Lagrangian
\begin{equation}
  \label{e:lambda-xi}
\Lambda_\xi:=\{(m,\partial_m\phi_\xi(m))\mid (m,\xi)\in U_\infty\}.
\end{equation}
In the model case $M=\mathbb R^{n+1}$ we can put for any $R>0$,
$U_\infty=\{(m,\xi)\mid |m|> R\}$, and $\phi_\xi(m)=m\cdot\xi$, so
that $\tau$ is the canonical map from $\mathbb R^{n+1}\times \mathbb S^n$
to $S^* \mathbb R^{n+1}$. Then $U^+_\infty=\{(m,\xi)\mid |m|>R,\
m\cdot\xi\geq 0\}$ and $V^+_\infty=\{(m,\nu)\mid |m|\geq R,\ m\cdot \nu\geq 0\}$; the
difference is that $U^+_\infty$ is considered as a subset of $\mathbb
R^{n+1}\times \mathbb S^n$, while $V^+_\infty$ is considered as a
subset of $S^* \mathbb R^{n+1}$.

The condition \as{G5} can also be viewed as a condition on $U_\infty^+$, saying that for any 
$(m,\xi)\in U_\infty^+$, the whole geodesic `segment' $\gamma(m,\xi)$ relating $m$ and the point $\xi\in \pl\bbar{M}$ is such that $\gamma(m,\xi)\x \{\xi\}$ is contained in $U_\infty^+$.

The condition~\as{G6} is required in Proposition~\ref{l:trace-1}. To explain
it, note that under the assumption~\as{G4}, if $(m,\xi)\in U^+_\infty$ and
$(m(t),\nu(t))=g^t(\tau(m,\xi))$, then
\begin{equation}
  \label{e:g6-calculation}
\partial_t\phi_\xi(m(t))|_{t=0}=
\partial_m\phi_\xi(m)\cdot \partial_t m(t)|_{t=0}
=g(\partial_m\phi_\xi(m),\partial_m\phi_\xi(m))=1.
\end{equation}
Therefore, $\partial_\xi\phi_\xi(m)$ is constant on the geodesic passing
through $\tau(m,\xi)$.

\subsection{Analytic assumptions}
  \label{s:general.analytic}

In this subsection, we formulate the analytic assumptions on plane
waves. We will prove that they are satisfied in the Euclidean near infinity setting (in Section \ref{s:euclidean.analysis})
and in the hyperbolic near infinity setting (in Section \ref{s:eisenstein.analytic}).
Let $M$ be as in the previous subsection, $\Delta$ be the
(nonnegative definite) Laplace--Beltrami operator on $M$, and $h>0$ be
the semiclassical parameter.  We make the following assumptions:
%
%
\begin{enumerate}
\asItem{A1}
There exists $c_0\geq 0$ (equal to 0 for the Euclidean and to $n^2/4$
for the hyperbolic case), such that for each $\la>0$, $h>0$
and $\xi\in \plM$, there exists a function, called
\emph{distorted plane wave}, $E_h(\la,\xi;m)=E(\la/h,\xi;m)$, smooth in all variables and
solving on $M$ the differential equation~\eqref{e:e-eq-h} in $m$:
$$
(h^2\Delta-c_0h^2-\la^2)E_h(\la,\xi;\cdot)=0.
$$
Here $\xi$ gives the direction of the plane wave, 
while $\la$ corresponds to its semiclassical energy;
\asItem{A2}
for each $0< \la_1\leq \la_2$, the Schwartz kernel of the semiclassical spectral
projector
$$
\Pi_{[\lambda_1,\lambda_2]}:=\indic_{[\lambda_1^2+c_0h^2,\lambda_2^2+c_0h^2]}(h^2\Delta)
$$
can be written in the form
\begin{equation}
  \label{e:spectrum-eis-2}
\Pi_{[\lambda_1,\lambda_2]}(m,m')=(2\pi h)^{-n-1}\int_{\lambda_1}^{\lambda_2}
\lambda^n f_\Pi(\lambda/h) \int_{\plM}E_h(\lambda,\xi;m)\overline{E_h(\lambda,\xi;m')}\,d\xi d\lambda.
\end{equation} 
Here integration in $\xi$ is carried with respect to a certain given
volume form $d\xi$ on $\plM$ and $f_\Pi(z)>0$ is a
smooth function of $z$ such that $|\partial_z^kf_\Pi(z)|\leq
C_k\langle z\rangle^{-k}$ for each $k$ and $f_\Pi(z)\to 1$ as
$z\to\infty$.
\end{enumerate}
%
%
We now assume that plane waves admit the decomposition
\begin{equation}
  \label{e:e-h-decomposition}
E_h(\lambda,\xi;m)=\chi_0(m;\xi)E^0_h(\lambda,\xi;m)+E^1_h(\lambda,\xi;m),
\end{equation}
where $\chi_0,E^0_h,E_h^1$ are respectively a cutoff function, an explicit
`outgoing' part of the wave, and the `incoming' part, satisfying more precisely the following properties: 
\begin{enumerate}
\asItem{A3}
$\chi_0(m;\xi)$ is a function smooth in $m\in\overline M$ and $\xi\in \plM$,
supported inside the set $U_\infty$ from~\as{G4}
and $\chi_0(m,\xi)=1$ for $m$ sufficiently close to $\xi$;

\asItem{A4}
$E^0_h(\lambda,\xi;m)$ is a smooth function of $\lambda/h \in \mathbb R^*$
and $(m,\xi)\in U_\infty$, of the form
\begin{equation}
  \label{e:e0-h}
E^0_h(\lambda,\xi;m)=e^{{i\lambda\over h}\phi_\xi(m)}b^0(\lambda,\xi,m;h),
\end{equation}
where $U_\infty$ and $\phi_\xi$ are defined in~\as{G4} and  $b^0(\la,\xi,m;h)=b^0(\la/h,\xi,m)$ 
is a classical symbol in $h$ for $(m,\xi)\in U_\infty$, $\la\in\rr^*$;
that is, $b^0(\la,\xi,m;h)$ is smooth in all variables, including $h$, up to $h=0$.
The limit $b^0(\la,\xi,m;0)=\lim_{h\to 0}b^0(\la/h,\xi,m)$ for $\la>0$ is independent of $\la$;
\asItem{A5}
for $\lambda$ in a fixed compact subset of $(0,\infty)$ and $\varepsilon_0$ defined in~\as{G2}, the function
\begin{equation}
  \label{e:e-1-h-tilde}
\widetilde E^1_h(\lambda,\xi;m):={E^1_h(\lambda,\xi;m)\over 1+\|E_h(\lambda,\xi;m)\|_{L^2(\{x\geq\varepsilon_0\})}}
\end{equation}
is $h$-tempered in the sense of~\eqref{tempered};

\asItem{A6}
for $\lambda$ in a fixed compact subset of $(0,\infty)$,
each $\xi\in \plM$, and each
$(m,\lambda\nu)\in\WFh(\widetilde E^1_h(\lambda,\xi))$, we have $(m,\nu)\in S^*M$
and either the geodesic $\gamma(t)=g^t(m,\nu)$
does not escape in the forward direction (i.e. $(m,\nu)\in\Gamma_-$)
or there exists $t\geq 0$ such that $\gamma(t)$ lies in the set
\begin{equation}
  \label{e:w-xi}
W_\xi:=\{(m,\partial_m\phi_\xi(m))\mid m\in \supp(\partial_m \chi_0)\}.
\end{equation}
The constants in the corresponding estimates (in the definition of the
wave front set of a distribution given in
Section~\ref{s:prelim.notation}) are uniform in $\lambda$ and $\xi$;

\asItem{A7}
there exists $\varepsilon_1\in (0,\varepsilon_0)$ such that
for $(m,\nu)\in S^*M$ directly escaping in the forward direction
and $x(m)\leq\varepsilon_1$, the point
$(m,\xi_{+\infty}(m,\nu))$ lies in the set~$U_\infty$ defined in~\as{G4}
and  $\chi_0=1$ near this point;

\asItem{A8}
Let $\tau:U^+_\infty\to V^+_\infty$ be the diffeomorphism from~\as{G4}.
Then its Jacobian with respect to the volume measure
$\Vol(m)d\xi$ on $U^+_\infty$ and the Liouville measure on $V^+_\infty$,
is equal to $|b^0(1,\xi,m;0)|^2$, with $b^0$ defined in~\as{A4}.

\end{enumerate}
%
%

For example, for $M=\mathbb R^{n+1}$ we put $c_0=0$,
$E_h(\la,\xi;m)=e^{i\la\xi\cdot m/h}$ and use the standard volume
form on the sphere $\plM=\mathbb S^{n}$.  The
equation~\eqref{e:spectrum-eis-2} then follows from the Fourier inversion
formula.

Let us informally explain how the
decomposition~\eqref{e:e-h-decomposition} is constructed and provide a
justification for assumptions~\as{A3}--\as{A6}, putting for simplicity $\lambda=1$.
First of all, \as{A4} implies
that  for any $\chi\in
C_0^\infty(M)$, $\chi\chi_0 E^0_h$, as a function of $m$, is a
Lagrangian distribution associated to the Lagrangian $\Lambda_\xi$
from~\eqref{e:lambda-xi}. In fact, in the cases considered in the present paper,
$E^0_h$ solves on its domain the equation~\eqref{e:e-eq-h};
however, we do not make this assumption here, as in more complicated
cases (such as asymptotically hyperbolic manifolds of variable
curvature) $E^0_h$ might only be an approximate solution
to~\eqref{e:e-eq-h} in a certain sense.

If we assume that $E^0_h$ solves~\eqref{e:e-eq-h} on its domain, then
the function
$$
F_h(\lambda,\xi;m)=(h^2\Delta-\lambda^2-c_0h^2)(\chi_0(m)E^0_h(\lambda,\xi;m))
$$
is equal to $[h^2\Delta,\chi_0]E^0_h$. Since $E^0_h$ is a Lagrangian
distribution associated to $\Lambda_\xi$, the wavefront set of $F_h$
is contained in $W_\xi$.  We will now take $E^1_h=-R_h(\lambda)F_h$,
where $R_h(\lambda)$ is the \emph{incoming scattering resolvent}, a
certain right inverse of $h^2\Delta-\lambda^2-c_0h^2$.  Moreover, in
our cases $R_h(\lambda)$ will be microlocally incoming in the weak
sense: if we multiply $F_h$ by a (possibly small)
constant to make $R_h(\lambda)F_h$ bounded polynomially in $h$, then
each point in the wavefront set of $R_h(\lambda)F_h$, when propagated
forward by the geodesic flow, will either converge to the trapped set
or pass through $\WFh(F_h)$. Thus, the assumption~\as{A6} should be
viewed as a direct consequence of the fact that the scattering
resolvent is microlocally incoming and of propagation of
singularities.

The assumption~\as{A7} looks less natural, but will play an essential
role in our proofs, in Propositions~\ref{l:geometry-1}
and~\ref{l:geometry-2}. It holds for both Euclidean and hyperbolic
infinities, but for different reasons. For the hyperbolic infinity,
$\chi_0(\cdot;\xi)$ is equal to $1$ in a small neighborhood of $\xi$
and one can see that for $(m,\nu)$ directly escaping in the forward
direction and converging to $\xi$, the distance from $m$ to $\xi$ in
$\overline M$ is $\mathcal O(x(m))$.  This is not true in the
Euclidean case; however, in that case $\chi_0$ is equal to $1$ outside
of a compact subset of $M$ (that is, near the whole boundary
$\plM$, not just near $\xi$).

The assumption~\as{A8} is required to relate the natural measure
arising from the function $E_h^0$ to the Liouville measure. If $E_h$
were equal to $E^0_h$, then this assumption would simply follow
by taking the trace in~\eqref{e:spectrum-eis-2} with a compactly supported
pseudodifferential operator and a smooth cutoff function
in $\lambda$.

\subsection{Limiting measures}
  \label{s:general.limiting}

We now define the family of limiting measures $\mu_\xi$. These
measures result from propagating the natural measure arising from the
`outgoing' part $E^0_h$ of the plane wave, which is supported on the
Lagrangian $\Lambda_\xi$ from~\eqref{e:lambda-xi}, backwards along the geodesic
flow. In contrast with~\cite{eiscusp}, where the exponential decay of
the measure along the flow ensured its convergence, our measures will
only be defined for almost every $\xi$.

We first define
the measure $\tilde\mu_\xi$ on $S^*M$, corresponding to~$E^0_h$,
as follows: for each compactly supported continuous function $a$ on $S^*M$, put
\begin{equation}
  \label{e:mu-xi-tilde}
\int_{S^*M} a\,d\tilde\mu_\xi=\int_{(m,\xi)\in U_\infty^+}|b^0(1,\xi,m;0)|^2 a(\tau(m,\xi))\,\Vol(m).
\end{equation}
The support of $\tilde\mu_\xi$ is contained in the Lagrangian $\Lambda_\xi$
from~\eqref{e:lambda-xi} and the integral~\eqref{e:mu-xi-tilde}
depends continuously on $\xi$. 
We see from~\as{A8} that for any continuous function $f$ on $\plM$,
\begin{equation}
  \label{e:d2-cor}
\int_{\plM} f(\xi)\int_{S^*M} a(m,\nu)\,d\tilde\mu_\xi(m,\nu) d\xi=\int_{V_\infty^+} f(\xi_{+\infty}(m,\nu))
a(m,\nu)\,d\mu_L(m,\nu). 
\end{equation}
We now want to define the measure $\mu_\xi$ by
\begin{equation}
  \label{e:mu-xi-def}
\int_{S^*M} a\,d\mu_\xi=\lim_{t\to +\infty}\int_{S^*M} a\circ g^{-t}\,d\tilde\mu_\xi,
\end{equation}
valid for all compactly supported continuous functions $a$. 
To show that the limit exists for almost every $\xi$ (chosen independently of $a$)
and for every $a$, we will use monotonicity. By~\eqref{e:d2-cor}, \as{G5},
and using the invariance of the function $\xi_{+\infty}$ and the Liouville measure $\mu_L$ under
the geodesic flow, we see that if $a$ and $f$ are nonnegative, then
$$
\int_{\plM} f(\xi)\int_{S^*M}(a\circ g^{-t})\,d\tilde\mu_\xi d\xi
=\int_{g^{-t}(V^+_\infty)} f(\xi_{+\infty}(m,\nu))a(m,\nu)\,d\mu_L(m,\nu)
$$ 
is increasing with $t$. Therefore, for each $\xi$ the integral
$$
I_{a,t}(\xi)=\int_{S^*M} (a\circ g^{-t})\,d\tilde\mu_\xi
$$
is increasing in $t$ for any nonnegative $a$. Moreover, the integral
of $I_{a,t}(\xi)$ in $\xi$ is bounded by a $t$-independent
constant, namely by the integral of $a$ by the
Liouville measure. 
Taking $a$ to be an approximation of the characteristic function
of each member of a countable family of compact sets exhausting $S^*M$,
and using the monotone convergence theorem, we see that there exists
a measure zero set $\mathcal X\subset \plM$ such
that for $\xi\not\in \mathcal X$, we have
for any compactly supported continuous function $a$,
$$
\lim_{t\to+\infty}\int_{S^*M} (a\circ g^{-t})\,d\tilde\mu_\xi<\infty.
$$
This limit is a continuous functional on the space of continuous
compactly supported functions on $S^*M$; therefore, there exists
unique Borel measure $\mu_\xi$ such that~\eqref{e:mu-xi-def} holds.
Moreover, we see that the limit~\eqref{e:mu-xi-def} is uniform in $a$,
as soon as we fix a compact set containing $\supp a$ and impose a
bound on $\sup_{S^*M}|a|$. One also sees immediately~\eqref{e:supp-mu-xi},
namely that for compactly supported continuous $a$,
$$
\supp a\cap\overline{\xi_{+\infty}^{-1}(\xi)}=\emptyset\Longrightarrow
\int_{S^*M}a\,d\mu_\xi=0,
$$
as $\int_{S^*M}(a\circ g^{-t})\,d\tilde\mu_\xi=0$ for all $t$.

We can integrate the measure $\mu_\xi$ in $\xi$, getting back
the Liouville measure:
%
%
\begin{prop}
  \label{l:mu-xi-0}
For each $f\in C^\infty(\plM)$
and each $a\in C_0^\infty(S^*M)$ we have
\begin{equation}
  \label{e:disintegration-2}
\int_{\plM}f(\xi)\int_{S^*M} a(m,\nu)\,d\mu_\xi(m,\nu)d\xi
=\int_{S^*M\setminus\Gamma_-} f(\xi_{+\infty}(m,\nu))a(m,\nu)\,d\mu_L(m,\nu).
\end{equation}
In particular, if $\mu_L(\Gamma_-)=0$ (which will always be the case
in our theorems, see~\eqref{e:mes0}), then $\int \mu_\xi\,d\xi$ is the Liouville measure.
\end{prop}
\begin{proof}
The left-hand side can be written as
$$
\lim_{t\to +\infty} \int_{g^{-t}(V^+_\infty)} f(\xi_{+\infty}(m,\nu))a(m,\nu)\,d\mu_L(m,\nu).
$$
It remains to use the dominated convergence theorem; indeed, the function
under the integral is bounded and compactly supported, we have
$g^{-t_1}(V^+_\infty)\subset g^{-t_2}(V^+_\infty)$ for $t_1<t_2$, and
the union of $g^{-t}(V^+_\infty)$ over all $t\in \mathbb R$ is exactly $S^*M\setminus\Gamma_-$,
as for every geodesic $\gamma(t)$ escaping in the forward direction and for $t$
large enough, the point $\gamma(t)$ is directly escaping in the forward direction
and  $(\gamma(t),\xi_{+\infty}(\gamma(t)))\in U_\infty$.
\end{proof}
%
%
Finally, the following lemma will be useful to relate our measure $\mu_\xi$
to the one obtained from $E^0_h$ in the proofs of Theorems~\ref{t:convergence}
and~\ref{t:remainder}:
%
%
\begin{lemm}
  \label{l:mu-xi-1}
Let $\xi\not\in \mathcal X$, so that $\mu_\xi$ is well-defined. Let
$a$ be a compactly supported continuous function on $S^*M$.

1. $\mu_\xi$ is invariant under the geodesic flow: for each $t\in \mathbb R$,
\begin{equation}
  \label{e:mu-xi-invariance}
\int_{S^*M} a\circ g^t\,d\mu_\xi=\int_{S^*M} a\,d\mu_\xi.
\end{equation}

2. If $\supp a\subset \mathcal{DE}_+\cap \{x\leq \varepsilon_1\}$,
where $\mathcal{DE}_+$ is given by Definition~\ref{d:directly-escape}
and $\varepsilon_1$ is defined in~\as{A7}, then
\begin{equation}
  \label{e:mu-xi-reduction}
\int_{(m,\xi)\in U_\infty} |b^0(1,\xi,m;0)\chi_0(m;\xi)|^2
a(m,\partial_m\phi_\xi(m))\,\Vol(m)=\int_{S^*M} a\,d\mu_\xi.
\end{equation}
\end{lemm}
\begin{proof}
1. Follows immediately from the definition~\eqref{e:mu-xi-def}.

2. First of all, note that for $m$ in the support of the function
$a(m,\partial_m\phi_\xi(m))$, we have $(m,\xi)\in U^+_\infty$
and $\chi_0(m;\xi)=1$ by~\as{A7}; therefore, the left-hand side of~\eqref{e:mu-xi-reduction}
becomes the integral of $a$ over the measure $\tilde\mu_\xi$ defined
in~\eqref{e:mu-xi-tilde}. By~\eqref{e:mu-xi-def}, it is enough to show
that for $t\geq 0$,
$$
\int_{S^*M} a\circ g^{-t}\,d\tilde\mu_\xi=\int_{S^*M} a\,d\tilde\mu_\xi.
$$
For that, it is enough to show that for each $f\in C_0^\infty(\plM)$,
$$
\int_{\pl M}f(\xi)\int_{S^*M} a\circ g^{-t}\,d\tilde\mu_\xi
=\int_{\pl M}\int_{S^*M} a\,d\tilde\mu_\xi.
$$
Using~\eqref{e:d2-cor}, we rewrite this as
$$
\int_{g^{-t}(V^+_\infty)} f(\xi_{+\infty}) a\,d\mu_L
=\int_{V^+_\infty} f(\xi_{+\infty}) a\,d\mu_L.
$$
This is true as $\supp a\subset V^+_\infty\subset g^{-t}(V^+_\infty)$.
\end{proof}
%
%

\subsection{Averaged estimates on plane waves}
  \label{s:general.averaged}

One of the principal tools of the present paper are microlocal
estimates on the plane waves $E_h(\lambda,\xi)$ \emph{on average} in
$\lambda,\xi$, where $\lambda$ takes values in a size $h$
interval. They are direct consequences of~\eqref{e:spectrum-eis-2} and
the Hilbert--Schmidt norm estimate~\eqref{e:h-s-estimate-1}.  More
precisely, restricting to the case $\lambda=1+\mathcal O(h)$ for
simplicity, we have the following
%
%
\begin{prop}
  \label{l:avg-bdd}
Let $\chi\in C_0^\infty(M)$. Then:

1. $\chi\Pi_{[1,1+h]}$ is a Hilbert--Schmidt operator and there exists
a global constant $C$ such that for each bounded operator $A:L^2(M)\to
L^2(M)$, we have
\begin{equation}
  \label{e:hs-est}
h^{-1}\|A\chi(m) E_h(\lambda,\xi;m)\|_{L^2_{m,\xi,\lambda}(M\times \plM\times [1,1+h])}^2
\leq C h^{n}\|A\chi\Pi_{[1,1+h]}\|_{\HS}^2.
\end{equation}

2. The functions $\chi E_h$ are bounded in $L^2$
on average in the following sense: there exists a constant $C(\chi)$ such that
for any $h$,
\begin{equation}
  \label{e:avg-bdd}
h^{-1}\|\chi(m) E_h(\lambda,\xi;m)\|^2_{L^2_{m,\xi,\lambda}(M\times\plM\times [1,1+h])}\leq C(\chi).
\end{equation}

The $h^{-1}$ prefactor in both cases is due to the fact that we are integrating
over an interval of size $h$ in $\lambda$.
\end{prop}
\begin{proof}
1. It follows immediately from~\eqref{e:spectrum-eis-2} that 
\begin{equation}
  \label{e:spectrum-eis-3}
h^{-1}\int_1^{1+h} f_\Pi(\lambda/h)\lambda^n \int_{\plM}
(\chi E_h(\lambda,\xi))\otimes (\chi E_h(\lambda,\xi))\,d\xi d\lambda
=(2\pi)^{n+1}h^n\chi\Pi_{[1,1+h]}\bar\chi.
\end{equation}
Here $\otimes$ denotes the Hilbert tensor product, defined in~\eqref{e:tensor-product}.
The integral on the left-hand side converges in the trace class norm, as the
Schwartz kernels of the integrated operators are smooth and compactly supported.
Therefore, $\chi\Pi_{[1,1+h]}\bar\chi$ is trace class. Since
$$
\chi\Pi_{[1,1+h]}\bar\chi=(\chi\Pi_{[1,1+h]})(\chi\Pi_{[1,1+h]})^*,
$$
we see that $\chi\Pi_{[1,1+h]}$ is a Hilbert--Schmidt operator. Now,
multiplying both sides of~\eqref{e:spectrum-eis-3} by $A$ on the left
and $A^*$ on the right and taking the trace, we get
\begin{equation}
  \label{e:spectrum-eis-A}
\begin{gathered}
h^{-1}\|\lambda^{n/2}f_\Pi(\lambda/h)^{1/2}A\chi(m)E_h(\lambda,\xi;m)\|_{L^2_{m,\xi,\lambda}(M\times \plM\times [1,1+h])}^2\\
=(2\pi)^{n+1}h^n\Tr((A\chi\Pi_{[1,1+h]})(A\chi\Pi_{[1,1+h]})^*)\\
=(2\pi)^{n+1}h^n\|A\chi\Pi_{[1,1+h]}\|_{\HS}^2.
\end{gathered}
\end{equation}

2. We would like to use Lemma~\ref{l:h-s-estimate} to estimate
$\|\chi \Pi_{[1,1+h]}\|_{\HS}$ (we can put $\chi$ on the other side of
the projector in~\eqref{e:h-s-estimate-1} by taking the adjoint),
however this is not directly possible as $\chi$ is not compactly
microlocalized. We thus use that $E_h$ solve the
equation~\eqref{e:e-eq-h}, writing by the elliptic parametrix
construction (same as for the proof of Proposition~\ref{l:elliptic})
\begin{equation}
  \label{e:bdd-internal}
\chi=B+Q_\lambda(h^2\Delta-\lambda^2-c_0h^2)+R_\lambda
\end{equation}
for $\lambda\in [1,1+h]$, where $B\in\Psi^{\comp}(M)$,
$Q_\lambda\in\Psi^{-2}(M)$, and $R_\lambda\in
h^\infty\Psi^{-\infty}(M)$ are compactly supported and $B$ is
independent of $\lambda$ and equal to $\chi$ microlocally near $S^*M$.
We can also assume that $Q_\lambda$ and $R_\lambda$ are smooth in
$\lambda$. Now, we substitute~\eqref{e:bdd-internal} into the
left-hand side of~\eqref{e:avg-bdd} and use the triangle inequality.
By~\eqref{e:hs-est}, the term featuring $B$ is bounded by a constant
times $h^n\|B\Pi_{[1,1+h]}\|_{\HS}^2$, which is bounded uniformly in
$h$ by Lemma~\ref{l:h-s-estimate}. The term featuring
$Q_\lambda$ is zero by~\eqref{e:e-eq-h}.

Finally, we show that the term featuring $R_\lambda$ is $\mathcal
O(h^\infty)$. This does not follow immediately from~\eqref{e:hs-est},
as the operator $R_\lambda$ depends on $\lambda$.  We use the
following variant of~\eqref{e:spectrum-eis-A}: for $\tilde\lambda\in [1,1+h]$,
$$
h^{-1}\|\lambda^{n/2}f_\Pi(\lambda/h)^{1/2}
R_{\tilde\lambda}E_h(\lambda)\|^2_{L^2_{m,\xi,\lambda}(M\times \plM\times [1,\tilde\lambda])}
=(2\pi)^{n+1}h^n\|R_{\tilde\lambda}\Pi_{[1,\tilde\lambda]}\|_{\HS}^2.
$$
Differentiating in $\tilde\lambda$, we get
$$
\begin{gathered}
(2\pi)^{n+1}h^n\partial_{\tilde\lambda}\|R_{\tilde\lambda}\Pi_{[1,\tilde\lambda]}\|_{\HS}^2=
h^{-1}\| \tilde\lambda^{n/2}f_\Pi(\tilde\lambda/h)^{1/2}R_{\tilde\lambda}E_h(\tilde\lambda)\|^2_{L^2(m,\xi)
(M\times\pl M)}+\\
2h^{-1}\Real\langle \lambda^{n/2}f_\Pi(\lambda/h)^{1/2}(\partial_{\tilde\lambda} R_{\tilde\lambda})E_h(\lambda),
\lambda^{n/2}f_\Pi(\lambda/h)^{1/2}R_{\tilde\lambda}E_h(\lambda)\rangle_{L^2_{m,\xi,\lambda}
(M\times\pl M\times [1,\tilde\lambda])}.
\end{gathered}
$$
We now integrate in $\tilde\lambda$ from $1$ to $1+h$. The integral of
the left-hand side is bounded by a constant times
$h^n\|R_{1+h}\|_{\HS}^2=\mathcal O(h^\infty)$.  The integral of the first
term on the right-hand side is the quantity we are estimating.
Finally, the second term on the right-hand side is bounded by a
constant times
$h^n|\Tr((\partial_{\tilde\lambda}R_{\tilde\lambda})\Pi_{[1,1+h]}
R_{\tilde\lambda}^*)|$, which is $\mathcal O(h^\infty)$ uniformly in
$\tilde\lambda$, as the Hilbert--Schmidt norms of both
$R_{\tilde\lambda}$ and $\partial_{\tilde\lambda}R_{\tilde\lambda}$
are $\mathcal O(h^\infty)$.
\end{proof}
%
%

\section{Proofs}
  \label{s:proofs}

\subsection{Proof of Theorem~\ref{t:convergence}}
  \label{s:proofs-1}

In this section, we prove the convergence Theorem~\ref{t:convergence}
under the following assumption:
\begin{equation}
  \label{e:mes-0}
\mu_L(K)=0,
\end{equation}
where $\mu_L$ denotes the Liouville measure on $S^*M$ and $K$ is the trapped set.
First of all, note that~\eqref{e:mes-0} implies
\begin{equation}
  \label{e:mes0}
\mu_L(\Gamma_\pm)=0.
\end{equation}
Indeed, fix $\varepsilon\in (0,\varepsilon_0)$, where
$\varepsilon_0$ is defined in~\as{G2}, and take the set
$\Gamma_+^\varepsilon=\Gamma_+\cap\{x\geq\varepsilon\}$.
For $(m,\nu)\in\Gamma_+\cap \{x=\varepsilon\}$, we have
$\dot x(m,\nu)<0$; indeed, otherwise $(m,\nu)$
directly escapes in the backward direction and thus cannot lie
in $\Gamma_+$. It follows that $g^{-t}(\Gamma_+^\varepsilon)\subset\Gamma_+^\varepsilon$
for $t\geq 0$. Since $\Gamma_+^\varepsilon$ is bounded, and
$\mu_L$ is invariant under the geodesic flow,
we have
$$
\mu_L(\Gamma_+^\varepsilon)=\lim_{t\to +\infty}\mu_L(g^{-t}(\Gamma_+^\varepsilon))
=\mu_L\bigg(\bigcap_{t\geq 0}g^{-t}(\Gamma_+^\varepsilon)\bigg)=\mu_L(K)=0.
$$
Letting $\varepsilon\to 0$, we get~\eqref{e:mes0}.

We next note that the averaged $L^2$ bound~\eqref{e:avg-bdd} on $E_h$
on compact sets, together with~\eqref{e:e-eq-h} and the elliptic
Proposition~\ref{l:elliptic}, give the following
%
%
\begin{prop}
  \label{l:elliptic-our}
Assume that $A\in\Psi^0(M)$ is compactly supported and $\WFh(A)\cap S^*M=\emptyset$.
Then
\begin{equation}
  \label{e:elliptic-our}
h^{-1}\|\langle A E_h(\lambda,\xi),E_h(\lambda,\xi)\rangle\|_{L^1_{\xi,\lambda}(\plM\times [1,1+h])}
=\mathcal O(h^\infty).
\end{equation}
\end{prop}
%
%
Therefore, it is enough to prove~\eqref{e:convergence-1} for a
compactly supported $A\in\Psi^{\comp}(M)$ microlocalized in an
arbitrarily small neighborhood of $S^*M$.

Take $t>0$; we will calculate limits of the form $\lim_{t\to
+\infty}\lim_{h\to 0}$, thus $\mathcal O_t(h^\infty)$ expressions
(that is, expressions that are $\mathcal O(h^\infty)$ with the
constants depending on $t$) will be negligible. Take $\chi\in
C_0^\infty(M)$ indepedendent of $t$ and such that $A=\chi A\chi$.  We
first use that $E_h$ is a generalized eigenfunction of the
Laplacian~\eqref{e:e-eq-h} and apply Lemma~\ref{l:key}:
for each $\lambda\in [1,1+h]$ and each $\xi\in \pl M$,
\begin{equation}
  \label{e:propagated0}
\chi E_h=\chi e^{-it(\lambda^2+c_0h^2)/2h}U(t)\chi_t E_h+\mathcal O_t(h^\infty\|E_h\|_{L^2(K_t)})_{L^2}.
\end{equation}
Here $U(t)=e^{ith\Delta/2}$ is the semiclassical Schr\"odinger propagator
and $\chi_t\in C_0^\infty(M)$ is supported in the interior of the compact set $K_t\subset M$
and satisfies $d_g(\supp\chi,\supp(1-\chi_t))>t$. We also assume that $|\chi_t|\leq 1$
everywhere and
$K_t$ contains $\{x\geq\varepsilon_0\}$, where $\varepsilon_0$ is defined in~\as{G2}.
By Proposition~\ref{l:egorov}, we can write
$U(-t)AU(t)=A^{-t}+\mathcal O_t(h^\infty)_{L^2\to L^2}$, where
$A^{-t}\in\Psi^{\comp}$ is compactly supported. Then
\begin{equation}
  \label{e:propagated}
\langle AE_h,E_h\rangle=\langle A^{-t}\chi_t E_h,\chi_t E_h\rangle+\mathcal O_t(h^\infty)\|E_h\|_{L^2(K_t)}^2.
\end{equation}
We will now write
\begin{equation}\label{defofA0A1}
A^{-t}=A^{-t}_0+A^{-t}_1,\
A^{-t}_0:=A^{-t}\varphi,\
A^{-t}_1:=A^{-t}(1-\varphi),
\end{equation}
where the $L^2$
norm of the principal symbol of $A_0^{-t}$ will decay with $t$ and
the operator $A^{-t}_1$ will be negligible on $E^1_h$. The function
$\varphi\in C_0^\infty(M)$ is taken independent of $t$
and such that $\supp\chi\subset\{x>\varepsilon_\chi\}$ for some $\varepsilon_\chi$
and $\varphi=1$ near $\{x\geq\varepsilon_\chi\}$. We also require that
$\varphi=1$ near $\{x\geq\varepsilon_1\}$, where $\varepsilon_1$ comes from
the assumption~\as{A7}.

We first show that the terms in~\eqref{e:propagated} featuring both $A^{-t}_1$
and $E_h^1$ are $\mathcal O(h^\infty)$. For that, we need to show that the trajectories
in $\WFh(A^{-t}_1)\subset \supp(1-\varphi)\cap g^t(\supp\chi)$ satisfy the
geometric property shown on Figure~\ref{f:general-explanation}:
%
%
\begin{lemm}
  \label{l:geometry-1}
Let $t\geq 0$.  Assume that $(m,\nu)\in S^*M$ satisfies
$m\in\supp(1-\varphi)$, but $g^{-t}(m,\nu)\in\supp\chi$.  Then:
\begin{enumerate}
\item $(m,\nu)$ directly escapes in the forward direction, in the sense
of Definition~\ref{d:directly-escape};
\item for each $s\geq 0$, $g^s(m,\nu)$ does not lie in the set $W_\xi$
defined in~\eqref{e:w-xi}, for any $\xi\in \plM$.
\end{enumerate}
\end{lemm}
\begin{proof}
(1) We have $x(m)<\varepsilon_1\leq\varepsilon_0$; therefore, if $(m,\nu)$ does
not directly escape in the forward direction, then it directly
escapes in the backward direction; this would imply that $x(g^{-t}(m,\nu))$
is decreasing in $t\geq 0$, which is impossible as
$x(m)<\varepsilon_\chi<x(g^{-t}(m,\nu))$.

(2) The point $g^s(m,\nu)$ directly escapes in the forward direction
and $x(g^s(m,\nu))<\varepsilon_1$. If $g^s(m,\nu)\in W_\xi$, then by~\as{G4},
$\xi=\xi_{+\infty}(m,\nu)$, but this is impossible as $\chi_0=1$
near $(g^s(m,\nu),\xi_{+\infty}(m,\nu))$ by~\as{A7}.
\end{proof}
%
%
Combining Lemma~\ref{l:geometry-1} with the microlocal information
we have on $E^1_h$, we get 
%
%
\begin{prop}
  \label{l:analysis-1}
If $E_h=\chi_0E^0_h+E^1_h$ is the decomposition~\eqref{e:e-h-decomposition},
then for each $t\geq 0$,
\begin{equation}
  \label{e:analysis-1}
\begin{gathered}
\langle AE_h,E_h\rangle=\langle A^{-t}_1\chi_t \chi_0E^0_h,\chi_t \chi_0E^0_h\rangle
+\langle A^{-t}_0 \chi_t E_h,\chi_t E_h\rangle
+\mathcal O_t(h^\infty(1+\|E_h\|_{L^2(K_t)}^2)).
\end{gathered}
\end{equation}
where $A_0^{-t},A_1^{-t}$ are defined in \eqref{defofA0A1}.
\end{prop}
\begin{proof}
By~\eqref{e:propagated}, it is enough to show that
$$
\langle A^{-t}_1\chi_t E_h,\chi_t E_h\rangle
-\langle A^{-t}_1\chi_t\chi_0 E_h^0,\chi_t \chi_0E_h^0\rangle
=\mathcal O_t(h^\infty(1+\|E_h\|_{L^2(K_t)}^2)).
$$
Given that $\|\chi_0 E_h^0\|_{L^2(K_t)}=\mathcal O(1)$, it suffices to prove 
$$
\|B\chi_t E^1_h\|_{L^2}=\mathcal O_t(h^\infty(1+\|E_h\|_{L^2(K_t)})),
$$
where $B$ is equal to either $A^{-t}_1$ or its adjoint.
This in turn follows from
\begin{equation}
  \label{e:analysis-1.1}
\|B\chi_t \widetilde E^1_h\|_{L^2}=\mathcal O_t(h^\infty),
\end{equation}
with $\widetilde E^1_h$ defined in~\eqref{e:e-1-h-tilde}. 
Take $(m,\nu)\in\WFh(B\chi_t\widetilde E^1_h)\subset S^*M$. Then
by Proposition~\ref{l:egorov},
$$
(m,\nu)\in\WFh(B)\subset\WFh(A^{-t})\cap\supp(1-\varphi)
\subset g^t(\WFh(A))\cap\supp(1-\varphi).
$$
Since $\WFh(A)\subset\supp\chi$, we see that
$m\in\supp(1-\varphi)$ and $g^{-t}(m,\nu)\in\supp\chi$;
therefore, by Lemma~\ref{l:geometry-1},
the geodesic $g^s(m,\nu)$ escapes in the forward direction
and does not pass through $W_\xi$ for
$s\geq 0$. But then by~\as{A6} the point $(m,\nu)$ cannot
lie in $\WFh(\widetilde E^1_h)$, a contradiction.
We showed that the wavefront set of $B\chi_t\widetilde E^1_h$
is empty, which implies~\eqref{e:analysis-1.1}.
\end{proof}
%
%
We now use the averaged estimate~\eqref{e:hs-est} and the
Hilbert--Schmidt norm estimates from
Section~\ref{s:prelim.propagator}, to estimate the second term on the
right-hand side of~\eqref{e:analysis-1}:
%
%
\begin{prop}
  \label{l:analysis-2}
There exists a constant $C$ independent of $t$ such that
\begin{equation}
  \label{e:analysis-2}
h^{-1}\|\langle A^{-t}_0\chi_t E_h,\chi_t E_h\rangle\|_{L^1_{\xi,\lambda}(\plM\times
[1,1+h])}\leq C\|(\sigma(A)\circ g^{-t})\varphi\|_{L^2(S^*M)}+\mathcal O_t(h).
\end{equation}
Here $\|a\|_{L^2(S^*M)}$ is the $L^2$ norm of the restriction
of $a$ to $S^*M$ with respect to the Liouville measure.
\end{prop}
\begin{proof}
Take a real-valued function $\varphi_1\in C_0^\infty(M)$ independent of $t$ such
that $\varphi_1=1$ near $\supp\varphi$. 
Then the left-hand side of~\eqref{e:analysis-2} is bounded by
$$
h^{-1}\|\langle A^{-t}_0 \chi_t E_h,\varphi_1 \chi_t E_h\rangle\|_{L^1_{\xi,\lambda}}
+h^{-1}\|\langle (1-\varphi_1)A^{-t}_0 \chi_t E_h, \chi_t E_h\rangle\|_{L^1_{\xi,\lambda}},
$$
where the $L^1$, and later $L^2$, norms in $\xi,\lambda$ are taken over
$\plM\times [1,1+h]$.
The second term here is $\mathcal O_t(h^\infty)$ by the bound~\eqref{e:avg-bdd}
and since $(1-\varphi_1)A^{-t}_0=\mathcal O_t(h^\infty)_{L^2\to L^2}$
is compactly supported. The first term can be estimated by applying the Cauchy--Schwarz
inequality first in $m$ and then in $(\lambda,\xi)$:
$$
\begin{gathered}
h^{-1}\|\langle A^{-t}_0 \chi_t E_h,\varphi_1 \chi_t E_h\rangle\|_{L^1_{\xi,\lambda}}
\leq h^{-1}\|\,\|A^{-t}_0 \chi_t E_h\|_{L^2(M)}\cdot \|\varphi_1 \chi_t E_h\|_{L^2(M)}\|_{L^1_{\xi,\lambda}}\\\leq
h^{-1/2}\|A^{-t}_0\chi_t E_h\|_{L^2_{m,\xi,\lambda}}\cdot h^{-1/2} \|\varphi_1 \chi_t E_h\|_{L^2_{m,\xi,\lambda}}.
\end{gathered}
$$
Now, $h^{-1/2}\|\varphi_1 \chi_t E_h\|_{L^2_{m,\xi,\lambda}}$ is bounded (independently of $t$)
uniformly in $h$ by~\eqref{e:avg-bdd}. As for $h^{-1/2}\|A^{-t}_0\chi_t E_h\|_{L^2_{m,\xi,\lambda}}$,
we can estimate it using~\eqref{e:hs-est} by a constant times
$$
h^{n/2}\|A^{-t}_0\chi_t\Pi_{[1,1+h]}\|_{\HS}.
$$
Note that the operator $A^{-t}_0\chi_t\in\Psi^{\comp}$ is compactly
supported and it is compactly microlocalized independently of $t$. It
then remains to apply~\eqref{e:h-s-estimate-1} (to the adjoint of
our operator); by
Proposition~\ref{l:egorov}, the principal symbol of $A^{-t}_0\chi_t$ is
given by $(\sigma(A)\circ g^{-t})\varphi$.
 \end{proof}
%
%
We now use the dynamical assumption that $\mu_L(K)=0$. The function
$(\sigma(A)\circ g^{-t})\varphi$ is supported in a $t$-independent
compact set and bounded uniformly in $t$. Moreover, it converges to
zero pointwise on $S^*M\setminus\Gamma_+$ as $t\to+\infty$. Therefore,
by~\eqref{e:mes0} and the dominated convergence theorem we have
$(\sigma(A)\circ g^{-t})\varphi\to 0$ in $L^2(S^*M)$, as $t\to
+\infty$. It then follows from~\eqref{e:analysis-1} together with the
bound~\eqref{e:avg-bdd} and from~\eqref{e:analysis-2} that
$$
\lim_{t\to +\infty}\limsup_{h\to 0}h^{-1}\|\langle AE_h,E_h\rangle
-\langle A^{-t}_1\chi_t\chi_0E_h^0,\chi_t\chi_0E_h^0\rangle\|_{L^1_{\xi,\lambda}(\plM\times [1,1+h])}=0.
$$
To prove Theorem~\ref{t:convergence}, it now remains to show that
\begin{equation}
  \label{e:convergence-int}
\lim_{t\to +\infty}\limsup_{h\to 0}h^{-1}\bigg\|\langle A^{-t}_1\chi_t\chi_0E_h^0,
\chi_t\chi_0E_h^0\rangle-\int_{S^*M}\sigma(A)\,d\mu_\xi\bigg\|_{L^1_{\xi}(\plM)}=0
\end{equation}
uniformly in $\lambda=1+\mathcal O(h)$.
We first note that by~\eqref{e:e0-h} the function
$$
\chi_t\chi_0E^0_h(\lambda,\xi;m)=e^{{i\lambda\over h}\phi_\xi(m)}\chi_t(m)\chi_0(m,\xi)
b^0(1,\xi,m;0)+\mathcal O_t(h)_{L^2}
$$
is a compactly supported Lagrangian distribution associated to the Lagrangian
$\Lambda_\xi$ from~\eqref{e:lambda-xi}. Therefore, by Proposition~\ref{l:lagrangian-mul}, we find
\begin{equation}
  \label{e:convergence-int-777}
A^{-t}_1\chi_t\chi_0E^0_h(\lambda,\xi)
=e^{{i\lambda\over h}\phi_\xi}\chi_t\chi_0b^0(1,\xi,m;0)\sigma(A^{-t}_1)(m,\partial_m\phi_\xi(m))
+\mathcal O_t(h)_{L^2}.
\end{equation}
Therefore,
$$
\begin{gathered}
\langle A^{-t}_1\chi_t\chi_0 E_h^0,\chi_t\chi_0E_h^0\rangle
=\int_M \sigma(A^{-t}_1)(m,\partial_m\phi_\xi(m))
|\chi_t\chi_0b^0(1,\xi,m;0)|^2\,\Vol(m)+\mathcal O_t(h).
\end{gathered}
$$
Now, by Proposition~\ref{l:egorov},
$\sigma(A^{-t}_1)=(\sigma(A)\circ g^{-t})(1-\varphi)$. By Lemma~\ref{l:geometry-1},
this function is supported in $\mathcal{DE}_+\cap\{x<\varepsilon_1\}$, with
$\mathcal{DE}_+$ from Definition~\ref{d:directly-escape}.
Also, $\chi_t=1$ near $\supp\sigma(A^{-t}_1)$.
Then by part~2 of Lemma~\ref{l:mu-xi-1},
\begin{equation}
  \label{e:int-calc}
\langle A^{-t}_1\chi_t\chi_0 E_h^0,\chi_t\chi_0 E_h^0\rangle
=\int_{S^*M} (\sigma(A)\circ g^{-t})(1-\varphi)\,d\mu_\xi+\mathcal O_t(h).
\end{equation}
Therefore, \eqref{e:convergence-int} reduces to
\begin{equation}
  \label{e:convergence-int-2}
\lim_{t\to +\infty}\bigg\|\int_{S^*M}(\sigma(A)\circ g^{-t})(1-\varphi)\,d\mu_\xi-
\int_{S^*M}\sigma(A)\,d\mu_\xi\bigg\|_{L^1_\xi(\plM)}=0.
\end{equation}
By part~1 of Lemma~\ref{l:mu-xi-1} and~\eqref{e:disintegration-2},
we write the norm on the left-hand side of~\eqref{e:convergence-int-2}
as
$$
\bigg\|\int_{S^*M}\sigma(A)(\varphi\circ g^t)\,d\mu_\xi\bigg\|_{L^1_\xi(\plM)}
\leq\int_{S^*M}|\sigma(A)(\varphi\circ g^t)|\,d\mu_L.
$$
The expression under the integral
on the right-hand side is bounded and compactly supported uniformly in $t$
and converges to zero pointwise on $S^*M\setminus\Gamma_-$; by~\eqref{e:mes0} and
the dominated convergence theorem, we get~\eqref{e:convergence-int-2}. This finishes
the proof of Theorem~\ref{t:convergence}.

\smallsection{The nontrapped case}
We briefly discuss the situation when
$\WFh(A)\cap\Gamma_-=\emptyset$. In this case, for $t$ large enough
(depending on $A$), for any $(m,\nu)\in\WFh(A)$ we have
$g^t(m,\nu)\not\in\supp\varphi$ and thus
$$
A_0^{-t}=\mathcal O(h^\infty)_{L^2\to L^2}.
$$
Then by~\eqref{e:analysis-1} and the bound~\eqref{e:avg-bdd},
$$
\langle AE_h,E_h\rangle=\langle A^{-t}_1\chi_t\chi_0E^0_h,\chi_t\chi_0E^0_h\rangle
+\mathcal O(h^\infty)_{L^1_{\xi,\lambda}(\plM\times [1,1+h])}.
$$
The quantity $\langle A^{-t}_1\chi_t\chi_0 E^0_h,\chi_t\chi_0
E^0_h\rangle$ is calculated in~\eqref{e:int-calc} up to $\mathcal
O(h)$. However, since $E^0_h$ is a Lagrangian distribution, one can
get by Proposition~\ref{l:lagrangian-mul} a full expansion of this
quantity in powers of $h$; this yields
\begin{equation}
  \label{e:full-expansion}
\langle AE_h(\lambda,\xi),E_h(\lambda,\xi)\rangle=\sum_{0\leq j<N}h^j\int_{S^*M}L_ja\,d\mu_\xi
+\mathcal O(h^{N+1})_{L^1_{\xi,\lambda}(\plM\times [1,1+h])},
\end{equation}
where $A=\Op_h(a)$ for some symbol $a$ and some quantization procedure
$\Op_h$ and each $L_j$ is a differential operator of order $2j$ on
$T^*M$, with $L_0=1$.

\subsection{Estimates on the remainder}
  \label{s:proofs-2}

In this subsection, we prove~\eqref{e:convergence-2} and establish an
approximation fact (Proposition~\ref{l:analysis2-main}) used in the
proofs of~\eqref{e:convergence-3} and Theorem~\ref{asympofs_A}.

\smallsection{Classical escape rate and Ehrenfest time}
Let $K_0\subset M$ be a compact geodesically convex set (in the sense of~\eqref{geodconvex}) containing a
neighborhood of the projection of the trapped set $K$ onto $M$. As
in~\eqref{e:T-t}, define the set
$$
\mathcal T(t)=\{(m,\nu)\in S^*M\mid m\in K_0,\ g^t(m,\nu)\in K_0\}.
$$
The choice of $K_0$ does not matter here: if $K'_0\subset M$ is another set with
same properties and $\mathcal T'(t)$ is defined using $K'_0$ in place of $K_0$,
then there exists a constant $T_0>0$ such that for each $T\geq T_0$ and $t\geq 0$,
\begin{equation}
  \label{e:who-cares-about-k-0}
g^T(\mathcal T'(t+2T))\subset \mathcal T(t).
\end{equation}
Indeed, assume that~\eqref{e:who-cares-about-k-0} were false. Then
there exists sequences $T_j\to +\infty$, $t_j\geq 0$, and
$(m_j,\nu_j)\in S^*M$ such that $g^{-T_j}(m_j,\nu_j)$ and
$g^{t_j+T_j}(m_j,\nu_j)$ both lie in $K'_0$, but for each $j$, either
(1) $(m_j,\nu_j)\not\in K_0$ or (2) $g^{t_j}(m_j,\nu_j)\not\in K_0$.
We may assume that case~(1) holds for all $j$; case (2) is handled
similarly, reversing the direction of the flow and taking
$g^{t_j}(m_j,\nu_j)$ in place of $(m_j,\nu_j)$.  Take $\varepsilon>0$
such that $K'_0\subset \{x\geq\varepsilon\}$; since
$\{x\geq\varepsilon\}$ is geodesically convex (in the sense of~\eqref{geodconvex}) for $\varepsilon$ small
enough, we have $(m_j,\nu_j)\in \{x\geq\varepsilon\}$.  Passing to a
subsequence, we can assume that $(m_j,\nu_j)\to (m,\nu)\in S^*M$ as
$j\to +\infty$.  Now, since $g^{-T_j}(m_j,\nu_j)\in K'_0$ and $T_j\to
+\infty$, we have $(m,\nu)\in \Gamma_+$ (indeed, otherwise there would
exist $s>0$ such that $g^{-s}(m,\nu)\in \{x<\varepsilon\}$ and
this would also hold in a neighborhood of $(m,\nu)$). Similarly, since
$g^{t_j+T_j}(m_j,\nu_j)\in K'_0$ and $t_j+T_j\to +\infty$, we have
$(m,\nu)\in\Gamma_-$. It follows that $(m,\nu)\in K$, which is
impossible, as each $(m_j,\nu_j)$ does not lie in $K_0$, which
contains a neighborhood of $K$.

By changing $\Lambda_0$ slightly and
using~\eqref{e:who-cares-about-k-0}, we see that the choice of $K_0$
does not matter for the validity of~\eqref{e:convergence-2}
and~\eqref{e:convergence-3}; more precisely, if
$\Lambda_0>\Lambda'_0$, then $r'(h,\Lambda'_0)\leq Cr(h,\Lambda_0)$,
where $r'$ is defined by~\eqref{defofrh} using $\mathcal T'$ in place
of $\mathcal T$. Also, the
maximal expansion rate $\Lambda_{\max}$ defined in~\eqref{lamax}
does not depend on the choice of $K_0$.

We now choose a geodesically convex $K_0$ in the sense of~\eqref{geodconvex} such that its interior
contains the supports of all cutoff functions and compactly supported
operators used in the argument below. We will rely on
Proposition~\ref{l:ehrenfest} (with $U$ equal to the interior of
$K_0$); we let $\Lambda_0>\Lambda_{\max}$ and fix $\varepsilon_e>0$
and $\Lambda'_0$ such that
$\Lambda_0>\Lambda'_0>(1+2\varepsilon_e)\Lambda_{\max}$. Define the
Ehrenfest time
\begin{equation}
  \label{e:t-e}
t_e:=\log(1/h)/(2\Lambda_0).
\end{equation}
Then when propagating an operator in $\Psi^{\comp}$ microlocalized
inside
\begin{equation}
  \label{e:e-e-e}
\mathcal E_{\varepsilon_e}:=\{1-\varepsilon_e\leq|\nu|_g\leq 1+\varepsilon_e\}
\end{equation}
with cutoffs supported inside $K_0$, as in Proposition~\ref{l:ehrenfest},
for time $t=lt_0\in [-t_e,t_e]$, we get a mildly exotic pseudodifferential operator
in $\Psi^{\comp}_{\rho_e}$, where
\begin{equation}
  \label{e:rho-e}
\rho_e:=t_e\Lambda'_0/\log(1/h)=\Lambda'_0/(2\Lambda_0)<1/2.
\end{equation}

\smallsection{First decomposition of $\langle AE_h,E_h\rangle$}
By Proposition~\ref{l:elliptic-our},
we may assume that $A\in\Psi^{\comp}(M)$ is compactly supported and microlocalized inside
the set $\mathcal E_{\eps_e}$ defined in~\eqref{e:e-e-e}.

We first establish the following decomposition similar to~\eqref{e:analysis-1}:
\begin{equation}
  \label{e:iterated-estimate-2}
\begin{gathered}
\langle AE_h,E_h\rangle
=e^{il\beta} \langle A(\varphi U(t_0))^l\varphi E_h,E_h\rangle\\
+\sum_{j=1}^l e^{ij\beta}\langle A(\varphi U(t_0))^j(1-\varphi)\varphi_{t_0}\chi_0E^0_h,E_h\rangle
+\mathcal O(h^\infty \mathcal N(E_h)^2),
\end{gathered}
\end{equation}
uniformly in $\xi\in\plM$ and $\lambda\in [1,1+h]$; comparing with \eqref{e:analysis-1}, the 
first term in the right hand side of \eqref{e:iterated-estimate-2} corresponds to the $A_0^{-t}$ term in 
\eqref{e:analysis-1}, the sum over $j$ corresponds to the $A_1^{-t}$ term in \eqref{e:analysis-1}.
Here $l=\mathcal
O(\log(1/h))$ is a nonnegative integer and $t_0>0$ and
$\varphi,\varphi_{t_0}\in C_0^\infty(M)$, specified below, are
independent of $j$.  The quantity $\mathcal N(E_h)$, defined
in~\eqref{e:n-e-h}, is related to the $L^2$ norm of $E_h$ on a certain
compact set, and is bounded on average by~\eqref{e:n-e-h-bdd}.  The
real-valued parameter $\beta$ is equal to
\begin{equation}
  \label{e:beta}
\beta=-t_0(\lambda^2+c_0h^2)/(2h)
\end{equation}
and will not play a big role in our argument.

To show~\eqref{e:iterated-estimate-2}, we start by considering the
functions $\varphi,\varphi_1,\varphi_2\in C_0^\infty(M)$ such that:
%
%
\begin{itemize}
\item $0\leq\varphi,\varphi_1,\varphi_2\leq 1$ everywhere,
\item $\varphi=1$ near $\supp\varphi_2$ and $\varphi_1=1$ near $\supp\varphi$, and
\item $\varphi_2=1$ both near the support of $A$ and
near the set $\{x\geq\varepsilon_1\}$, with $\varepsilon_1$
defined in~\as{A7}.
\end{itemize}
%
%
The proof of~\eqref{e:iterated-estimate-2} only uses the function
$\varphi$, however the other two functions will be required for the
more precise decomposition~\eqref{e:analysis2-main} below.

We now have the following analogue of Lemma~\ref{l:geometry-1}:
%
%
\begin{lemm}\label{l:geometry-2}
There exists $t_0\geq 0$ such that if $(m,\nu)\in S^*M$ satisfies
\begin{equation}
  \label{e:geometry-2}
m\in\supp(1-\varphi_2)\text{ and }g^{-t}(m,\nu)\in\supp\varphi_1
\text{ for some }t\geq t_0,
\end{equation}
then:
\begin{enumerate}
\item $(m,\nu)$ directly escapes in the forward direction;
\item for each $s\geq 0$, $g^s(m,\nu)$ does not lie in the set $W_\xi$
defined in~\eqref{e:w-xi} for any $\xi\in \plM$; and
\item for each $s\geq t_0$, $g^s(m,\nu)\not\in\supp\varphi_1$.
\end{enumerate}
\end{lemm}
\begin{proof}
%
%
\begin{figure}
\includegraphics{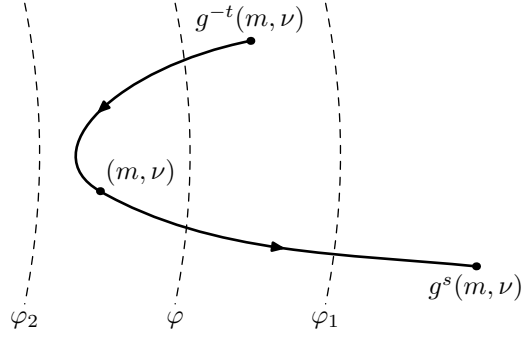}
\caption{An illustration of Lemma~\ref{l:geometry-2}. The functions
$\varphi,\varphi_1,\varphi_2$ are supported to the left of the corresponding
dashed lines; the right side of the figure represents infinity.}
\label{f:geometry-2}
\end{figure}
%
%
(1) Let $\supp\varphi_1\subset \{x\geq\varepsilon_\varphi\}$. The
set $\mathcal{DE}_-\cap \{x\geq\varepsilon_\varphi\}$, where
$\mathcal{DE}_-$ is specified in Definition~\ref{d:directly-escape},
is compact; therefore, there exists $t_0>0$ such that for $t\geq t_0$
and $(m,\nu)\in \mathcal{DE}_-\cap\{x\geq\varepsilon_\varphi\}$,
we have $g^{-t}(m,\nu)\not\in\supp\varphi_1$.

Now, assume that $(m,\nu)$ satisfies~\eqref{e:geometry-2}, but it does
not directly escape in the forward direction. Since $(m,\nu)\in\supp(1-\varphi_2)$,
we have $x(m)\leq\varepsilon_0$; therefore, $(m,\nu)\in \mathcal{DE}_-$.
Then $x(m)\geq x(g^{-t}(m,\nu))\geq\varepsilon_\varphi$; therefore,
$(m,\nu)\in \mathcal{DE}_-\cap \{x\geq\varepsilon_\varphi\}$, a contradiction
with the fact that $g^{-t}(m,\nu)\in\supp\varphi_1$ and $t\geq t_0$.

(2) This is proved exactly as part~2 of Lemma~\ref{l:geometry-1}.

(3) It is enough to use part~(1), take $t_0$ large enough, and use that
the set $\mathcal{DE}_+\cap\{x\geq\varepsilon_\varphi\}$ is compact. 
\end{proof}
%
%
Take $t_0$ from Lemma~\ref{l:geometry-2}. Let $\varphi_{t_0}\in C_0^\infty(M)$
be real-valued and satisfy $d_g(\supp\varphi_1,\supp(1-\varphi_{t_0}))>t_0$. Take
a compact set $K_{t_0}\subset M$ whose interior contains $\supp\varphi_{t_0}$. Put
\begin{equation}
  \label{e:n-e-h}
\mathcal N(E_h):=1+\|E_h\|_{L^2(K_{t_0})};
\end{equation}
this quantity depends on $\lambda$ and $\xi$ and we know by~\eqref{e:avg-bdd} that
\begin{equation}
  \label{e:n-e-h-bdd}
h^{-1}\|\mathcal N(E_h)\|^2_{L^2_{\xi,\lambda}(\plM\times [1,1+h])}=\mathcal O(1).
\end{equation}
By~\eqref{e:e-eq-h} and Lemma~\ref{l:key}, we have similarly to~\eqref{e:propagated0},
\begin{equation}
  \label{e:e-propagated-again}
\varphi E_h=e^{i\beta}\varphi U(t_0)\varphi_{t_0}E_h
+\mathcal O(h^\infty \mathcal N(E_h))_{L^2}.
\end{equation}
Here $\beta$ is given by~\eqref{e:beta}.
Iterating~\eqref{e:e-propagated-again} by writing $\varphi_{t_0}=
\varphi+(1-\varphi)\varphi_{t_0}$, we get for
$l=\mathcal O(\log(1/h))$ (or even for $l$ polynomially bounded in $h$)
\begin{equation}
  \label{e:iterated-estimate-0}
\varphi E_h=e^{il\beta} (\varphi U(t_0))^l \varphi E_h
+\sum_{j=1}^l e^{ij\beta} (\varphi U(t_0))^j (1-\varphi)\varphi_{t_0}E_h
+\mathcal O(h^\infty \mathcal N(E_h))_{L^2},
\end{equation}
uniformly in $\xi\in\plM$ and $\lambda\in[1,1+h]$.  Same is true if
$\varphi$ is replaced by any function $\varphi'\in C_0^\infty(M)$ such
that $d_g(\supp\varphi',\supp(1-\varphi_{t_0}))>t_0$.  One can also
replace $U(t_0)$ by $U(-t_0)$.

We now use our knowledge of the wavefront set of $\widetilde E^1_h$ to prove
the following analogue of Proposition~\ref{l:analysis-1}:
%
%
\begin{prop}\label{l:analysis2-1}
If $E_h=\chi_0E^0_h+E^1_h$ is the decomposition~\eqref{e:e-h-decomposition}, then
\begin{equation}
  \label{e:analysis2-1}
\|\varphi U(t_0)(1-\varphi)\varphi_{t_0}E^1_h\|_{L^2}=\mathcal O(h^\infty \mathcal N(E_h)),
\end{equation}
uniformly in $\xi\in\plM$ and $\lambda\in[1,1+h]$.  Same is true if we
replace each instance of $\varphi$ by any function in the set
$\{\varphi,\varphi_1,\varphi_2\}$.
\end{prop}
\begin{proof}
Recalling the definition~\eqref{e:e-1-h-tilde} of $\widetilde E^1_h$,
we see that~\eqref{e:analysis2-1} follows from
\begin{equation}
  \label{e:analysis2-1.1}
\|\varphi U(t_0)(1-\varphi)\varphi_{t_0}\widetilde E^1_h\|_{L^2}=\mathcal O(h^\infty).
\end{equation}
We now make the following observation: a point $(m,\nu)\in S^*M$ in
the wavefront set of $\widetilde E^1_h$ will make an $\mathcal
O(h^\infty)$ contribution to~\eqref{e:analysis2-1.1} unless $m\in\supp
(1-\varphi)$, but $g^{-t_0}(m,\nu)\in \supp\varphi$; however,
by~\as{A6} and Lemma~\ref{l:geometry-2}, in this case
$(m,\nu)\not\in\WFh(\widetilde E^1_h)$.  To make this argument
rigorous, we can write (bearing in mind that $\widetilde E^1_h$ is
polynomially bounded)
$$
\varphi_{t_0}\widetilde E^1_h=B \widetilde E^1_h+\mathcal O(h^\infty)_{L^2},
$$
where $B\in\Psi^{\comp}$ is compactly supported and such that
$$
(m,\nu)\in\WFh(B)\cap \supp(1-\varphi) \Longrightarrow
g^{-t_0}(m,\nu)\not\in\supp \varphi.
$$
Then the operator $\varphi U(t_0)(1-\varphi)B$ is $\mathcal O(h^\infty)_{L^2\to L^2}$
by part~2 of Proposition~\ref{l:egorov}, which proves~\eqref{e:analysis2-1.1}.
\end{proof}
%
%
Using~\eqref{e:analysis2-1}, we can replace $E_h$ by $\chi_0 E^0_h$ in
each term of the sum~\eqref{e:iterated-estimate-0}:
\begin{equation}
  \label{e:iterated-estimate-1}
\begin{gathered}
\varphi E_h=e^{il\beta}(\varphi U(t_0))^l\varphi E_h
+\sum_{j=1}^l e^{ij\beta}(\varphi U(t_0))^j(1-\varphi)\varphi_{t_0}\chi_0E^0_h
\\+\mathcal O(h^\infty \mathcal N(E_h))_{L^2}.
\end{gathered}
\end{equation}
Applying the operator $A=\varphi A\varphi$, we get~\eqref{e:iterated-estimate-2}.

\smallsection{Properties of propagators up to Ehrenfest time}
We will now establish certain properties of the cut off and iterated
propagators up to the Ehrenfest time $t_e$ defined in~\eqref{e:t-e},
or, in certain cases, up to twice the Ehrenfest time. The need for
these properties arises mostly because of the cutoffs present in the
argument. Define the Ehrenfest index
\begin{equation}
  \label{e:l-e}
l_e := \lfloor t_e/t_0\rfloor + 1\sim\log(1/h).
\end{equation}
%
%
\begin{lemm}\label{l:analysis2-2}
Assume that $\varphi',\varphi''\in C_0^\infty(M)$ satisfy
$|\varphi'|,|\varphi''|\leq 1$ everywhere.
Let $B\in\Psi^{\comp}$ be compactly supported and
microlocalized inside the set $\mathcal E_{\varepsilon_e}$ defined in~\eqref{e:e-e-e}.
Then:

1. If $\varphi''=1$ near $\supp\varphi'$, then for $0\leq j\leq l_e$,
\begin{gather}
  \label{e:analysis2-2.2}
(\varphi' U(\pm t_0))^jBU(\mp jt_0)=(\varphi' U(\pm t_0))^jB(U(\mp t_0)\varphi'')^j
+\mathcal O(h^\infty)_{L^2\to L^2},\\
  \label{e:analysis2-2.2-rev}
(U(\pm t_0)\varphi')^j BU(\mp jt_0)=(U(\pm t_0)\varphi')^j B(\varphi''U(\mp t_0))^j
+\mathcal O(h^\infty)_{L^2\to L^2}.
\end{gather}

2. If $B_1,B_2\in\Psi^{\comp}$ satisfy same conditions as $B$ and
moreover $\WFh(B_1)\cap \WFh(B_2)=\emptyset$, then for $0\leq j\leq 2l_e$
(that is, up to twice the Ehrenfest time)
\begin{equation}
  \label{e:analysis2-2.3}
B_1(\varphi' U(\pm t_0))^j B(U(\mp t_0)\varphi'')^j B_2=\mathcal O(h^\infty)_{L^2\to L^2}.
\end{equation}
Same is true if we replace $\varphi' U(\pm t_0)$ by $U(\pm t_0)\varphi'$
and/or replace $U(\mp t_0)\varphi''$ by $\varphi'' U(\mp t_0)$.

3. If $\varphi''=1$ near $\supp\varphi'$ and both $\varphi''$
and $B$ are supported at distance more than $t_0$ from $\supp(1-\varphi_{t_0})$, then
for $0\leq j\leq l_e$ ($\beta$ is defined in \eqref{e:beta})
\begin{gather}
  \label{e:analysis2-2.1}
e^{\pm ij\beta}(\varphi' U(\pm t_0))^j B E_h
=(\varphi' U(\pm t_0))^j B (U(\mp t_0)\varphi'')^j E_h
+\mathcal O(h^\infty \mathcal N(E_h))_{L^2},\\
  \label{e:analysis2-2.1-rev}
e^{\pm ij\beta}(U(\pm t_0)\varphi')^j B E_h
=(U(\pm t_0)\varphi')^j B (\varphi'' U(\mp t_0))^j\varphi_{t_0} E_h
+\mathcal O(h^\infty \mathcal N(E_h))_{L^2}.
\end{gather}
\end{lemm}
\begin{proof}
We will repeatedly use Propositions~\ref{l:egorov}
and~\ref{l:ehrenfest} and omit the $\mathcal O(h^\infty)_{L^2\to L^2}$
remainders present there.

1. We prove~\eqref{e:analysis2-2.2}; \eqref{e:analysis2-2.2-rev}
is proved similarly. Assume that the signs are chosen so that~\eqref{e:analysis2-2.2}
features~$\varphi'U(t_0)$. We argue by induction in $j$.
The case $j=0$ is obvious. Now, assume that~\eqref{e:analysis2-2.2} is true for $j-1$
in place of $j$. Then
$$
(\varphi' U(t_0))^j BU(- j t_0)
=\varphi' B'+\mathcal O(h^\infty)_{L^2\to L^2},
$$
where
$$
B'=U(t_0)(\varphi'U(t_0))^{j-1}B(U(-t_0)\varphi'')^{j-1}U(-t_0)
$$
is a compactly supported operator in $\Psi_{\rho_e}^{\comp}$
(modulo the $\mathcal O(h^\infty)_{L^2\to L^2}$ remainder from
Proposition~\ref{l:ehrenfest}, which we henceforth omit),
with $\rho_e$ defined in~\eqref{e:rho-e}. Since $\supp\varphi'\cap\supp
(1-\varphi'')=\emptyset$, we have
$$
\varphi' B'=\varphi'B'\varphi''+\mathcal O(h^\infty)_{L^2\to L^2}
$$
and~\eqref{e:analysis2-2.2} follows.

2. We again assume that the signs are chosen so that~\eqref{e:analysis2-2.3}
features~$\varphi'U(t_0)$.
Write $j=j_1+j_2$, where $0\leq j_1,j_2\leq l_e$, and write the left-hand side
of~\eqref{e:analysis2-2.3} 
as $U(j_1t_0)\widetilde B_1\widetilde B \widetilde B_2U(-j_1t_0)$, where
$$
\begin{gathered}
\widetilde B=(\varphi'U(t_0))^{j_2}B(U(-t_0)\varphi'')^{j_2},\\
\widetilde B_1=U(-j_1 t_0)B_1(\varphi'U(t_0))^{j_1},\
\widetilde B_2=(U(-t_0)\varphi'')^{j_1}B_2 U(j_1 t_0).
\end{gathered}
$$
Now, $\widetilde B$ is a compactly supported member
of~$\Psi_{\rho_e}^{\comp}$.  Same can be said about $\widetilde B_1$
and $\widetilde B_2$, by applying~\eqref{e:analysis2-2.2-rev} and its
adjoint (where the role of~$\varphi'$ is played by either $\varphi'$
or~$\varphi''$ and the role of~$\varphi''$, by a suitably chosen
cutoff function). Moreover, if $U_1,U_2$ are bounded open subsets of
$T^*M$ such that $\WFh(B_k)\subset U_k$ and $U_1\cap U_2=\emptyset$,
then by Proposition~\ref{l:ehrenfest}, $\widetilde B_k$ is
microsupported, in the sense of
Definition~\ref{d:microlocal-vanishing}, on the set $g^{j_1
t_0}(U_k)$; since these two sets do not intersect, we see that
$\widetilde B_1\widetilde B\widetilde B_2=\mathcal O(h^\infty)_{L^2\to
L^2}$ as needed.

3. We once again fix the sign so that $U(t_0)$ stands next to $\varphi'$.
Formally, \eqref{e:analysis2-2.1} and~\eqref{e:analysis2-2.1-rev}
follow by applying~\eqref{e:analysis2-2.2}
and~\eqref{e:analysis2-2.2-rev}, respectively, to the
identity~$e^{ij\beta}E_h=U(- jt_0)E_h$.  To make this observation rigorous, we
write by Lemma~\ref{l:key}
$$
\begin{gathered}
e^{i\beta} BE_h=BU(- t_0)\varphi_{t_0}E_h+\mathcal O(h^\infty \mathcal N(E_h))_{L^2},\\
e^{i\beta}\varphi'' E_h= \varphi'' U(-t_0)\varphi_{t_0}E_h+\mathcal O(h^\infty \mathcal N(E_h))_{L^2}.
\end{gathered}
$$
We now use induction in $j$. For $j=0$, both~\eqref{e:analysis2-2.1} and~\eqref{e:analysis2-2.1-rev}
are trivial. Now, assume that they both hold for $j-1$ in place of $j$. We then write
$$
\begin{gathered}
e^{ij\beta}(\varphi'U(t_0))^j BE_h
\\=e^{i\beta} (\varphi'U(t_0))^j B(U(-t_0)\varphi'')^{j-1}E_h
+\mathcal O(h^\infty \mathcal N(E_h))_{L^2}
\\=(\varphi' U(t_0))^j B(U(-t_0)\varphi'')^{j-1}U(-t_0)\varphi_{t_0}E_h
+\mathcal O(h^\infty \mathcal N(E_h))_{L^2}.
\end{gathered}
$$
The operator $(\varphi' U(t_0))^jB(U(-t_0)\varphi'')^{j-1}U(-t_0)$ is
a compactly supported element of $\Psi^{\comp}_\rho$;
moreover, as $j\geq 1$, the wavefront set of this operator is contained in $\supp\varphi'$.
Since $\varphi''=1$ near $\supp\varphi'$, we can replace $\varphi_{t_0}$ by $\varphi''$ in the
last formula, proving~\eqref{e:analysis2-2.1}.

We next write
$$
e^{ij\beta}(U(t_0)\varphi')^jBE_h
=e^{i\beta}(U(t_0)\varphi')^jB(\varphi''U(-t_0))^{j-1}\varphi_{t_0} E_h
+\mathcal O(h^\infty \mathcal N(E_h))_{L^2}.
$$
However, $\varphi' (U(t_0)\varphi')^{j-1}B(\varphi'' U(-t_0))^{j-1}$
is a compactly supported element of $\Psi^{\comp}_\rho$ 
and its wavefront set is contained in $\supp\varphi'$. Since $\varphi''=1$ near
$\supp\varphi'$, we can replace $\varphi_{t_0}$ by $\varphi''$, obtaining~\eqref{e:analysis2-2.1-rev}:
$$
\begin{gathered}
e^{ij\beta}(U(t_0)\varphi')^jBE_h
\\=e^{i\beta}(U(t_0)\varphi')^jB(\varphi''U(-t_0))^{j-1}\varphi'' E_h
+\mathcal O(h^\infty \mathcal N(E_h))_{L^2}
\\=(U(t_0)\varphi')^jB(\varphi'' U(-t_0))^{j-1}\varphi'' U(-t_0)\varphi_{t_0}E_h
+\mathcal O(h^\infty \mathcal N(E_h))_{L^2}.\qedhere
\end{gathered}
$$
\end{proof}
%
%

\smallsection{Second decomposition of~$\langle AE_h,E_h\rangle$}
We now analyse the terms of~\eqref{e:iterated-estimate-2}, reducing
$\langle AE_h,E_h\rangle$ to an expression depending on the `outgoing'
part $E^0_h$ of the plane wave (see~\eqref{e:e-h-decomposition}), with
remainder estimated by the classical escape rate for up to twice the
Ehrenfest time.

We will use Lemma~\ref{l:analysis2-2}; since it only applies
to pseudodifferential operators microlocalized inside the set $\mathcal E_{\eps_e}$
from~\eqref{e:e-e-e}, we take an operator
\begin{equation}
  \label{e:x0}
X_0\in\Psi^{\comp}(M),\
\WFh(X_0)\subset \mathcal E_{\eps_e},\
X_0=1\text{ near }S^*M\cap\supp\varphi_{t_0},
\end{equation}
compactly supported inside $K_{t_0}$.
By~\eqref{e:e-eq-h} and the elliptic estimate (Proposition~\ref{l:elliptic}), 
we have
\begin{equation}
  \label{e:x0-eq}
\varphi_{t_0} E_h=X_0\varphi_{t_0}E_h+\mathcal O(h^\infty \mathcal N(E_h))_{L^2}
=\varphi_{t_0}X_0E_h+\mathcal O(h^\infty \mathcal N(E_h))_{L^2}.
\end{equation}
Same is true if we replace $E_h$ by $\chi_0E_h^0$, as by~\as{A4} and
the fact that $|\partial_m\phi_\xi|_g=1$, we have $\WFh(\chi_0E_h^0)\subset
S^*M$.  We also recall that $\WFh(A)\subset \mathcal E_{\eps_e}$.

We start by estimating the first term on the right-hand side of~\eqref{e:iterated-estimate-2}
for $l$ up to twice the Ehrenfest time, in terms of the classical
escape rate:

%
%
\begin{prop}\label{l:analysis2-3} 
There exists a constant $C$ such that for $0\leq l\leq 2l_e$, we have
\begin{equation}
  \label{e:analysis2-3}
h^{-1}\|\langle A(\varphi U(t_0))^l\varphi E_h,E_h\rangle\|_{L^1_{\xi,\lambda}
(\plM\times [1,1+h])}\leq C\mu_L(\mathcal T(lt_0))
+\mathcal O(h^\infty).
\end{equation}
\end{prop}
\begin{proof}
We write $l=l_1+l_2$, where $0\leq l_1,l_2\leq l_e$; then
$$
\langle A(\varphi U(t_0))^l\varphi E_h,E_h\rangle
=\langle (\varphi U(t_0))^{l_1}\varphi E_h,(U(-t_0)\varphi)^{l_2}A^* E_h\rangle.
$$
Now, by~\eqref{e:analysis2-2.1}
$$
\begin{gathered}
e^{il_1\beta}(\varphi U(t_0))^{l_1}\varphi E_h
=e^{il_1\beta}(\varphi U(t_0))^{l_1}X_0\varphi E_h+\mathcal O(h^\infty \mathcal N(E_h))_{L^2}\\
=B_l^1 E_h+\mathcal O(h^\infty \mathcal N(E_h))_{L^2},
\end{gathered}
$$
where
$$
B_l^1=(\varphi U(t_0))^{l_1}X_0\varphi (U(-t_0)\varphi_1)^{l_1}.
$$
Similarly, by~\eqref{e:analysis2-2.1-rev} (recalling that $A=\varphi A\varphi$)
$$
e^{-i l_2\beta}(U(-t_0)\varphi)^{l_2}A^* E_h=B_l^2E_h+\mathcal O(h^\infty \mathcal N(E_h))_{L^2},
$$
where
$$
B_l^2=(U(-t_0)\varphi)^{l_2}A^* (\varphi_1U(t_0))^{l_2}\varphi_{t_0}.
$$
Put $B_l=(B_l^2)^*B_l^1$; recalling~\eqref{e:n-e-h-bdd}, it is then enough to show that
\begin{equation}
  \label{e:analysis2-3.1}
h^{-1}\|\langle B_l E_h,E_h\rangle\|_{L^1_{\xi,\lambda}(\plM\times [1,1+h])}
\leq C\mu_L(\mathcal T(lt_0))+\mathcal O(h^\infty).
\end{equation}
Now, by Proposition~\ref{l:ehrenfest}, the operator $B_l^1$ is a compactly supported element
of $\Psi^{\comp}_{\rho_e}$ (modulo an $\mathcal O(h^\infty)_{L^2\to L^2}$ remainder
which we will omit), and it is microsupported, in the sense of Definition~\ref{d:microlocal-vanishing},
inside the set $g^{-l_1t_0}(\{\varphi_1\neq 0\})$ (here
we only use that $\supp\varphi\subset \{\varphi_1\neq 0\}$). Similarly, $B_l^2\in\Psi^{\comp}_{\rho_e}$
is microsupported inside $g^{l_2 t_0}(\{\varphi_1\neq 0\})$. Therefore, $B_l$
is microsupported on the set
$$
\mathcal S_l=g^{-l_1t_0}(\{\varphi_1\neq 0\})\cap g^{l_2 t_0}(\{\varphi_1\neq 0\}).
$$
Note also that $B_l$ is compactly supported independently of $l$.

Now, by taking the convolution of the indicator function of an
$h^{\rho_e}$ sized neighborhood of $\mathcal S_l$ with an
appropriately rescaled cutoff function, we can construct a compactly
supported operator $\widetilde B_l\in\Psi^{\comp}_{\rho_e}$ such that
$B_l=\widetilde B_l^* B_l+\mathcal O(h^\infty)_{\Psi^{-\infty}}$ and
$\widetilde B_l$ is microsupported inside an $\mathcal O(h^{\rho_e})$
sized neighborhood $\widetilde{\mathcal S}_l$ of $\mathcal S_l$.
Using~\eqref{e:t-e}, \eqref{e:rho-e}, and the estimate
on the Lipschitz constant of the flow given by~\eqref{e:lambda-max},
we see that for $(\widetilde m,\tilde \nu)\in \widetilde{\mathcal S}_l\cap S^*M$,
there exists $(m,\nu)\in \mathcal S_l\cap S^*M$ such that
$d((\widetilde m,\tilde\nu),(m,\nu))\leq Ch^{\rho_e}$ and 
for $\Lambda'_0>\Lambda''_0>(1+2\varepsilon_e)\Lambda_{\max}$,
$$
d(g^{l_1t_0}(\widetilde m,\tilde \nu),g^{l_1t_0}(m,\nu))\leq Ce^{l_1t_0\Lambda''_0}h^{\rho_e}
\leq Ce^{t_e\Lambda''_0}h^{\rho_e}
\leq Ce^{-t_e(\Lambda'_0-\Lambda''_0)}
$$
is bounded by some positive power of $h$. Here $d$ denotes some smooth distance function on $T^*M$.
Same is true if we replace $g^{l_1t_0}$ with $g^{-l_2t_0}$; therefore, if the compact set $K_0$ used
in the definition~\eqref{e:T-t} of $\mathcal T(t)$ is chosen large enough, we have
\begin{equation}
  \label{e:tilde-s-l}
\widetilde{\mathcal S_l}\cap S^*M\subset g^{l_2t_0}(\mathcal T(lt_0)).
\end{equation}
Using the Cauchy--Schwartz inequality and~\eqref{e:hs-est},
we bound the left-hand side of~\eqref{e:analysis2-3.1} by
$$
\begin{gathered}
h^{-1}\|\langle B_l E_h,E_h\rangle\|_{L^1_{\xi,\lambda}}
\leq h^{-1}\|\langle B_l E_h,\widetilde B_l E_h\rangle\|_{L^1_{\xi,\lambda}}+\mathcal O(h^\infty)\\
\leq h^{-1}\|B_lE_h\|_{L^2(M)L^2_{\xi,\lambda}}\cdot \|\widetilde B_lE_h\|_{L^2(M)L^2_{\xi,\lambda}}+\mathcal O(h^\infty)\\
\leq C(h^{n/2}\|B_l\Pi_{[1,1+h]}\|_{\HS})(h^{n/2}\|\widetilde B_l\Pi_{[1,1+h]}\|_{\HS})+\mathcal O(h^\infty).
\end{gathered}
$$
It remains to use~\eqref{e:h-s-estimate-2} (or rather its adjoint). Indeed, both $B_l$ and
$\widetilde B_l$ are bounded in $\Psi^{\comp}_{\rho_e}$ uniformly in $l$, and they
are microsupported in $\widetilde{\mathcal S}_l$; therefore, by~\eqref{e:tilde-s-l}
$$
h^{-1}\|\langle B_l E_h,E_h\rangle\|_{L^1_{\xi,\lambda}}\leq C\mu_L(\widetilde{\mathcal S}_l\cap S^*M)
+\mathcal O(h^\infty)\leq C\mu_L(\mathcal T(lt_0))+\mathcal O(h^\infty).\qedhere 
$$
\end{proof}
%
%
As for the sum in~\eqref{e:iterated-estimate-2}, we have the following
%
%
\begin{prop}\label{l:analysis2-35} 
For $1\leq j\leq 2 l_e$, we have
\begin{equation}
  \label{e:analysis2-35}
\begin{gathered}
e^{ij\beta}\langle A(\varphi U(t_0))^j(1-\varphi)\varphi_{t_0}\chi_0 E^0_h,
E_h\rangle=\langle \widetilde A^j \chi_0E_h^0,\chi_0E_h^0\rangle
+\mathcal O(h^\infty \mathcal N(E_h)^2),\\
\widetilde A^{j}:=\varphi_{t_0}(1-\varphi_2)(U(-t_0)\varphi_1)^jA(\varphi U(t_0))^j(1-\varphi)\varphi_{t_0},
\end{gathered}
\end{equation}
uniformly in $\xi\in\plM$ and $\lambda\in [1,1+h]$.
\end{prop}
\begin{proof}
Since $A=\varphi A\varphi$, we can replace $E_h$ by $\varphi_1 E_h$ on
the left-hand side of~\eqref{e:analysis2-35}.  Writing
down~\eqref{e:iterated-estimate-0} for $\varphi_1$ in place of
$\varphi$ and using $\varphi_2$ in place of $\varphi_1$ in the
splitting $\varphi_{t_0}=\varphi_1+(1-\varphi_1)\varphi_{t_0}$ in the
last step, we get
\begin{equation}
  \label{e:analysis2-35.1}
\begin{gathered}
\varphi_1 E_h=e^{ij\beta}(\varphi_1 U(t_0))^j\varphi_2 E_h
+e^{ij\beta}(\varphi_1 U(t_0))^j (1-\varphi_2)\varphi_{t_0}E_h
\\
+\sum_{k=1}^{j-1}e^{ik\beta}
(\varphi_1 U(t_0))^k(1-\varphi_1)\varphi_{t_0}E_h
+\mathcal O(h^\infty \mathcal N(E_h))_{L^2}.
\end{gathered}
\end{equation}
We now substitute~\eqref{e:analysis2-35.1} into the left-hand side of~\eqref{e:analysis2-35}.
The first term gives, after using~\eqref{e:x0-eq}
to replace $\varphi_{t_0}\chi_0E^0_h$ by $X_0\varphi_{t_0}\chi_0 E^0_h$
and $\varphi E_h$ by $X_0\varphi E_h$ 
$$
\langle A(\varphi U(t_0))^j (1-\varphi)\varphi_{t_0}\chi_0 E^0_h,(\varphi_1 U(t_0))^j\varphi_2 E_h\rangle
=\langle B_0 \chi_0 E^0_h,E_h\rangle+\mathcal O(h^\infty \mathcal N(E_h)^2),
$$
where
$$
B_0=\varphi_2X_0^*(U(-t_0)\varphi_1)^j A(\varphi U(t_0))^j (1-\varphi)X_0\varphi_{t_0}=\mathcal O(h^\infty)_{L^2\to L^2}
$$
by~\eqref{e:analysis2-2.3}, as $\supp\varphi_2\cap \supp(1-\varphi)=\emptyset$.

Next, we use Proposition~\ref{l:analysis2-1} to write the second term of~\eqref{e:analysis2-35.1} as
$$
e^{ij\beta}(\varphi_1 U(t_0))^j(1-\varphi_2)\varphi_{t_0}\chi_0E^0_h+\mathcal O(h^\infty \mathcal N(E_h))_{L^2};
$$
therefore, this term gives the right-hand side of~\eqref{e:analysis2-35}.

It remains to estimate the contribution of each term of the sum in~\eqref{e:analysis2-35.1},
which we can write, using~\eqref{e:x0-eq},
as $e^{i(j-k)\beta}\langle B_k\chi_0 E^0_h,\chi_0E^0_h\rangle+\mathcal O(h^\infty \mathcal N(E_h)^2)$,
with
$$
B_k=\varphi_{t_0}X_0^*(1-\varphi_1)(U(-t_0)\varphi_1)^k A(\varphi U(t_0))^j(1-\varphi)X_0\varphi_{t_0}.
$$
We need to show that $\|B_k\|_{L^2\to L^2}=\mathcal O(h^\infty)$
for $1\leq k<j$. For that, we consider two cases.
First, assume that $k\leq l_e$. Then we have
\begin{equation}
  \label{e:interm-1}
\varphi_{t_0} X_0^*(1-\varphi_1)(U(-t_0)\varphi_1)^kA(\varphi U(t_0))^k\varphi=\mathcal O(h^\infty)_{L^2\to L^2},
\end{equation}
as the supports of $1-\varphi_1$ and $\varphi$ do not intersect and
the operator in between them is a compactly supported element of
$\Psi^{\comp}_{\rho_e}$ (modulo an $\mathcal O(h^\infty)_{L^2\to
L^2}$ remainder which we will omit). Since $B_k$ is obtained by
multiplying the left-hand side of~\eqref{e:interm-1} on the right by
$U(t_0)(\varphi U(t_0))^{j-1-k}(1-\varphi)X_0\varphi_{t_0}$, it is
also $\mathcal O(h^\infty)_{L^2\to L^2}$.

Now, assume that $k\geq l_e$. Take $\widetilde\varphi_1\in C_0^\infty(M)$
equal to 1 near $\supp\varphi_1$ and such that $|\widetilde\varphi_1|\leq 1$ everywhere.
We write by~\eqref{e:analysis2-2.2} and its adjoint,
$$
\begin{gathered}
U((k-l_e)t_0)B_k U(-(j-l_e)t_0)=B_k^1B_k^2B_k^3+\mathcal O(h^\infty)_{L^2\to L^2},\\
B_k^1=(\widetilde\varphi_1 U(t_0))^{k-l_e}\varphi_{t_0}X_0^*(1-\varphi_1)(U(-t_0)\varphi_1)^{k-l_e},\\
B_k^2=(U(-t_0)\varphi_1)^{l_e}A(\varphi U(t_0))^{l_e},\\
B_k^3=(\varphi U(t_0))^{j-l_e}(1-\varphi)X_0\varphi_{t_0}(U(-t_0)\varphi_1)^{j-l_e}.
\end{gathered}
$$
Now all $B_k^i$, $i=1,2,3$, are compactly supported members of
$\Psi^{\comp}_{\rho_e}$.  Let $U_1,U_2$ be two bounded open sets such
that $\supp(\varphi_{t_0}(1-\varphi_1))\subset U_1$,
$\supp\varphi\subset U_2$, and $U_1\cap U_2=\emptyset$.
Since~$k-l_e>j-l_e$ and by Proposition~\ref{l:ehrenfest}, the operator
$B_k^1$ is microsupported, in the sense of
Definition~\ref{d:microlocal-vanishing}, on the set
$g^{-(k-l_e)t_0}(U_1)$, while $B_k^3$ is microsupported on the set
$g^{-(k-l_e)t_0}(U_2)$; since these two sets do not intersect, we get
$B_k=\mathcal O(h^\infty)_{L^2\to L^2}$, finishing the proof.
\end{proof}
%
%
Combining~\eqref{e:iterated-estimate-2} with~\eqref{e:analysis2-3}
and~\eqref{e:analysis2-35}, we have finally proved
%
%
\begin{prop}
  \label{l:analysis2-main}
For $0\leq l\leq 2 l_e$ and $A\in \Psi^{\comp}$
microlocalized inside the set~$\mathcal E_{\varepsilon_e}$ defined in~\eqref{e:e-e-e}, we have 
\begin{equation}
  \label{e:analysis2-main}
\begin{gathered}
\langle A E_h,E_h\rangle
=\sum_{j=1}^l\langle \widetilde A^j\chi_0 E^0_h,\chi_0 E^0_h\rangle
+\mathcal O\big(h\mu_L(\mathcal T(lt_0))+h^\infty\big)_{L^1_{\xi,\lambda}(\plM\times [1,1+h])},\\
\widetilde A^{j}:=\varphi_{t_0}(1-\varphi_2)(U(-t_0)\varphi_1)^jA(\varphi U(t_0))^j(1-\varphi)\varphi_{t_0}.
\end{gathered}
\end{equation}
Here $l_e$ is defined in~\eqref{e:l-e}
and $\mathcal T(t)$ in~\eqref{e:T-t}.
\end{prop}
Note that the sum on the right-hand side corresponds to the 
$A_1^{-t}$ term in \eqref{e:analysis-1}, while the remainder contains 
both the $A_0^{-t}$ term (dealt with in Proposition~\ref{l:analysis2-3}) and the remainder in \eqref{e:analysis-1}.
We have kept the $\mathcal O(h^\infty)$ remainder to include the nontrapping case.
%
%

\smallsection{Proof of~\eqref{e:convergence-2}}
Put $l$ equal to the number $l_e$ defined in~\eqref{e:l-e}.
By~\eqref{e:analysis2-main}, it is enough to approximate the terms
$\langle\widetilde A^j\chi_0 E^0_h,\chi_0E^0_h\rangle$.  This is done
by the following improvement of~\eqref{e:convergence-int}, relying on the Lagrangian structure of
$E^0_h$ and featuring the interpolated escape rate $r(h,\Lambda)$
from~\eqref{defofrh}:
%
%
\begin{prop}\label{l:analysis2-4}
Put $l=l_e$ given by \eqref{e:l-e}, and $r(h,\Lambda)$ defined in \eqref{defofrh}. 
Then the sum on the right-hand side of~\eqref{e:analysis2-main} is approximated as follows:
\begin{equation}
  \label{e:analysis2-4}
\sum_{j=1}^l \langle \widetilde A^j\chi_0 E^0_h,\chi_0E^0_h\rangle=
\int_{S^*M} \sigma(A)\,d\mu_\xi+\mathcal O(hr(h,2\Lambda_0))_{L^1_{\xi,\lambda}(\plM
\times [1,1+h])}.
\end{equation}
\end{prop}
\begin{proof}
By Proposition~\ref{l:ehrenfest}, the operator $\widetilde A^j$ is
compactly supported and lies in $\Psi_{\rho_j}^{\comp}$, modulo an
$\mathcal O(h^\infty)_{L^2\to L^2}$ remainder, where
$$
\rho_j={jt_0\over t_e}\rho_e, \quad \textrm{ with }\, jt_0/t_e\leq 1+o(1)
$$
with~$t_e$ and $\rho_e$ defined in~\eqref{e:t-e} and~\eqref{e:rho-e}, respectively.
Next, $\widetilde A^j$ is microsupported, in the sense of Definition~\ref{d:microlocal-vanishing},
in the set
$$
\mathcal Q_j:=g^{t_0}(\{\varphi_1\neq 0\})\cap g^{jt_0}(\{\varphi_1\neq 0\}).
$$
If the set $K_0$ from the definition~\eqref{e:T-t} of $\mathcal T(t)$ is large enough, then
 $\mathcal Q_j\cap S^*M\subset g^{jt_0}(\mathcal T(jt_0))$; by the definition~\eqref{defofrh} of $r(h,\Lambda)$,
we find
\begin{equation}
  \label{e:q-bound}
h^{1-jt_0/t_e} \mu_L(\mathcal Q_j\cap S^*M)\leq r(h,2\Lambda_0).
\end{equation}
By~\as{A4} and Proposition~\ref{l:lagrangian-mul}, we have the following analogue of~\eqref{e:convergence-int-777}:
$$
\widetilde A^j\chi_0E^0_h=e^{{i\lambda\over h}\phi_\xi}\chi_0 b^0(1,\xi,m;0)\sigma(\widetilde A^j)(m,\partial_m
\phi_\xi(m))+\mathcal O(h^{1-2\rho_j})_{L^\infty}.
$$
Here and below, we also use  the fact that the seminorms of $\widetilde{A}^j$ are uniform with respect to 
$j$ in the sense stated in Proposition  \ref{l:ehrenfest}, in order to control the remainders.
Moreover, by part~2 of the same proposition, we see that $\widetilde
A^j\chi_0E^0_h$ is $\mathcal O(h^\infty)$ outside of the set of points
$m\in U_\xi$ such that $(m,\partial_m\phi_\xi(m))\in \mathcal Q_j$.
By Lemma~\ref{l:geometry-2}, $\sigma(\widetilde A^j)$ is
supported in $\supp(1-\varphi)\cap
g^{t_0}(\supp\varphi_1)\subset\mathcal{DE}_+\cap \{x<\varepsilon_1\}$,
with $\mathcal{DE}_+$ from Definition~\ref{d:directly-escape}.  Using
part~2 of Lemma~\ref{l:mu-xi-1}, we then get
\begin{equation}
  \label{e:almost-there}
\langle \widetilde A^j \chi_0 E^0_h,\chi_0 E^0_h\rangle
=\int_{S^*M} \sigma(\widetilde A^j)\,d\mu_\xi
+\mathcal O(h^{1-2\rho_j}\mu_\xi(\mathcal Q_j))+\mathcal O(h^\infty),
\end{equation}
uniformly in $\xi\in\plM$ and $\lambda\in [1,1+h]$. Now, we write by~\eqref{e:q-bound}
and Proposition~\ref{l:mu-xi-0},
$$
\begin{gathered}
\sum_{j=1}^l h^{1-2\rho_j}\|\mu_\xi(\mathcal Q_j)\|_{L^1_\xi}
=\sum_{j=1}^l h^{1-2\rho_j}\mu_L(\mathcal Q_j\cap S^*M)\\
\leq r(h,2\Lambda_0)\sum_{j=1}^l h^{(1-2\rho_e)jt_0/t_e}
=r(h,2\Lambda_0)\sum_{j=1}^l e^{-2\Lambda_0(1-2\rho_e)jt_0}
\leq Cr(h,2\Lambda_0).
\end{gathered}
$$
It remains to sum up the integrals in~\eqref{e:almost-there}.  We have
by Proposition~\ref{l:ehrenfest}, bearing in mind that
$\varphi\varphi_1=\varphi$, $(1-\varphi)(1-\varphi_2)=1-\varphi$,
$A=\varphi A\varphi$, and
$d_g(\supp(1-\varphi_{t_0}),\supp\varphi)>t_0$,
$$
\sigma(\widetilde A^j)=(\sigma(A)\circ g^{-jt_0})(1-\varphi)\prod_{k=1}^{j-1}\varphi\circ g^{-kt_0}.
$$
By part~1 of Lemma~\ref{l:mu-xi-1}, we write
$$
\begin{gathered}
\sum_{j=1}^l\int_{S^*M}\sigma(\widetilde A^j)\,d\mu_\xi=
\int_{S^*M} \sigma(A)\sum_{j=1}^l (1-\varphi\circ g^{jt_0})
\prod_{k=1}^{j-1}\varphi\circ g^{kt_0}\,d\mu_\xi\\
=\int_{S^*M}\sigma(A)\Big(1-\prod_{k=1}^l\varphi\circ g^{kt_0}\Big)\,d\mu_\xi.
\end{gathered}
$$
It remains to note that by Proposition~\ref{l:mu-xi-0},
$$
\int_{\plM}\int_{S^*M}|\sigma(A)|\prod_{k=1}^l\varphi\circ g^{kt_0}\,d\mu_\xi d\xi
=\int_{S^*M}|\sigma(A)|\prod_{k=1}^l\varphi\circ g^{kt_0}\,d\mu_L
=\mathcal O(\mu_L(\mathcal T(lt_0)))
$$
since the expression under the last integral is supported in $\mathcal T(lt_0)$.
\end{proof}
%
%

\subsection{Trace estimates}
  \label{s:proofs-3}

In this subsection, we prove a stronger remainder
bound~\eqref{e:convergence-3} for the case when $\langle
AE_h,E_h\rangle$ is paired with a test function in $\xi$ and obtain an
expansion of the trace of spectral projectors with a fractal
remainder~-- Theorem~\ref{asympofs_A}.

\smallsection
{Expressing $E^0_h\otimes E^0_h$ via Schr\"odinger propagators}
Our argument will be based on the
decomposition~\eqref{e:analysis2-main}. The remainder in this
decomposition is already controlled by the escape rate at twice the
Ehrenfest time $t_e$ defined in~\eqref{e:t-e}. However, in the
previous subsection (see Proposition~\ref{l:analysis2-4}), we were
only able to estimate the sum in~\eqref{e:analysis2-main} for $l$ up
to the Ehrenfest index $l_e\sim t_e/t_0$ defined in~\eqref{e:l-e}. We
therefore need a better way of writing down the Lagrangian states
$E^0_h$, when coupled with a test function in $\xi$, and such a way is
provided by (with the cutoffs allowing for representation of the
spectral measure in terms of a finite time integral)
%
%
\begin{lemm}
  \label{l:trace-1}
Let $f(\xi)\in C^\infty(\plM)$ and define for $\lambda\in (1/2,2)$,
\begin{equation}
  \label{e:pi-0-f}
\Pi^0_f(\lambda):=\int_{\plM}f(\xi)(\chi_0 E^0_h(\lambda,\xi))
\otimes (\chi_0 E^0_h(\lambda,\xi))\,d\xi.
\end{equation}
Here $\otimes$ denotes the Hilbert tensor product,
see~\eqref{e:tensor-product}. Assume that $\widetilde X_1,\widetilde
X_2\in \Psi^{\comp}(M)$ are compactly supported and the projections
$\pi_S(\WFh(\widetilde X_j))$ of $\WFh(\widetilde X_j)$ onto $S^*M$
along the radial rays in the fibers of $T^*M$ lie inside
$\mathcal{DE}_+\cap\{x\leq\varepsilon_1\}$, with $\mathcal{DE}_+$
defined in~\eqref{d:directly-escape} and~$\varepsilon_1$
from~\as{A7}. Then
$$
\widetilde X_1\Pi^0_f(\lambda)\widetilde X_2^*=(2\pi h)^n\int_{-T_0}^{T_0} e^{-i\lambda^2 s/(2h)}  U(s)B_s(\lambda)\,ds
+\mathcal O(h^\infty)_{L^2\to L^2},
$$
where $T_0>0$ depends on the support of the Schwartz kernels of $\widetilde X_j$ but
not on $h$,  $B_s(\lambda)\in\Psi^{\comp}(M)$
is compactly supported on $M$, smooth and compactly supported in $s\in (-T_0,T_0)$,
depending smoothly on $\lambda$. Moreover, if $\xi_{+\infty}$ is the function defined in~\as{G3}, then
\begin{equation}
  \label{e:pi-0-symbol}
\sigma(B_0(1))|_{S^*M}=f(\xi_{+\infty})\sigma(\widetilde X_1\widetilde X_2^*).
\end{equation}
\end{lemm}
\begin{proof}
We write
$$
\widetilde X_1\Pi^0_f(\lambda)\widetilde X_2^*
=\int_{\plM} f(\xi)(\widetilde X_1 \chi_0E^0_h(\lambda,\xi))
\otimes (\widetilde X_2\chi_0 E^0_h(\lambda,\xi))\,d\xi.
$$
By~\as{A4}, $\chi_0E^0_h(\lambda,\xi)$
is a Lagrangian distribution associated to $\lambda$ times the
Lagrangian $\Lambda_\xi$ from~\eqref{e:lambda-xi}.
By Proposition~\ref{l:lagrangian-mul}, we can write
$$
\widetilde X_j \chi_0E^0_h(\lambda,\xi)(m)=e^{{i\lambda\over h}\phi_\xi(m)}
b_j(\lambda,\xi,m;h)+\mathcal O(h^\infty)_{C_0^\infty},
$$
where $\phi_\xi$ is defined in~\as{G4} and $b_j$ is a classical symbol in $h$
smooth in $\lambda,\xi,m$ and compactly supported in $m$. The symbol $b_j$
depends on the operator $\widetilde X_j$; in fact, we can make
$\supp b_j\subset \tau^{-1}(\pi_S(\WFh(\widetilde X_j)))$, with $\tau$
defined in~\eqref{e:tau}. We then write the Schwartz kernel
of $\widetilde X_1\Pi^0_f(\lambda)\widetilde X_2^*$, modulo
an $\mathcal O(h^\infty)_{C^\infty_0}$ remainder, as
\begin{equation}
  \label{e:pi-kernel}
\widetilde\Pi(m,m';\lambda,h)=\int_{\plM} e^{{i\lambda\over h}(\phi_\xi(m)-\phi_\xi(m'))}
f(\xi)b_1(\lambda,\xi,m;h)\overline{b_2(\lambda,\xi;m',h)}\,d\xi.
\end{equation}
Now, the support of each $b_j$ in the $(m,\xi)$ variables lies in the set
$U^+_\infty$ defined in~\as{G4}. The critical points of the phase
$\lambda(\phi_\xi(m)-\phi_\xi(m'))$ are given by
$\partial_\xi\phi_\xi(m)=\partial_\xi\phi_\xi(m')$; using~\as{G6}, we see
that $h^{-n/2}\widetilde\Pi(m,m';\lambda,h)$ is a Lagrangian distribution
associated to the Lagrangian
\begin{equation}
  \label{e:Lambda-lambda}
\widetilde\Lambda_\lambda:=\{(m,\nu;m',\nu')\mid |\nu|_g=\lambda,\
\exists s\in (-T_0,T_0): g^s(m,\nu)=(m',\nu')\}.
\end{equation}
Here $T_0>0$ is large, but fixed.

Now, take some family $B_s(\lambda)\in\Psi^{\comp}(M)$ smooth and
compactly supported in $s\in(-T_0,T_0)$ and define the operator
\begin{equation}
  \label{e:pi-b}
\Pi_B(\lambda):=(2\pi h)^n\int_{-T_0}^{T_0}e^{-i\lambda^2 s/(2h)}U(s)B_s(\lambda)\,ds.
\end{equation}
Following the proof of Lemma~\ref{l:h-s-estimate}, we see that $h^{-n/2}$ times
the Schwartz kernel $\Pi_B(m,m';\lambda,h)$ of $\Pi_B(\lambda)$ is,
up to an $\mathcal O(h^\infty)_{L^2\to L^2}$ remainder,
a compactly supported and compactly microlocalized Lagrangian distribution
associated to the Lagrangian $\widetilde\Lambda_\lambda$.
Moreover, the principal symbol of $h^{-n/2}\Pi_B(m,m';\lambda,h)$ at
$(m,\nu,m',\nu')$ such that $g^s(m,\nu)=(m',\nu')$ is a nonvanishing
factor times $\sigma(B_s)(m',\nu')$. Arguing as in the proof of part~2
of Proposition~\ref{l:lagrangian-basic}, we see that we can find a family
of operators $B_s(\lambda)$ such that
$$
\widetilde\Pi(m,m';\lambda,h)=\Pi_B(m,m';\lambda,h)+\mathcal O(h^\infty)_{C^\infty_0}.
$$

It remains to check that the family $B_s(\lambda)$ can be chosen to
depend smoothly on $\lambda$ uniformly in $h$ (this is not automatic,
as multiplication by $e^{{i\over h}\psi(\lambda)}$ for some function
$\psi$ destroys this property, but does not change the Lagrangians
where our kernels are microlocalized for each $\lambda$).  For that,
it is enough to note (by Proposition~\ref{l:lagrangian-basic}) that if
we consider~$h^{-n/2}\widetilde\Pi$ and $h^{-n/2}\Pi_B$ as Lagrangian distributions in
$m,m',\lambda$, they are associated to the Lagrangian
$$
\{(m,\nu,m',\nu',\lambda,q_\lambda)\mid
|\nu|_g=\lambda,\ \exists s\in (-T_0,T_0):
g^s(m,\nu)=(m',\nu'),\ q_\lambda=-\lambda s\},
$$
where $q_\lambda$ is the momentum corresponding to $\lambda$.
For $\widetilde\Pi$, this is true as when
$\tau(m',\xi)=g^{\lambda s}(\tau(m,\xi))$, we have $\phi_\xi(m)-\phi_\xi(m')=-\lambda s$
by~\eqref{e:g6-calculation}; for $\Pi_B$,
this is seen directly from the parametrization~\eqref{e:parametrix},
keeping in mind the factor $e^{-i\lambda^2 s/(2h)}$ in the definition of $\Pi_B$.

Finally, to show the formula~\eqref{e:pi-0-symbol}, put $\lambda=1$,
take an arbitrary $Z\in\Psi^{\comp}$, and
compute the trace
\begin{equation}
  \label{e:trace-backwards}
\Tr(\widetilde X_1\Pi^0_f(1)\widetilde X_2^*Z)=(2\pi h)^n\int_{-T_0}^{T_0}e^{-is/(2h)}\Tr(U(s)B_s(1)Z)\,ds
+\mathcal O(h^\infty).
\end{equation}
The left-hand side of~\eqref{e:trace-backwards} can be computed as at
the end of Section~\ref{s:proofs-1}, using the limiting measure
$\mu_\xi$; by Proposition~\ref{l:mu-xi-0}, it is equal to the integral
of $f(\xi_{+\infty})\sigma(\widetilde X_2^*Z\widetilde X_1)$ over the
Liouville measure on $S^*M$, plus an $\mathcal O(h)$ remainder.  The
right-hand side of~\eqref{e:trace-backwards} can be computed by the
trace formula~\eqref{e:robert-trace}, and is equal to the integral of
$B_0(1)Z$ over the Liouville measure on $S^*M$, plus an $\mathcal
O(h)$ remainder. Therefore,
$$
\int_{S^*M}\sigma(Z)f(\xi_{+\infty})\sigma(\widetilde X_1\widetilde X_2^*)\,d\mu_L
=\int_{S^*M}\sigma(Z)\sigma(B_0(1))\,d\mu_L
$$
for any $Z$; \eqref{e:pi-0-symbol} follows.
\end{proof}
%
%
\smallsection{Proof of~\eqref{e:convergence-3}}
By~\eqref{e:analysis2-main}, it is enough to approximate the sum in this formula
up to twice the Ehrenfest time:
%
%
\begin{prop}\label{l:trace-2}
Fix $f\in C^\infty(\plM)$ and $A\in\Psi^{\comp}(M)$ is microlocalized
inside the set~$\mathcal E_{\varepsilon_e}$ defined in~\eqref{e:e-e-e}.
Put $l=2l_e$, where $l_e$ is defined in~\eqref{e:l-e},
and consider the following function 
\begin{equation}
  \label{e:trace-2.1}
S_f(\lambda):=\sum_{j=1}^l \int_{\plM}f(\xi)\langle
\widetilde A^j\chi_0E^0_h(\lambda,\xi),\chi_0E^0_h(\lambda,\xi)\rangle\,d\xi.
\end{equation}
If $\xi_{+\infty}(m,\nu)$ is the limit of $g^t(m,\nu)$ as $t\to +\infty$,
for $(m,\nu)\in S^*M\setminus \Gamma_-$ (see~\as{G3}), then
for $\lambda\in [1,1+h]$,
\begin{equation}
  \label{e:trace-2.2}
S_f(\lambda)=\int_{S^*M} f(\xi_{+\infty})\sigma(A)\,d\mu_L+\mathcal O(r(h,\Lambda_0)).
\end{equation}
Here $r(h,\Lambda)$ is defined in~\eqref{defofrh}. Moreover, for each $k$
\begin{equation}
  \label{e:trace-2.3}
\sup_{\lambda\in [1,1+h]}|\partial_\lambda^k S_f(\lambda)|\leq C_k h^{-k\rho_e},
\end{equation}
where $\rho_e$ is defined in~\eqref{e:rho-e}.
\end{prop}
\begin{proof}
First of all, take a compactly supported operator $\widetilde X\in
\Psi^{\comp}$ such that $\WFh(\widetilde X)\cap S^*M$ lies inside the
set $\mathcal{DE}_+\cap \{x\leq\varepsilon_1\}$ and
for $X_0$ defined in~\eqref{e:x0},
$$
\begin{gathered}
\varphi U(t_0)(1-\varphi)\varphi_{t_0}X_0(1-\widetilde X)=\mathcal O(h^\infty)_{L^2\to L^2},\\
\varphi_1 U(t_0)(1-\varphi_2)\varphi_{t_0}X_0(1-\widetilde X)=\mathcal O(h^\infty)_{L^2\to L^2}.
\end{gathered}
$$
Such an operator exists by Lemma~\ref{l:geometry-2} (it can be
easily seen that in part~1 of this lemma, $(m,\nu)$ actually
lies in the interior of $\mathcal{DE}_+$).
Then by~\eqref{e:x0-eq}, the definition~\eqref{e:analysis2-main} of~$\widetilde A^j$,
and Lemma~\ref{l:trace-1},
$$
\begin{gathered}
S_f(\lambda)
=\sum_{j=1}^l\int_{\plM} f(\xi)\langle \widetilde A^jX_0\widetilde X\chi_0 E_h^0(\lambda,\xi),
\widetilde X\chi_0 E_h^0(\lambda,\xi)\rangle\,d\xi+\mathcal O(h^\infty)
\\=\sum_{j=1}^l\Tr(\widetilde A^jX_0\widetilde X\Pi^0_f(\lambda)\widetilde X^*)
+\mathcal O(h^\infty)\\
=(2\pi h)^n\sum_{j=1}^l\int_{-T_0}^{T_0} e^{-i\lambda^2s/(2h)}\Tr(\widetilde A^j X_0U(s)B_s(\lambda))\,ds
+\mathcal O(h^\infty)
\end{gathered}
$$
for some fixed $T_0>0$ and some family $B_s(\lambda)\in\Psi^{\comp}$
smooth in $s$ and $\lambda$ and compactly supported in $s$; we can
make $B_s$ microlocalized inside the set $\mathcal E_{\eps_e}$ defined
in~\eqref{e:e-e-e}.  We will henceforth ignore the $\mathcal
O(h^\infty)$ term.

Now, take $1\leq j\leq l$ and put $j=j_1+j_2$, where $0\leq
j_1,j_2\leq l_e$, $j_2\geq 1$, and $|j_1-j_2|\leq 1$. Using the
cyclicity of the trace, we find
$$
\begin{gathered}
\Tr(\widetilde A^j X_0U(s)B_s(\lambda))
= \Tr(U(s)B_1^j B_{2,s}^j(\lambda)),\\
B_1^j := (U(-t_0)\varphi_1)^{j_1}A(\varphi U(t_0))^{j_1},\\
B_{2,s}^j(\lambda) := (\varphi U(t_0))^{j_2}(1-\varphi)\varphi_{t_0}X_0U(s)B_s(\lambda)\varphi_{t_0}(1-\varphi_2)
(U(-t_0)\varphi_1)^{j_2}U(-s).
\end{gathered}
$$
Put $\rho_j=(jt_0/t_e)\rho_e$; since $j_1,j_2\leq j/2+1$, by Proposition~\ref{l:ehrenfest}
the operator $B_1^j$ is a compactly supported element of $\Psi^{\comp}_{\rho_j/2}$
(modulo an $\mathcal O(h^\infty)_{L^2\to L^2}$ remainder which we will ignore).
Same can be said about the operator
$$
B_{2,s}^j(\lambda)=(\varphi U(t_0))^{j_2}\cdot
(1-\varphi)\varphi_{t_0}X_0 U(s)B_s(\lambda)\varphi_{t_0}(1-\varphi_2)U(-s)\cdot
(U(-t_0)\cdot U(s)\varphi_1U(-s))^{j_2}.
$$
(The operator $U(s)\varphi_1 U(-s)$ is not pseudodifferential because
$\varphi_1$ is not compactly microlocalized, but this can be easily
fixed by taking $\widetilde X_0\in \Psi^{\comp}$ equal to the identity
on a sufficiently large compact set and replacing $U(s)\varphi_1U(-s)$
by $U(s)\varphi_1\widetilde X_0 U(-s)$ in~$B_{2,s}(\lambda)$, with an
$\mathcal O(h^\infty)$ error.)  Therefore, $B_1^jB^j_{2,s}(\lambda)$
also lies in $\Psi^{\comp}_{\rho_j/2}$; moreover, it depends smoothly
on $s$ and $\lambda$, uniformly in this operator class. (In principle,
we get powers of $l\sim\log(1/h)$ when differentiating in $s$, due to
the $(U(-t_0)\cdot U(s)\varphi_1 U(s))^{j_2}$ term, but they can be
absorbed into the powers of $h$ in the
expansion~\eqref{e:robert-trace}.)

We now use the trace formula of Lemma~\ref{l:robert-trace}, writing
\begin{equation}
  \label{e:s-f-ultimate}
S_f(\lambda)=(2\pi h)^n\sum_{j=1}^l\int_{-T_0}^{T_0}e^{-i\lambda^2 s/(2h)}\Tr(U(s)B_1^jB_{2,s}^j(\lambda))\,ds.
\end{equation}
The operator $B_{2,s}(\lambda)^j$ is microsupported, in the sense of
Definition~\ref{d:microlocal-vanishing}, inside the set
$g^{-j_2t_0}(\{\varphi_2\neq 1\})\cap g^{-(j_2-1)t_0}(\{\varphi_1\neq
0\})$; by Lemma~\ref{l:geometry-2}, this set lies
inside $g^{-j_2t_0}(\mathcal{DE}_+)$ and in particular does not intersect
any closed geodesics, therefore~\eqref{e:robert-trace-condition}
holds. The estimate~\eqref{e:trace-2.3} now follows immediately
from~\eqref{e:robert-trace}.  The power $h^{-k\rho_e}$ arises because
we integrate over the energy surface $\{|\nu|_g=\lambda\}$ depending
on $\lambda$; therefore, $\partial^k_\lambda S_{f}(\lambda)$ will
involve $k$th derivatives of the full symbol of
$B_1^jB^j_{2,s}(\lambda)$ in the direction transversal to the energy
surface, which are bounded by $h^{-k\rho_j}$. The
sum~\eqref{e:s-f-ultimate} has $l\sim\log(1/h)$ terms; however, our
estimate is not multiplied by $\log(1/h)$ because one can see that the
sum of Liouville measures of the sets where these terms are
microsupported is bounded.

As for the approximation~\eqref{e:trace-2.2}, we write (note that we take $s=0$ in
$B_{2,s}$)
$$
\begin{gathered}
\sigma(B_1^j)=(\sigma(A)\circ g^{-j_1t_0})\prod_{k=1}^{j_1-1} \varphi\circ g^{-kt_0},\\
\sigma(B_{2,0}^j(\lambda))|_{S^*M}=\big((1-\varphi)\sigma(B_0(\lambda))\big)\circ g^{j_2t_0}\prod_{k=0}^{j_2-1}\varphi\circ g^{kt_0}.
\end{gathered}
$$
Since the Liouville measure is invariant under the geodesic flow,
the contribution of the principal term of~\eqref{e:robert-trace} to $S_f(\lambda)$ for
$\lambda=1$ is
$$
\int_{S^*M}\sigma(A)\sum_{j=1}^l\big((1-\varphi)\sigma(B_0(\lambda))\big)\circ g^{jt_0}
\prod_{k=1}^{j-1} \varphi\circ g^{kt_0}\,d\mu_L.
$$
Now, by~\eqref{e:pi-0-symbol}, $\sigma(B_0(1))=f(\xi_{+\infty})$
on the support of the integrated expression; recombining the terms as in the proof
of Proposition~\ref{l:analysis2-4}, we get the right-hand side of~\eqref{e:trace-2.2},
with an $\mu_L(\mathcal T(lt_0))$ remainder. The subprincipal terms (and
also the difference $S_f(\lambda)-S_f(1)$ for $\lambda\in [1,1+h]$) are estimated
using the bound on the Liouville measure of the set where $B_1^jB^j_{2,s}$ is
microsupported; arguing as in the proof of Proposition~\ref{l:analysis2-4},
we see that they are bounded by a constant times $r(h,\Lambda_0)$.
\end{proof}
%
%

\smallsection{Expansion of the trace of spectral projectors in powers of $h$}
We now use the results obtained so far to derive an asymptotic
expansion for the trace of the product of the spectral projector
$\indic_{[0,\lambda^2]}(P(h))$ with a compactly supported
pseudodifferential operator, with the remainder depending on the
classical escape rate for up to twice the Ehrenfest time. Here we denote
\begin{equation}
  \label{e:p-h}
P(h):=h^2(\Delta-c_0),
\end{equation}
with the constant $c_0$ from~\as{A1}. It will also be more convenient for us to use
the spectral parameter $s=\lambda^2$ in the following corollary and theorem
(not to be confused with the time variable $s$ used in Lemma~\ref{l:trace-1}).

We start with the following consequence of the decomposition~\eqref{e:analysis2-main},
the bound~\eqref{e:trace-2.2}, and the spectral formula~\eqref{e:spectrum-eis-2}:
%
%
\begin{corr}
  \label{estiminhbox}
Take $\Lambda_0>\Lambda_{\max}$, with $\Lambda_{\max}$ defined
in~\eqref{lamax}, and let $\mathcal T(t)$ be defined
in~\eqref{e:T-t}. For $\eps>0$, let $\varphi\in
C_0^\infty((1-\eps,1+\eps))$ equal to $1$ near $[1-\eps/2,1+\eps/2]$
and for $s\in \mathbb R$, define $\varphi_s:=\varphi\cdot
\indic_{(-\infty,s]}$. If $\eps>0$ is small enough, then for each compactly
supported $A \in \Psi^0(M)$, there exist some functions
$S_h(s),Q_h(s)$ and some constants $C>0$, $C_k>0$ such that for all
$s\in\rr$ and all $k\in\nn$
\begin{equation}
  \label{e:estiminhbox}
\begin{gathered}
\Tr\big(A\varphi_s(P(h))\big)=S_h(s)+Q_h(s),\
|\pl_{s}^kS_h(s)|\leq C_kh^{-n-1-k/2},\\
 |Q_h(s+u)-Q_h(s)|\leq Ch^{-n}\mu_L\bigg(\mathcal T\bigg(\frac{|\log h|}{\Lambda_0}\bigg)\bigg)
+\mathcal O(h^\infty)
\text{ for }u\in [0,h].
\end{gathered}
\end{equation}
\end{corr}
\begin{proof}
By~\eqref{e:spectrum-eis-2},
$$
\Tr\big(A\varphi_s(P(h))\big)=(2\pi h)^{-n-1}
\int\limits_{\sqrt{1-\eps}}^{\sqrt{1+\eps}} \lambda^n f_\Pi(\lambda/h)\varphi_s(\lambda^2)\int\limits_{\plM}
\langle A E_h(\lambda,\xi),E_h(\lambda,\xi)\rangle \,d\xi d\lambda.
$$
By Proposition~\ref{l:elliptic-our}, we may assume that $A\in\Psi^{\comp}$.
Now, note that the decomposition~\eqref{e:analysis2-main} (with $l=2l_e$) is still valid in
any $\mathcal O(h)$ sized interval inside $(\sqrt{1-\eps},\sqrt{1+\eps})$, if
 $\eps$ is small enough. More precisely, if we write
$$
S_h(s):=(2\pi h)^{-n-1}\int\limits_{\sqrt{1-\varepsilon}}^{\sqrt{1+\varepsilon}}
\lambda^nf_\Pi(\lambda/h)\varphi_s(\lambda^2)S_1(\lambda)\,d\lambda,
$$
where $S_1(\lambda)$ is defined by~\eqref{e:trace-2.1} with
$f(\xi)\equiv 1$, then we have the expansion~\eqref{e:estiminhbox}
with $Q_h(s)$ satisfying the required bound. To estimate the
derivatives of $S_h(s)$, it now suffices to use the
bound~\eqref{e:trace-2.3}, noting that it is valid for
$|\lambda^2-1|<\varepsilon$ if $\eps$ is small enough.
\end{proof}
%
%

Using the last corollary, we can show the following trace decomposition
with a fractal remainder, the proof of which is based on a Tauberian argument:
%
%
\begin{theo}\label{asympofs_A}
Let $P(h)$ be defined in~\eqref{e:p-h}, let $a\in S^{0}(M)$ be compactly
supported 
 and $A={\rm Op}_h(a)\in
\Psi^{0}(M)$ a compactly supported  quantization. Then there exist some smooth differential 
operators $L_j$ of order $2j$ on $T^*M$, depending on the quantization procedure ${\rm
Op}_h$, with $L_0=1$, such that for any compact interval $I\subset (0,\infty)$, all $s\in I$, all $h>0$ 
small, and all $N\in\nn$ 
\[
\begin{split}
{\rm Tr}\big(A\indic_{[0,s]}(P(h))\big)= &(2\pi h)^{-n-1}\sum_{j=0}^{N} h^j\int\limits_{|\nu|_g^2\leq s}L_ja\, d\mu_\omega
+h^{-n}\mc{O}\big(\mu_L(\mc{T}(\Lambda_0^{-1}|\log h|))+h^{N}\big)
\end{split}
\]
where the remainder is uniform in $s$. Here $\mu_\omega$ is the
standard volume form on $T^*M$; we have
$\mu_\omega=\omega_S^{n+1}/(n+1)!$, where $\omega_S$ is the symplectic
form.
\end{theo}
\begin{proof} 
By rescaling $h$, it suffices to prove the result for $|s-1|\leq
\eps/2$ where $\eps>0$ is obtained in Corollary~\ref{estiminhbox}, we
can thus assume $|s-1|\leq \eps/2$.Ê Let $\varphi_s$ be defined as in
Corollary~\ref{estiminhbox}, and $\psi\in C_0^\infty((-1+\varepsilon/2,1-\varepsilon/2))$ such
that $\psi+\varphi=1$ on $[0,1+\eps/2]$. For
$s\in(1-\eps/2,1+\eps/2)$, one has
$\indic_{[0,s]}(P(h))=\varphi_s(P(h))+\psi(P(h))$ and it suffices to study
the expansion in $h$ of $ \sigma_{A,h}(s)$ and $\Tr(A\psi(P(h)))$
where
\begin{equation}
  \label{sigmaAh}
\sigma_{A,h}(s):=\Tr\big(A\varphi_s(P(h))\big)=\Tr\big(A\varphi_s(P(h))\chi\big).
\end{equation}  
if $\chi\in C_0^\infty(M)$ is such that $A=\chi A\chi$.  Since $A$ is
compactly supported, one can use the functional calculus of
Helffer--Sj\"ostrand \cite[Chapters~8-9]{d-s} to deduce that
$A\psi(P(h))\chi\in \Psi^{\comp}(M)$ is a compactly supported
and microsupported pseudodifferential operator\footnote{An alternative method in the Euclidean near infinity setting is the
functional calculus of Helffer--Robert \cite{HeRo}.}. Its trace 
has a complete expansion in powers of $h$ (see \cite[Th 9.6]{d-s}):
\begin{equation}
\label{traapsi}
\Tr(A\psi(P(h))\chi)=(2\pi h)^{-n-1}\sum_{j=0}^Nh^j\int_{T^*M}L''_ja \,d\mu_\omega+\mc{O}(h^{-n+N})
\end{equation} 
where $L''_j$ are some differential operators of order $2j$ and $L''_0a(m,\nu)=a(m,\nu)\psi(|\nu|_g^2)$. 

Let us now analyze the function $\sigma_{A,h}$. This is a smooth
function of $s>0$ by the smoothness assumption on the $E_h(\la,\xi)$
in $\la$, it is constant in $s$ for $|1-s|>\varepsilon$, and we know
that $\sigma_{A,h}(s)=\mc{O}(h^{-n-1})$ uniformly in $s$ by
Lemma~\ref{l:h-s-estimate}. Let $\theta(s)\in \mathscr{S}(\rr)$
be a Schwartz function such that $\hat{\theta}\in
C_0^\infty(-\eta,\eta)$ for some small $\eta>0$ and
$\hat{\theta}(t)=1$ near $t=0$, and let
$\theta_h(s)=h^{-1}\theta(s/h)$. We write
\[
\sigma'_{A,h}(s):=\pl_s \sigma_{A,h}(s)=\Tr(A\varphi(P(h))d\Pi_{s}(P(h))\chi)\in C_0^\infty((0,\infty))
\] 
where $d\Pi_s(P(h))$ is the spectral measure of $P(h)$.  The operator
$A\varphi(P(h))d\Pi_{s}(P(h))\chi $ has a smooth compactly supported
kernel and is trace class. We clearly have $\sigma'_{A,h}\star
\theta_h\in \mathscr{S}(\rr)$ and by a simple computation, its
semi-classical Fourier transform is given by
\[
\int_{\rr}e^{-i\frac{t}{h}s}\sigma'_{A,h}\star \theta_h(s)ds=
\Tr(A\varphi(P(h))e^{-i\frac{t}{h}P(h)})\hat{\theta}(t)
\]
and thus
\[
\sigma'_{A,h}\star \theta_h(s)=\frac{1}{2\pi h}\int_{\rr}e^{i\frac{s}{h}t} \Tr(A\varphi(P(h))e^{-i\frac{t}{h}P(h)})\hat{\theta}(t)dt.
\]
Now we can apply Lemma~\ref{l:robert-trace} with
$B_s=\demi e^{ic_0hs/2}\hat{\theta}(-s/2)A\varphi(P(h))$; the
condition~\eqref{e:robert-trace-condition} is satisfied because
$\hat\theta$ is supported in a small neighborhood of zero. This shows
that, as $h\to 0$, we have the expansion (locally uniformly in $s$)
\[
\sigma'_{A,h}\star \theta_h(s)=(2\pi h)^{-n-1}\Big(\sum_{j=0}^{N} h^j\int_{S^*M}\widetilde L_jb(m,\sqrt{s}\nu)d\mu_L(m,\nu)+
\mc{O}(h^{N+1})\Big)
\]
for all $N\in\nn$, where $b$ is a symbol such that ${\rm
Op}_h(b)=\demi A\varphi(P(h))+\mc{O}(h^\infty)$, $\widetilde L_j$ are
differential operators of order $2j$ on $T^*M$, smooth in $\sqrt{s}$,
with $\widetilde L_0=s^{\frac{n-1}{2}}$.  In particular, one has
$\widetilde L_0b(m,\sqrt{s}\nu)=\demi
s^{\frac{n-1}{2}}a(m,\sqrt{s}\nu)\varphi(s|\nu|_g^2)+\mc{O}(h)$.  Notice that $b$
is supported in $\{(m,\nu)\in T^*M\mid |\nu|_g^2\in \supp\varphi\}$ thus
$\widetilde L_jb(m,\sqrt{s}\nu)$ is smooth in $s\in\rr$ when $(m,\nu)\in
S^*M$. Since $\sigma_{A,h}(s)$ is bounded in $s$, the convolution
$\sigma_{A,h}\star \theta_h(s)$ is well defined (as an element in
$L^\infty(\rr)$) and we have for all $N$ and for all $s\leq 2$
\[
\begin{split}
\sigma_{A,h}\star\theta_h(s)= (2\pi h)^{-n-1}\Big(\sum_{j=0}^{N} h^j\int_{0}^s
\int_{S^*M}\widetilde L_jb(m,\sqrt{u}\nu)d\mu_L(m,\nu)du+\mc{O}(h^{N+1})\Big). 
\end{split}
\]
We are going to show that, uniformly in $s\in\rr$,
\begin{equation}\label{needtoprove}
\sigma_{A,h}(s)-\sigma_{A,h}\star\theta_h(s)=\mc{O}(h^\infty)+\mc{O}\big(h^{-n}\mu_L(\mc{T}(\Lambda_0^{-1}|\log h|))\big),
\end{equation}
using the decomposition
$$
\sigma_{A,h}(s)=S_h(s)+Q_h(s)
$$ 
defined in \eqref{e:estiminhbox}. Since $\pl_s S_h(s)$ is a compactly supported symbol  
we get by integrating by parts $N$ times
\[
(1-\hat{\theta}(t)) \int e^{-i\frac{t}{h}s}\pl_s S_{h}(s)ds=(1-\hat{\theta}(t))\int e^{-i\frac{t}{h}s}
\bigg(\frac{h}{it}\bigg)^N\pl^{N+1}_{s}
(S_{h}(s))ds=\mc{O}\bigg(\frac{h^{N/4}}{(1+|t|)^{N}}\bigg)
\]
for all $t\in \rr$ and all $N\gg 1$. Thus, taking the Fourier transform 
we deduce that $S_h(s)-S_h\star \theta_h(s)=\mc{O}(h^{\infty})$
uniformly in $s$.  From \eqref{e:estiminhbox}, we obtain by induction that for all $s,u$  
\[
|Q_h(s+u)-Q_h(s)|\leq Ch^{-n}\bigg(1+\frac{|u|}{h}\bigg)\mu_L(\mc{T}(\Lambda_0^{-1}|\log h|))
+\mathcal O(h^\infty)
\]
then multiplying by $\theta_h(-u)$ and integrating in $u$, we obtain~\eqref{needtoprove}.

Given~\eqref{needtoprove}, we have
\begin{equation}
\label{expsigma}
\begin{gathered}
\sigma_{A,h}(s)=(2\pi h)^{-n-1}\sum_{j=0}^{N} h^j\int_{0}^s
\int_{S^*M}\widetilde L_jb(m,\sqrt{u}\nu)d\mu_L(m,\nu)du\\
+\mc{O}\big(h^{-n}\mu_L(\mc{T}(\Lambda_0^{-1}|\log h|))\big)+\mathcal O(h^{-n+N}).
\end{gathered}
\end{equation}
Since the symbol of $b$ is explicitly obtained from $a$ using Moyal
product, we can rewrite this expression with $a$ instead of $b$ and
with some new differential operators with the same properties as
$\widetilde L_j$ but supported in $\{|\nu|^2_g\in \supp\varphi\}$; using
polar coordinates $S^*M\x \rr^+_{\sqrt{u}}$ on $T^*M$, we deduce that
there exist some differential operators $L_j'$ of order $2j$ on $T^*M$
such that
\[
\int_{0}^s \int_{S^*M}\widetilde L_jb(m,\sqrt{u}\nu)d\mu_L(m,\nu)du=\int_{|\nu|^2_g\leq s}L'_ja(m,\nu)d\mu_\omega(m,\nu)
\] 
and $L'_0a(m,\nu)= \varphi(|\nu|^2_g)a(m,\nu)$.  Combining this with
\eqref{expsigma} and \eqref{traapsi}, we obtain the desired result
where $L_j$ in the statement of the Theorem corresponds now to
$L'_j+L''_j$.
\end{proof}
%
%

\section{Euclidean near infinity manifolds}
  \label{s:euclidean}

In this section, we assume that $(M,g)$ is a complete Riemannian
manifold such that there exists a compact set $K_0\subset M$ such that
for $\mathcal E:=M\setminus K_0$,
\[
(\mathcal E,g) \textrm{ is isometric to }(\rr^{n+1}\setminus B(0,R), g_{\rm eucl})
\]
where $R>0$, $B(0,R)$ is the Euclidean ball of center $0$ and radius
$R$ and $g_{\rm eucl}$ is the Euclidean metric. We will check that all
the assumptions of Section~\ref{s:general} are satisfied.

\subsection{Geometric assumptions}
  \label{s:euclidean.geometry}

We let $x\in C^\infty(M)$ be an everywhere positive function equal to
$x(m)=|m|^{-1}$ in $\mc{E}$ identified with $\rr^{n+1}\setminus
B(0,R)$, and such that $x\geq R^{-1}$ in $K_0$. (We take it instead of
the function $(1+|m|^2)^{-1/2}$ used in Section~\ref{s:general} for
the model case of $\mathbb R^{n+1}$, to simplify the calculations and
since we no longer need smoothness at zero.)  We shall use the polar
coordinates $m=\omega/ x$ in $\mc{E}$, where $\omega\in \mathbb
S^{n}$.  Assumption~\as{G1} is satisfied by taking the radial
compactification of $M$, i.e. adding the sphere at infinity: the map
$\psi :\rr^{n+1}\setminus B(0,R)\to (0,1/R)\x \mathbb S^n$ defined by
$\psi(m)=(x(m),x(m)m)$ is a diffeomorphism and the radial
compactification of $M$ is obtained by setting $\bbar{M}=M\sqcup
\plM$ where $\plM:=\mathbb S^n$, the smooth structure on
$\bbar{M}$ is the same as before on $M$ but we extend it to $\bbar{M}$
by asking that $\psi$ extends smoothly to the boundary $\plM$
and $\psi(\xi)=(0,\xi)$ if $\xi\in \plM=\mathbb S^n$ (see for
instance~\cite{Me} for more details). In other words, smooth functions
on $\bbar{M}$ are smooth functions on $M$ with an asymptotic expansion
in integer powers of $1/|m|$ to any order near infinity.

Assumption~\as{G2} is clearly satisfied for $\eps_0:=1/2R$ since the
trajectories of the geodesic flow in $x\leq \eps_0$ are simply
$g^t(m,\nu)=(m+t\nu,\nu)$. A point $(m,\nu)\in S^*M$ is directly
escaping in the forward direction in the sense of
Definition~\ref{d:directly-escape} if and only if $x(m)\leq\eps_0$ and
$m\cdot\nu\geq 0$. Now, \as{G3} is satisfied with
$\xi_{+\infty}(m,\nu)=\nu$ for $(m,\nu)\in \mathcal{DE}_+$.

For the assumption~\as{G4}, we define
$$
\begin{gathered}
\widetilde U_\infty=\{x<\eps_0\}\times \plM\subset \overline M\times\plM,\\
\phi_\xi(m)=m\cdot\xi,\ (m,\xi)\in U_\infty.
\end{gathered}
$$
Then $\tau:U_\infty\to S^*M$ from~\eqref{e:tau} maps
each $(m,\xi)\in (\mathbb R^n\setminus B(0,2R))\times \mathbb S^n$
to itself as an element of $S^*(\mathbb R^n\setminus B(0,2R))$.
Assumptions~\as{G4} and~\as{G5} follow immediately.
To see assumption~\as{G6}, we note that for $x(m),x(m')<\varepsilon_0$
and some $\xi\in \mathbb S^n$, we have $\partial_\xi\phi_\xi(m)=\partial_\xi\phi_\xi(m')$
if and only if $m-m'$ is a multiple of $\xi$.

\subsection{Distorted plane waves and analytic assumptions}
  \label{s:euclidean.analysis}

We recall a few well-known facts about scattering theory on
perturbations of $\rr^n$, we refer to~\cite{Me} for a geometric
approach and to \cite{MeZw,HaVa} in a more general setting
(asymptotically Euclidean case).  A plane wave for the flat Laplacian
on $\rr^{n+1}$ is the function, for $\la\in (1/2,2)$,
\begin{equation}
  \label{planewave}
u(\la, \xi;m):=ce^{{i\la\over h} m\cdot\xi} , \quad \xi\in \mathbb S^{n},\
m\in\rr^{n+1},\ c\in \cc.
\end{equation}
This is a semiclassical Lagrangian distribution, its oscillating
phase has level sets given by planes orthogonal to $\xi$. The
continuous spectrum of the Laplacian $\Delta$ associated to the metric
$g$ is the half-line $[0,\infty)$. We will take the resolvent of $h^2\Delta$ to be the $L^2$-bounded operator
\begin{equation}
  \label{e:r-h-eucl}
R_h(\la):=(h^2\Delta-\la^2)^{-1} \textrm{ in } {\rm Im}(\la)<0.
\end{equation}
This admits a continuous extension to $\{\la\not=0, {\rm Im}(\la)\leq 0\}$ as a bounded operator from 
$L^2_{\rm comp}$ to $L^2_{\rm loc}$.
For $\la>0$ we call $R_h(\la)$ the incoming resolvent and $R_h(-\la)$ the outgoing resolvent.
For $\la>0$, $h>0$, and $m\in M$ fixed, the Schwartz kernel $R_h(\la;m,m')$ of the incoming resolvent 
$R_h(\la)$ has an asymptotic expansion along the lines $m'\in m'_0+\rr\xi$ directed by $\xi\in S^n$ given by
$$
R_h(\la;m,\xi/x')\sim_{x'\to 0} (x')^{\ndemi}e^{-\frac{i\la}{hx'}}f(\la,\xi;m)+\mc{O}((x')^{\ndemi+1})
$$
for some smooth function $f$ and the remainder is uniform for $m$ in compact sets 
(see for example~\cite{MeZw,HaVa}). Using this expansion, we define
the \emph{distorted plane wave}  by 
\begin{equation}
  \label{defEeucl}
 E_h(\la,\xi;m):={2i\la h} 
 \Big(\frac{2\pi  h}{i\la}\Big)^\ndemi \lim_{x'\to 0}[(x')^{-n/2}e^{{i\la\over hx'}}R_h(\la; m,\xi/x')], 
\end{equation}
with $\xi\in S^n$ and $\xi/x'\in \mc{E}$. This is a smooth function of $(m,\xi)\in M\x
\mathbb S^{n}$, and in the case of $M=\rr^{n+1}$ it is given by
\eqref{planewave} with $c=1$ (see~\cite[Chapter~1]{Me}).  We shall use
the notation $E_h(\la,\xi)$ for the $C^\infty(M)$ function defined by
$m\mapsto E_h(\la,\xi;m)$ and we notice that
$(h^2\Delta-\la^2)E_h(\la,\xi)=0$ in $M$. One has
$\bbar{E_h(\la,\xi;m)}=E_h(-\la,\xi;m)$ since
$R_h(\la)^*=R_h(-\la)$ for $\lambda\in \mathbb R$, and the decomposition of the
spectral measure in terms of these functions in given as follows: by
Stone's formula, the semiclassical spectral measure is given by
\begin{equation}
  \label{dPih}
d\Pi_h(\la)=\frac{i\la}{\pi}(R_h(\la)-R_h(-\la))\,d\lambda \quad \textrm{ for }\la\in(0,\infty)
\end{equation}  
in the sense that $F(h^2\Delta)=\int_0^\infty F(\la^2)d\Pi_h(\la)$ for
any bounded function $F$; by combining this with the Green's type
formula of~\cite[Lemma~5.2]{HaVa}, we deduce that
\[
d\Pi_h(\la;m,m')=\la^n(2\pi h)^{-n-1}\int_{\mathbb S^n}E_h(\la,\xi;m)\bbar{E_h(\la,\xi;m')}\, d\xi d\lambda.
\]
Here $d\xi$ corresponds to the standard volume form on the sphere
$\mathbb S^n$. The assumptions~\as{A1} and~\as{A2} are then satisfied.
In fact, using~\cite{HaVa}, one can define distorted plane waves and
verify assumptions~\as{A1} and~\as{A2} for the more general case of
scattering manifolds. 

\smallsection{Outgoing/incoming decomposition}
We now construct the decomposition~\eqref{e:e-h-decomposition} of
$E_h$ into the outgoing and incoming parts and verify
assumptions~\as{A3}--\as{A8}. Take $\chi_0\in C^\infty(\bbar{M})$
(thus constant in $\xi$) supported in $\{x<\eps_0\}$ and equal to $1$
near $\{x\leq \eps_0/2\}$, so that assumptions~\as{A3} and~\as{A7} hold,
where we put $\varepsilon_1:=\varepsilon_0/2$. We next put
$$
E_h^0(\lambda,\xi;m):=e^{{i\lambda\over h}m\cdot\xi},\
x(m)<\eps_0,
$$
so that~\as{A4} holds with $b^0\equiv 1$ and~\as{A8} follows. We then
claim that
\begin{equation}
  \label{paramEheucl}
E_h=\chi_0 E_h^0+E_h^1,
\end{equation}
where
$$
E_h^1:=-R_h(\lambda)F_h,\
F_h(\lambda,\xi)=(h^2\Delta-\lambda^2)(\chi_0E_h^0(\lambda,\xi))
=[h^2\Delta,\chi_0]E_h^0(\lambda,\xi).
$$
We can apply $R_h(\lambda)$ to $F_h(\lambda,\xi)$ as the latter lies
in $C_0^\infty(M)$; in fact, $\supp F_h\subset \{\varepsilon_0/2<x<\varepsilon_0\}$.
To show~\eqref{paramEheucl}, note that the incoming resolvent $R_h(\la)$ satisfies
\begin{equation}
  \label{Rhlachi1}
R_h(\la)\chi_1=\chi_0R_h^0(\la)\chi_1-R_h(\la)[h^2\Delta,\chi_0]R_h^0(\la)\chi_1
\end{equation}
if $\chi_1\in C^\infty(\overline M)$ is such that
$\chi_0=1$ on $\supp(\chi_1)$ and $R_h^0(\la)$ is the incoming scattering
resolvent of the free semiclassical Laplacian $h^2\Delta$ on
$\rr^{n+1}$ (we use again the isometry $\mc{E}\simeq
\rr^{n+1}\setminus B(0,R)$). 
To obtain $E_h(\la,\xi)$ from \eqref{Rhlachi1}, we shall use definition \eqref{defEeucl}; 
consider the Schwartz kernels of the operators in \eqref{Rhlachi1}  
and multiply them by $(x')^{-\ndemi}e^{i\la\over hx'}$ in the right variable; since $\chi_1=1$ near infinity, 
one has by the remark following \eqref{defEeucl}
$$
E^0_h(\la,\xi;m)={2i\la h} 
 (\frac{2\pi  h}{i\la})^\ndemi \lim_{x'\to 0}[(x')^{-n/2}e^{{i\la\over hx'}}R^0_h(\la; m,\xi/x')\chi_1(\xi/x')].
$$
Now the Schwartz kernel $\kappa(m'',m')$ of $[h^2\Delta,\chi_0]R_h^0(\la)\chi_1$ is smooth in 
$M\x M$ by ellipticity and compactly supported in the first variable, moreover $R_h(m,m'')$ is in $L^1_{\rm loc}$
in the $m''$ variable for $m\in M$ fixed, thus we get by dominated convergence for fixed $m,\xi$
$$
\begin{gathered} 
\lim_{x'\to 0}\Big({x'}^{-\ndemi}e^{\frac{i\la}{hx'}}\int_MR_h(m,m'')\kappa(m'',\xi/x'){\rm dvol}_M(m'')\Big)=\\
\int_{M}R_h(m,m'')[h^2\Delta,\chi_0]E_h^0(\la,\xi; m''){\rm dvol}_M(m'')
\end{gathered}
$$
and by combining this with \eqref{Rhlachi1}, this proves \eqref{paramEheucl}.

\smallsection{Microlocalization of $E^1_h$}
It remains to verify assumptions~\as{A5} and~\as{A6}. By rescaling $h$
and using that $E_h(\lambda,\cdot)$ depends only on $\lambda/h$, we
may assume that $\lambda=1$.  Fix $\xi$ and take $\chi_2\in
C^\infty(\overline M)$ equal to 1 near $\{x\leq\varepsilon_0\}$, but
supported inside $\mathcal E$. Then
\begin{equation}
  \label{e:euclidean-far}
\chi_2 E^1_h=R^0_h(\lambda)F^0_h,\quad
F^0_h:=(h^2\Delta-\lambda^2)(\chi_2E^1_h)=-F_h+[h^2\Delta,\chi_2]E^1_h.
\end{equation}
The function $F^0_h$ is supported inside $\{x>\varepsilon_0/2\}$ and
$$
\|F^0_h\|_{H^{-1}_h}\leq Ch(1+\|E_h\|_{L^2(\{x\geq\varepsilon_0\})}).
$$
The free resolvent $R^0_h(\lambda)$ is bounded $H^{-1}_{h,\comp}\to L^2_{\loc}$
with norm $\mathcal O(h^{-1})$ by \cite[Proposition~2.1]{Bur}; therefore, for each compact set
$K\subset M$, there exists a constant $C_K$ such that
$$
\|E^1_h\|_{L^2(K)}\leq C_K (1+\|E_h\|_{L^2(\{x\geq\varepsilon_0\})}).
$$ 
This shows~\as{A5}, namely that the function
$$
\widetilde E^1_h:={E^1_h\over 1+\|E_h\|_{L^2(\{x\geq\varepsilon_0\})}}
$$
is $h$-tempered. To prove~\as{A6}, we use semiclassical elliptic estimate
and propagation of singularities (see for example~\cite[Section~4.1]{v}). We have
$$
(h^2\Delta-\lambda^2)\widetilde E^1_h=-\widetilde F_h,\
\widetilde F_h:={F_h\over 1+\|E_h\|_{L^2(\{x\geq\varepsilon_0\})}}.
$$
Now, $F_h$ is a Lagrangian distribution associated to $\{(m,\xi)\mid
m\in\supp(d\chi_0)\}$; therefore,
$$
\WFh(\widetilde F_h)\subset \WFh(F_h)\subset W_\xi,
$$
with $W_\xi\subset S^*M$ defined in~\eqref{e:w-xi}.

Take $(m,\nu)\in\WFh(\widetilde E^1_h)$. By the elliptic estimate,
$(m,\nu)\in S^*M$. Next, if $\gamma(t)=g^t(m,\nu)$, then by
propagation of singularities, either $\gamma(t)\in \WFh(\widetilde
F_h)\subset W_\xi$ for some $t\geq 0$ or $\gamma(t)\in\WFh(\widetilde
E^1_h)$ for all $t\geq 0$.  Now, the free resolvent $R^0_h(\lambda)$
is semiclassically incoming in the following sense: if $f$ is a
compactly supported $h$-tempered family of distributions, then for
each $(m',\nu')\in\WFh(R^0_h(\lambda)f)$, there exists $t\geq 0$ such
that $g^t(m',\nu')\in\supp f$. This can be seen for example from the
explicit formulas for $R^0_h(\lambda)$, see \cite{Me}.
By~\eqref{e:euclidean-far} and since $\supp(F^0_h)\subset
\{x>\varepsilon_0/2\}$, we see that for $(m',\nu')\in\WFh(\widetilde
E^1_h)$, we cannot have $x(m')<\varepsilon_0/2$ and $m'\cdot\nu'\geq
0$. Therefore, if $\gamma(t)\not\in W_\xi$ for all $t\geq 0$, then
$\gamma(t)$ is trapped as $t\to +\infty$; this proves~\as{A6}.

\section{Hyperbolic near infinity manifolds}
  \label{s:ah}

In this section, we verify the assumptions of Section~\ref{s:general}
for certain asymptotically hyperbolic manifolds. Let $(M,g)$ be an
$(n+1)$-dimensional asymptotically hyperbolic manifold as defined in
the introduction. It has a compactification $\bbar{M}=M\cup \plM$ and
the metric can be written in the product form~\eqref{e:k-1}:
\[
g=\frac{dx^2+h(x)}{x^2}
\]
where $x$ is a boundary defining function and $h(x)$ a smooth family
of metrics on $\plM$ defined near $x=0$.  The function $x$ putting
the metric in the form~\eqref{e:k-1} is not unique, and those
functions (thus satisfying $|d\log(x)|_g=1$ near $\plM$) are
called \emph{geodesic boundary defining functions}.  The set of such
functions parametrizes the conformal class of $h(0)$, as shown
in~\cite[Lemma~5.2]{GRL}.  The metric is called
\emph{even} if $h(x)$ is an even function of $x$, this condition is
independent of the choice of geodesic boundary defining function.
A choice of geodesic boundary defining function induces a metric on
$\plM$ by taking $h_0=h(0)=x^2g|_{T\plM}$, and therefore one has a
Riemannian volume form, denoted $d\xi$, on $\plM$ induced by the choice of
$x$.  Any other choice $\hat{x}=e^{\omega}x$ of boundary defining
function induces a volume form
\begin{equation}
  \label{measurexi}
\widehat{d\xi}=e^{n\omega_0}d\xi \quad \textrm{ where }\omega_0=\omega|_{\plM}.
\end{equation}
We will further assume that $M$ has constant sectional curvature $-1$
outside of some compact set, even though some of the assumptions of
Section~\ref{s:general} hold for general asymptotically hyperbolic
manifolds with no simplification provided by
the additional assumption on curvature~-- we will give the proofs in
higher generality where appropriate. 

\subsection{Geometric assumptions}
  \label{s:eisenstein.geometry}

Let $(M,g)$ be an asymptotically hyperbolic manifold. The
assumption~\as{G1} is satisfied. We are now going to prove a Lemma
which implies directly that the assumptions~\as{G2} and~\as{G3} are
satisfied, except that this only proves continuous dependence of
$\xi_{+\infty}$ in $(m,\nu)$ in~\as{G3}. To prove $C^1$ dependence in
a general setting, a bit more analysis would be required, but we shall
later concentrate only on cases with constant curvature near infinity,
in which case the dependence is smooth (see below).
%
%
\begin{lemm}
  \label{geodesicallyconvex}
Let $(M,g)$ an asymptotically hyperbolic manifold. Then there exists
$\eps_0>0$ such that the function $x$ satisfies~\eqref{e:kinda-convex}
and for any unit speed geodesic $\gamma(t)=(m(t),\nu(t))$ with
$x(m(0))\leq \eps_0$ and $\pl_t x(m(t))|_{t=0}\leq 0$, one has the
following: $\pl_t x(m(t)) \leq 0$ for all $t\geq0$ and $m(t)$ converges
in the topology of $\overline M$ to some point $\xi_{+\infty}\in
\plM$. More precisely, the distance with respect to the compactified
metric $\bar{g}=x^2g$ between $m(t)$ and $\xi_{+\infty}$ is bounded by
\[
d_{\bar{g}}( m(t), \xi_{+\infty})\leq Ct^{-1}.
\] 
\end{lemm}
\begin{proof} 
Consider coordinates $(m,\nu)=(x,y;\rho dx+\theta\cdot dy)$ on $T^*M$ near
the boundary $\plM=\{x=0\}$. The geodesic flow is the Hamiltonian flow
of $p/2$, where $p(m,\nu)=x^2(\rho^2+|\theta|^2_{h_m})$; if dots
denote time derivatives with respect to the geodesic flow, we get
\begin{equation}
  \label{eqHp}
\dot{x}=\rho x^2, \quad \dot{\rho}=-x^{-1} p(m,\nu)-x^2\pl_xh_{(x,y)}(\theta,\theta)/2.
\end{equation}
Since $\partial_x h_{(x,y)}$ is smooth up to $x=0$, there exists a constant $C$ such that
$$
|x^2\partial_x h_{(x,y)}(\theta,\theta)/2|\leq C x^2 h_{(x,y)}(\theta,\theta)\leq C p(m,\nu).
$$
Therefore, there exists $\varepsilon_0>0$ such that along any unit speed geodesic, we have
\begin{equation}
  \label{e:goodrho}
x\leq \varepsilon_0\Longrightarrow \dot\rho=-x^{-1}+\mathcal O(1)\leq -x^{-1}/2<0.
\end{equation}
This in particular implies~\eqref{e:kinda-convex}.

Now, let $\gamma(t)=(x(t),y(t);\rho(t),\theta(t))$ be a unit speed geodesic
and assume that $x(0)\leq\varepsilon_0$ and $\dot x(0)\leq 0$.
It follows from~\eqref{e:kinda-convex} that for $t\geq 0$,
we have $\dot x(t)\leq 0$ and thus $x(t)\leq\varepsilon_0$.
(Indeed, for each $s\geq 0$ the minimal value of $x(t)$ on
the interval $[0,s]$ has to be achieved at $t=s$.)
It remains to show that as $t\to +\infty$, $x(t)$ converges to $0$ and
$y(t)$ converges to some $\xi_{+\infty}\in \pl M$. For that, note that
by~\eqref{e:goodrho}, $\dot\rho(t)\leq -\varepsilon_0^{-1}/2$ for $t\geq 0$;
since $\dot x(0)\leq 0$, we have $\rho(0)\leq 0$ and thus
$$
\rho(t)\leq -{\varepsilon_0^{-1}\over 2}t.
$$ 
Setting $u(t):=x(t)^{-1}$, we find $\dot u(t)=-\rho(t)\geq (\varepsilon_0^{-1}/2)t$; therefore,
$$
x(t)\leq {\varepsilon_0\over 1+t^2/4}.
$$
In particular, $x(t)\to 0$ as $t\to +\infty$. Now the equation for $\dot{y}(s)$ tells us that
\[
\dot{y}_i(t)=x\sum_{k=1}^nh^{ki}_{(x,y)}x\theta_k=\mc{O}(x(t))=\mathcal O(t^{-2})
\]
and therefore $|y(t)-y(t')|\leq C/t'$ for any $t>t'>0$, which implies
$\lim_{t\to \infty}m(t)= \xi_{\infty}$ for some $\xi_\infty\in
\plM$ and $|m(t)-\xi_\infty|=\mc{O}(1/t)$.
\end{proof}
%
%
The geometric assumption~\as{G4} is a more complicated one, and we
will restrict ourselves to asymptotically hyperbolic manifolds with
constant curvature $-1$ in a neighbourhood of $\pl M$ and $x$ a
geodesic boundary defining function. Let $\xi\in \plM$, then there exists a neighborhood
$V_{\xi}$ of $\xi$ in $\bbar{M}$, and an isometric diffeomorphism
$\psi_{\xi}$ from $V_{\xi}\cap M$ into the following
neighbourhood $V_{q_0}$ of the north pole $q_0$ in the unit ball
$\mathbb{B}:=\{m\in\rr^{n+1}; |m|<1\}$ equipped with the hyperbolic metric $g_{0}$ 
\begin{equation}
  \label{Uxi0}
V_{q_0}:=\{q\in \mathbb{B}\mid |q-q_0|< 1/4\}, \quad g_{0}=4\frac{dq^2}{(1-|q|^2)^2}
\end{equation} 
where $\psi_{\xi}(\xi)=q_0$ and $|\cdot|$ denotes the
Euclidean length. 
This statement is proved for
instance in \cite[Lemma~3.1]{GuZw}. We shall choose the boundary
defining function on the ball $\mathbb{B}$ to be
\begin{equation}
  \label{definx0}
x_0=2\frac{1-|q|}{1+|q|}.
\end{equation}
and the induced metric $x_0^2g_0|_{\mathbb{S}^n}$ on $\mathbb S^n=\pl\bbar{\mathbb{B}}$ is the canonical one with curvature $+1$. The function $x_0$ can be viewed locally as a boundary defining function (through the chart $\psi_\xi$) near a point $\xi\in \pl\bbar{M}$ but in general there does not exist a global 
geodesic boundary defining function $x$ on $M$ so 
that $x=\psi_\xi^*x_0$ in a whole family of charts $V_\xi$ covering a neighbourhood of $\pl\bbar{M}$.
We define for each $p\in \mathbb S^n$ the Busemann function on $\mathbb B$
\[
\phi^{\mathbb B}_{p}(q)=\log\Big(\frac{1-|q|^2}{|q-p|^2}\Big).
\]
The geodesic trajectory $g^t(q,d \phi^{\mathbb B}_{p}(q))$ generated by the
differential $d\phi^{\mathbb B}_{p}$ converges (in the Euclidean ball topology) to
$p$ and the Lagrangian manifold
\[
\Lambda^{\mathbb B}_{p}:=\{ (q,d\phi^{\mathbb B}_{p}(q))\in S^* {\hh^{n+1}}\mid q\in {\mathbb B}\}
\]
is the stable manifold of the geodesic flow associated to $p$ on
${\mathbb B}$. The level sets of $\phi^{\mathbb B}_p$ are horospheres based at
$p$. We cover a neighbourhood of $\plM$ by finitely many
$V_{\xi_j}$ for some $\xi_j\in\plM$ and take a partition of unity
$\chi_j\in C^\infty(\pl \bbar{M})$ on $\pl \bbar{M}$ with $\chi_j$ supported in $V_{\xi_j}\cap\plM$.
Then there exists $\eps>0$
such that for all $j$ and all $\xi\in\supp \chi_j$, the set
\begin{equation}
  \label{Uxi}
U_{\xi}:=\{m\in \bbar{M}\mid d_{\bar{g}}(m,\xi)<\eps\}
\end{equation}
lies inside $V_{\xi_j}$, where $\bar{g}=x^2g$ is the compactified
metric. Put
$$
U_\infty:=\{(m,\xi)\in M\x\plM\mid m\in U_\xi\}.
$$
Define the function
\begin{equation}
  \label{definitionphixi}
\phi_\xi(m):=\sum_j \chi_j(\xi) \phi^{\mathbb B}_{\psi_{\xi_j}(\xi)}(\psi_{\xi_j}(m)),\
(m,\xi)\in U_\infty.
\end{equation}
Since $\psi_{\xi_j}$ are isometries, each function
$\phi^j_\xi(m):=\phi^{\mathbb B}_{\psi_{\xi_j}(\xi)}(\psi_{\xi_j}(m))$ is
such that $d\phi^j_\xi(m)$ is the unit covector which generates the
unique geodesic in $M$ starting at $m$, staying in $U_\xi$ for
positive times, and converging to $\xi$ (therefore, the difference of
any two functions $\phi^j_\xi$ for different $j$ is a function of
$\xi$ only). Therefore
$\partial_m\phi_\xi(m)=\sum_j\chi_j(\xi)\partial_m\phi^j_\xi(m)$ is
also equal to this unit covector; \as{G4} and~\as{G5} follow. The
dependence of all objects in $m,\xi$ is smooth here. Finally,
\as{G6} can be reduced via $\psi_{\xi_j}$ to the following statement
that can be verified by a direct computation:
if $p\in \mathbb S^n$ and $q,q'\in \mathbb B$,
then $\partial_p\phi^{\mathbb B}_p(q)=\partial_p\phi^{\mathbb B}_p(q')$ if and
only if $q$ and $q'$ lie on a geodesic converging to $p$, and
the matrix $\partial^2_{pq}\phi^{\mathbb B}_p(q)$ has rank $n$.

\subsection{Eisenstein functions and analytic assumptions}
  \label{s:eisenstein.analytic}

Let $(M,g)$ be asymptotically hyperbolic. The Laplacian $\Delta$ on
$(M,g)$ has absolutely continuous spectrum on $[n^2/4,\infty)$ and a
possibly non-empty finite set of eigenvalues in $(0,n^2/4)$.
By~\cite{m-m,Gu}, if $g$ is an even metric\footnote{There is a simpler
proof by Guillop\'e-Zworski~\cite{GuZw} when the curvature is constant
outside a compact set.}, the resolvent of the Laplacian
\[
R(s):=(\Delta-s(n-s))^{-1}  \quad \textrm{ defined in the half plane }{\rm Re}(s)>n/2
\]
admits a meromorphic continuation to the whole complex plane $\cc$,
with poles of finite rank (i.e. the Laurent expansion at each pole
consists of finite rank operators), as a family of bounded operators
\[
R(s): x^NL^2(M)\to x^{-N}L^2(M), \quad \textrm{ if }\,\,\, \Real(s)-n/2+N>0,
\]
moreover it has no poles on the line ${\rm Re}(s)=\ndemi$ except
possibly $s=\ndemi$, as shown by Mazzeo~\cite{Ma}.  
The continuous spectrum for the spectral parameter $s$ correspond to $\Real(s)=n/2$ and we write 
$s=\ndemi+i\la/h$  with $\la>0$ bounded and $h>0$ small for the high-frequency regime.
Let us fix a geodesic boundary defining function $x$ on
$\bbar{M}$. By~\cite{m-m}, the resolvent integral kernel
$R(s;m,m')$ near the boundary $\plM$ has an asymptotic expansion given
as follows: for any $m\in M$ fixed
\[
m'\mapsto R(s;m ,m' )x(m')^{-s} \in  C^\infty(\bbar{M})
\]
and similarly for $m'\in M$ fixed and $m\to \plM$. Since we are interested in high frequency asymptotics, 
we will consider the semiclassical rescaled versions
\begin{equation}
  \label{e:r-h-hyp}
R_h(\lambda):=h^{-2}R(n/2+i\la/ h).
\end{equation}
Note that the physical region $\Real s>n/2$, in which the resolvent is bounded on $L^2$,
corresponds to $\Imag\lambda<0$, which agrees with our convention for Euclidean case,
see~\eqref{e:r-h-eucl}.
%
%
\begin{defi} Let $1/2\leq |\la|\leq 2$ and $h>0$, then 
Eisenstein functions are the functions in $C^\infty(M\x \plM)$ defined for any fixed $\xi\in \plM$
by the following limit of the resolvent kernel at infinity
\begin{equation}
  \label{definitionE}
\begin{gathered}
E_h(\la,\xi;m):=\frac{2i \la h}{C(\la/h)}\lim_{m'\to \xi}
x(m')^{-n/2-i\la/h}R_h(\la;m,m'), \\ 
C(z):=2^{-iz}(2\pi)^{-\ndemi}\frac{\Gamma(\ndemi+iz)}{\Gamma(iz)}.
\end{gathered}
\end{equation}
\end{defi}
%
%
The normalisation constant in \eqref{definitionE} is like the constant
in \eqref{defEeucl} so that in $\mathbb B$, $E_h(\la)$ is a
\emph{horospherical wave} as described below in
\eqref{formulaE_0}. For any $\xi\in \plM$, we will denote by
$E_h(\la,\xi)$ the function $m\mapsto E_h(\la,\xi;m)$, and we observe
that they solve~\eqref{e:e-eq-h}:
$$
(h^2(\Delta-n^2/4)-\la^2)E_h(\la,\xi)=0.
$$
One also has $\bbar{E_h(\la,\xi;m)}=E_h(-\la,\xi;m)$ as an easy
consequence of $R_h(\la)^*=R_h(-\la)$ for $\lambda\in \mathbb R$.  From
its definition, $E_h(\la,\xi)$ depends on the choice of the boundary
defining function $x$, but considering such a change we easily see
from \eqref{measurexi} that the density on $\plM$
\begin{equation}
  \label{densityAEh}
\cjg AE_h(\la,\xi),E_h(\la,\xi)\cjd_{L^2(M)}\,d\xi \quad \textrm{ for }\, A\in \Psi^{\comp}(M)
\end{equation}
is \emph{independent} of $x$. 

Let us recall the decomposition of the spectral measure in terms of these functions. By Stone's formula,
the semiclassical spectral measure is given by
\[
d\Pi_h(\la)=\frac{i\la}{\pi}(R_h(\la)-R_h(-\la))\,d\lambda \quad \textrm{ for }\la\in(0,\infty)
\]
in the sense that $F(h^2(\Delta-n^2/4))=\int_0^\infty
F(\la^2)d\Pi_h(\la)$ for any bounded function $F$ supported in
$(0,\infty)$.  Now we can write (see~\cite{Gu}) for any $m,m'$
\begin{equation}
  \label{greensform}
d\Pi_h(\la;m,m')=\frac{|C(\la/h)|^2}{2\pi h}\int_{\plM}E_h(\la,\xi;m)E_h(-\la,\xi;m')d\xi \,\, d\lambda.
\end{equation}
where $(2\pi h)^n|C(\la/h)|^2\to \lambda^n$ as $h\to 0$
uniformly in $\lambda\in [1/2,2]$. The
assumptions~\as{A1} and~\as{A2} are then satisfied in the general
asymptotically hyperbolic case (without asking the constant curvature
near infinity).

\smallsection{Outgoing/incoming decomposition}
To check assumptions~\as{A3}--\as{A8}, we give a representation of the
Eisenstein functions as sums of the `outgoing' part $E^0_h$ and the
`incoming' part $E^1_h$. We assume constant curvature near
infinity in what follows. The expression for $E_h^{\mathbb B}(\la)$
on hyperbolic space~$\hh^{n+1}$ viewed as a unit ball $\mathbb B$, defined using the boundary defining
function $x_0$ of~\eqref{definx0}, is given by~\cite[Section~2.2]{g-n}
\begin{equation}
  \label{formulaE_0}
\begin{gathered}
E^{\mathbb B}_h(\la,p; q)=\Big(\frac{1-|q|^2}{|q-p|^2}\Big)^{n/2+i\la/ h},\
p\in \mathbb S^n,\ q\in \mathbb B
\end{gathered}
\end{equation}
We thus set $E_h^0(\la,\xi;m)$ to be 
\begin{equation}
\label{definitionEh0}
E_h^0(\la,\xi;m):= e^{(n/2+i\la/h)\phi_\xi(m)},
\end{equation}
where $\phi_\xi$ is the Busemann function defined in
\eqref{definitionphixi}.  Viewing the neighbourhood $U_{\xi}$ as a
subset of one of the $V_{\xi_j}\simeq_{\psi_{\xi_j}} V_{q_0}$ where
$V_{q_0}\subset \mathbb B$ is defined in~\eqref{Uxi0}, the
Laplacian in this hyperbolic chart pulls back to
$\Delta_{\hh^{n+1}}$. Since
$\phi_\xi(m)=\phi^{\mathbb B}_{\psi_{\xi_j}(\xi)}(\psi_{\xi_j}(m))+c_j(\xi)$ for
some function $c_j(\xi)$ independent of $m$, we directly have in
$U_\xi$ ($U_\xi$ is defined in \eqref{Uxi})
\[
(h^2(\Delta-n^2/4)-\la^2)E_h^0(\la;m,\xi)=0.
\]
We let $\chi_0\in C^\infty(\plM\x \bbar{M})$ be a function such that $\chi_0(\xi,\cdot)$ is supported in $U_\xi$, 
equal to $1$ near $\xi$ and smooth in $x^2$.
Therefore we obtain 
\begin{equation}
  \label{PE_0}
\begin{split}
F_h(\la, \xi):=(h^2(\Delta-n^2/4)-\la^2)\chi_0E^0_h(\la,\xi)=[h^2\Delta,\chi_0]E^0_h(\la,\xi)\,\, 
\end{split}
\end{equation}
and we claim that 
\[
\,F_h(\la, \xi)\in x^{\ndemi+2+\frac{i\la}{h}}C^\infty(\bbar{M})
\textrm{ and } \|x^{-1}F_h(\la, \xi)\|_{L^2(M)}=\mc{O}(h)
\]
uniformly in $\xi$. Indeed, this is an elementary calculation since from \eqref{formulaE_0} we see that 
$E^0_h(\la,\xi)\in x^{\ndemi+i\frac{\la}{h}}C^\infty(\bbar{M}\setminus\{\xi\})$
and in geodesic normal coordinates near the boundary 
\[
[\Delta,\chi_0] = -x^2(\pl^2_x\chi_0)-2x(\pl_x\chi_0) x\pl_x +x^2[\Delta_{h(x)},\chi_0]+n(x\pl_x\chi_0)-\demi \Tr_{h(x)}(\pl_xh(x))x^2(\pl_x\chi_0)
\] 
is a first order operator with coefficients vanishing in a neighbourhood of $\xi$.  
We thus correct the error by the incoming resolvent 
$R_h(\lambda)$ by setting  
\begin{equation}
  \label{Epara}
E_h(\la,\xi):=\chi_0E^0_h(\la,\xi)+E^1_h(\lambda,\xi),\,\,\,\, \textrm{ with }\,\, 
E^1_h(\lambda,\xi):=-R_h(\la)F_h(\la,\xi)
\end{equation}
and this makes sense since for $\la\in \rr$, $R_h(\la):x^{\alpha}L^2(M)\to x^{-\alpha}L^2(M)$ for any $\alpha>0$
and $F_h\in xL^2(M)$. We claim that 
%
%
\begin{prop}
  \label{l:eisenstein-dec}
The function $E_h(\la,\xi)$ of \eqref{Epara} is the Eisenstein function defined in \eqref{definitionE} for a 
certain boundary defining function $x$.
\end{prop}
\begin{proof}
Let $R_h^{\mathbb B}(\la)$ be the resolvent of the hyperbolic
space (that is, the incoming right inverse to $h^2(\Delta_{\hh^{n+1}}-n^2/4)-\la^2$) in the ball model and let $\chi_1\in
C^\infty(\plM\x\bbar{M})$ be such that $\chi_1(\xi,\cdot)$ is
supported in $U_\xi$ and $\chi_1\chi_0=\chi_1$. Through the
pull-back by $\psi_{\xi_j}$ (for each $j$), the operator
$R_h^{\mathbb B}(\la)$ induces an operator $R_h^j(\lambda)$ on $V_{\xi_j}$;
if $U_\xi\subset V_{\xi_j}$, then we have the resolvent identity
\begin{equation}
  \label{resolventeq}
R_h(\la)\chi_1=\chi_0 R_h^j(\la)\chi_1-R_h(\la)[h^2\Delta,\chi_0]R_h^j(\la)\chi_1
\end{equation} 
for $\la\in\rr$, the composition makes sense as a map
$x^{\alpha}L^2\to x^{-\alpha}L^2$ for any $\alpha>0$.
Let $x$ be a boundary defining function, so in $V_{\xi_j}$, one has
$\psi_{\xi_j}^*x_0=xe^{\omega_j}$ for some function $\omega_j\in
C^\infty(\plM\cap V_{\xi_j})$. Then multiplying
\eqref{resolventeq}Ê by $x^{-n/2-i\la/h}$ on the right, and taking
the restriction of the Schwartz kernels on $M\x \plM$, we have
\[
E_h(\la,\xi)=\chi_0\til{E}_h^0(\la,\xi)-R_h(\la)[h^2\Delta,\chi_0]\til{E}_h^0(\la,\xi)
\]
with $\til{E}_h^0(\la;m,\xi)= \frac{2i\la h}{C(\la/h)}\lim_{m'\to
\xi}(x(m')^{-\ndemi-i\frac{\la}{h}}R_h^j(\la;m,m'))$ a smooth function
of $m\in U_\xi$ and $C(\la/h)$ the
constant in \eqref{definitionE}.   Note that the Schwartz kernels of $R_h^j$ and $R_h^k$ are the same
on the intersection of their domains, therefore $\widetilde E_h^0$ does not depend on
the choice of $j$.
Now, since $E^{\mathbb B}_h(\la;m,\xi)$ in \eqref{formulaE_0} is the Eisenstein
function on $\mathbb B$ for the defining function $x_0$, we deduce
that in $U_\xi\subset V_{\xi_j}$, one has
$\til{E}_h^0(\la,\xi;m)=E_h^0(\la,\xi;m)e^{(\ndemi+i\frac{\la}{h})(\omega_j(\xi)-c_j(\xi))}$.
Here $c_j(\xi)=\phi_\xi(m)-\phi^{\mathbb B}_{\psi_{\xi_j}(\xi)}(\psi_{\xi_j}(m))$.
Since $E_h^0(\la;m,\xi)$ does not vanish, this shows that on any
intersection $\plM\cap V_{\xi_j}\cap V_{\xi_k}$ of the cover of
$\plM$ by the open sets $V_{\xi_j}\cap \plM$, we get
$\omega_j(\xi)-c_j(\xi)=\omega_k(\xi)-c_k(\xi)$ and therefore this
defines a global smooth function $\theta$ on $\plM$.  In its
definition, $E_h(\la,\xi)$ only depends on the first jet of $x$ at
$\plM$ and thus modifying $x$ to be $xe^{\theta}$, this shows
the claim.
\end{proof}
%
%
It follows that~\as{A3} and~\as{A4} are satisfied, with $b^0=e^{{n\over 2}\phi_\xi(m)}$.
Assumption~\as{A8} is then checked by a direct calculation, with the measure $d\xi$ on
$\plM$ corresponding to the choice of the function $x$ in Proposition~\ref{l:eisenstein-dec}.

Assumption~\as{A7} can be reduced, using the isometries $\psi_{\xi_j}$, to the
following statement: if $(q,\nu)\in S^* \hh^{n+1}$ is directly escaping
in the forward direction and converging to some $p\in \mathbb S^n$, then
$|q-p|\leq C x_0(q)$ for some global constant $C$; the latter statement is
verified directly, see Figure~\ref{f:localize}.
%
%
\begin{figure}
\includegraphics{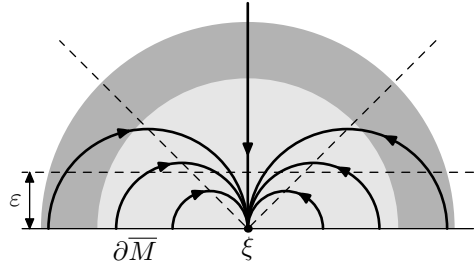}
\caption{Illustration of~(A7) for the half-plane model of $\mathbb H^{n+1}$: the set of points on trajectories converging
to $\xi\in \partial \mathbb H^{n+1}$ with $\dot x<0$ and
$x<\varepsilon$ is the triangle formed by dashed lines, lying
$O(\varepsilon)$ close to $\xi$. For $\varepsilon$ small enough, this
triangle lies inside the lighter shaded region, denoting the set $\{\chi_0=1\}$.}
\label{f:localize}
\end{figure}
%
%

\smallsection{Microlocalization of $E^1_h$}
Finally, assumptions~\as{A5} and~\as{A6} follow, by rescaling $h$ and using
that $E_h(\lambda,\cdot)$ is a function of $\lambda/h$, from
%
%
\begin{prop}\label{WFE}
Let $K_0\subset M$ be a compact set containing a neighborhood of the trapped set.
Assume that $\lambda=1$ and define
\begin{equation}
  \label{tildeE_1}
\widetilde E^1_h(\la,\xi)=\frac{E_h^1(\la,\xi)}{1+\|E_h(\la,\xi)\|_{L^2(K_0)}}.
\end{equation}
Then:

1. $\til{E}^1_h(\la,\xi)$ is $h$-tempered in the sense of \eqref{tempered}.

2. The wavefront set $\WFh(\til{E}^1_h)$ is contained in $S^*M$.

3. If $(m,\nu)\in S^*M$ and $g^t(m,\nu)$ escapes to infinity as $t\to+\infty$
and never passes through the set
$$
W_\xi:=\{(m,\partial_m\phi_\xi(m))\mid m\in \supp(\partial_m\chi_0)\}
$$
for $t\geq 0$, then $(m,\nu)\not\in\WFh(\til E^1_h)$.

Moreover, the corresponding estimates are uniform in $\lambda\in [1/2,2]$ and $\xi\in \pl M$.
\end{prop}
\begin{proof} 
We will use the construction of~\cite{v}. (See also~\cite{v-big};
note however that in that paper $L_+$ and $L_-$ switch places
compared to the notation of~\cite{v} that we are using.)
Let $\overline M_{\even}$
(called $X_{0,\even}$ in~\cite{v})
be the space $\overline M$ with the smooth structure at the boundary
$\plM$ changed so that $x^2$ is the new boundary
defining function. As in~\cite[(3.5)]{v}, introduce the modified Laplacian
$$
P_1(\lambda):=x^{-2}x^{-s}(1+x^2)^{s/4-n/8}(h^2\Delta-s(n-s))(1+x^2)^{n/8-s/4}x^{s},\
s:=n/2+i\lambda/h.
$$
(The conjugation by $(1+x^2)^{s/4-n/8}$ is irrelevant in our case, as
$s/4-n/8=i\lambda/(4h)$ is purely imaginary. In~\cite{v}, it is needed to
show estimates far away in the physical plane, that is for $\Real s\gg 1$.)
Note that we change the sign of $\lambda$ in the conjugation
(in the notation of~\cite{v}, $P_1(\lambda)=P_\sigma$ with $\sigma=-\lambda/h$);
therefore,
our resolvent will be semiclassically incoming, instead of semiclassically
outgoing, for $\lambda>0$.
The operator $P_1$ is smooth up to the boundary of $\overline M_{\even}$; 
as in~\cite[Section~3.5]{v}, we embed $\overline M_{\even}$ as an open set in
a certain compact manifold without boundary $X$, and extend $P_1$
to a differential operator in $\Psi^2(X)$. We also consider the
semiclassical complex absorbing operator $Q(\lambda)\in\Psi^2(X)$ satisfying
the assumptions of~\cite[Section~3.5]{v}; in particular,
$Q(\lambda)$ is supported outside of $\overline M_{\even}\subset X$.
Then $(P_1(\lambda)-iQ(\lambda))^{-1}:C^\infty(X)\to C^\infty(X)$ is a meromorphic
family of operators in $\lambda$, and for $f\in C^\infty(X)$, we have
(see the proof of~\cite[Theorem~5.1]{v})
$$
x^s(1+x^2)^{n/8-s/4}(P_1(\lambda)-iQ(\lambda))^{-1}f|_M=R_h(\lambda)
(1+x^2)^{n/8-s/4}x^s x^2 (f|_M).
$$
Here $R_h(\lambda)$ is the incoming scattering resolvent on $M$. In
principle, depending on the choice of $Q(\lambda)$, the operator
$(P_1(\lambda)-iQ(\lambda))^{-1}$ could have a pole at $\lambda$.
However, as $R(\lambda)$ does not have a pole for $\lambda\in [1/2,2]$,
the terms in the Laurent expansion of
$(P_1(\lambda)-iQ(\lambda))^{-1}$ have to be supported outside of $\overline M_{\even}$
and we can ignore them in the analysis.
%
%
\begin{figure}
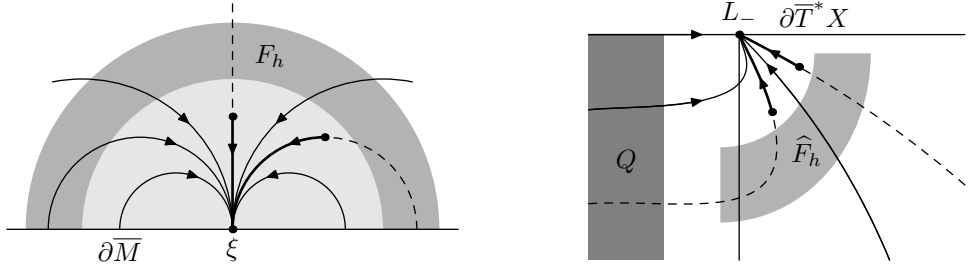

\includegraphics{qeefun.1}
\qquad\qquad
\includegraphics{qeefun.2}
\caption{Left: physical space picture of geodesics converging
to $\xi$. The darker shaded region is the support of $d\chi_0$, and
thus of $F_h$. In the lighter shaded region, $\chi_0=1$.  Right: phase
space picture near $\xi$ after the conjugation of~\cite{v}.  $L_-$ is
the sink consisting of radial points, $Q$ is the complex absorbing
operator, and the shaded region corresponds to the wavefront set of
$\widehat F_h$.  The vertical line hitting $L_-$ is the boundary of
$\overline M_{\even}$, while the horizontal line is the fiber
infinity. In both pictures, we mark two points $(m,\nu)$ satisfying
the assumption of part~3 of Proposition~\ref{WFE} and the forward
geodesics starting at these points.}
\end{figure}
%
%
 
Let $\widehat F_h\in C^\infty(X)$ be any function such that $\widehat F_h=\mc{O}(h)_{H^N_{h}}$
for all $N$, and
$$
F_h=(1+x^2)^{n/8-s/4}x^{s+2}(\widehat F_h|_M).
$$
Such a function exists as $x^{-s}\chi_0E^0_h\in C^\infty(\overline M_{\even}\setminus\xi)$,
$\chi_0\in C^\infty(\overline M_{\even})$, and
$$
F_h=x^{2+s}(1+x^2)^{n/8-s/4}[P_1(s),\chi_0](1+x^2)^{s/4-n/8}x^{-s}E^0_h
$$
is supported away from $\xi$. Define the function
$\widehat E^1_h\in C^\infty(X)$ by
$$
\widehat E^1_h=-{(P_1(\lambda)-iQ(\lambda))^{-1}\widehat F_h\over 1+\|E_h(\lambda,\xi)\|_{L^2(K_0)}}.
$$
Then
$$
\widetilde E^1_h=x^{s}(1+x^2)^{n/8-s/4}\widehat E^1_h|_M.
$$
Consider the map $\iota: T^*M\to T^*X$
given by
$$
\iota( m,\nu)=\bigg(m,\nu
-d\bigg(\ln x(m)-{1\over 4}\ln(1+x(m)^2)\bigg)\bigg),\
m\in M,\
\nu\in T^*_{m}M;
$$
then for an $h$-tempered $u\in C^\infty(X)$,
$$
\WFh(x^s(1+x^2)^{n/8-s/4} u|_M)=\iota^{-1}(\WFh(u)).
$$
Then
\begin{equation}
  \label{e:propagation-rhs}
\WFh((P_1(\lambda)-iQ(\lambda))\widehat E^1_h)\cap T^*M\subset\iota(\WFh(F_h))
\subset\iota(W_\xi).
\end{equation}
Now, as $\|E^0_h\|_{L^2(K_0)}\leq C$ and thus $\|E^1_h\|_{L^2(K_0)}\leq C+\|E_h\|_{L^2(K_0)}$,
we have
$$
\|\widehat E^1_h\|_{L^2(K_0)}\leq C.
$$
Consider an operator $Q_K\in\Psi^{\comp}(X)$ supported in $K_0$ such
that $\sigma(Q_K)\leq 0$ everywhere and each unit speed geodesic
$\gamma(t)$ either escapes as $t\to +\infty$ or passes through the
region $\{\sigma(Q_K)<0\}$ at some positive time.  This is possible
since $K_0$ contains a neighborhood of the trapped set. Then the
operator $P_1(\lambda)-iQ(\lambda)-iQ_K$ satisfies the semiclassical
nontrapping assumptions~\cite[Section~3.5]{v}; therefore, by the
nontrapping estimate~\cite[Theorem~4.8]{v},
$$
\begin{gathered}
\|\widehat E^1_h\|_{L^2(X)}\leq Ch^{-1}\|(P_1(\lambda)-iQ(\lambda)-iQ_K)\widehat E^1_h\|_{L^2(X)}
\\\leq Ch^{-1} \|\widehat F_h\|_{L^2(X)}+Ch^{-1}\|Q_K \widehat E^1_h\|_{L^2(X)}.
\end{gathered}
$$
However, $\|Q_K \widehat E^1_h\|$ is bounded by $\|\widehat E^1_h\|_{L^2(K_0)}$;
therefore, $\|\widehat E^1_h\|_{L^2(X)}=\mc{O}(h^{-1})$ and in particular
$\widehat E^1_h$ is tempered; it follows that $\widetilde E^1_h$ is also tempered.
This proves part~1 of the proposition; part~2 follows by ellipticity
(note that $\WFh(F_h)\subset W_\xi\subset S^*M$).

Now, assume that $(m,\nu)\in S^*M$ satisfies the assumption of part~3
of this proposition.  Then it follows directly
from~\eqref{e:propagation-rhs}, the analysis
of~\cite[Section~2.2]{v}, and the definition of $\iota$, that the
Hamiltonian flow line of $\sigma(P_1)$ starting at $\iota(m,\nu)$
converges to the set $L_-$ of radial points as $t\to +\infty$ and does
not intersect $\WFh((P_1(\lambda)-iQ(\lambda))\widehat E^1_h)$ for
$t\geq 0$.  In a fashion similar to the global argument
of~\cite[Section~4.4]{v} (see also a similar semiclassical outgoing
property of~\cite[Theorem~4.9]{v}), we combine elliptic regularity and
propagation of singularities (see~\cite[Section~4.1]{v}) with the
radial points lemma~\cite[Proposition~4.5]{v} for $L_-$, to get
$\iota(m,\nu)\not\in \WFh(\widehat E^1_h)$. Therefore,
$(m,\nu)\not\in\WFh(\widetilde E^1_h)$ as required.
\end{proof}
%
%

\appendix

\section{Limiting measures for hyperbolic quotients}
  \label{s:k-1}

In this appendix, we give an explicit description of the limiting
measures $\mu_\xi$ in case when $M$ is a hyperbolic quotient
$\Gamma\backslash \hh^{n+1}$, in terms of the group $\Gamma$.
This is a particular case of asymptotically hyperbolic manifolds
discussed in Section~\ref{s:ah}.

\subsection{Convex co-compact groups} 
  \label{s:k-1-groups}

Let $\mathbb B$ be the unit ball in $\rr^{n+1}$, and $\hh^{n+1}$ the 
$(n+1)$-dimensional hyperbolic space, which we view as $\mathbb B$ equipped with the constant negative curvature 
metric $g_{\hh^{n+1}}:=4|dm|^2/(1-|m|^2)^2$. The boundary $\mathbb S^n=\partial \overline{\mathbb B}$
is the sphere of radius $1$, which is also the conformal boundary of $\hh^{n+1}$. 
A \emph{convex co-compact} group $\Gamma$ of isometries of $\hh^{n+1}$ is a
discrete group of hyperbolic transformations (i.e., transformations
having 2 disjoint fixed points on $\overline{\mathbb B}$)
with a compact convex core, and $\Gamma$ is not co-compact. The
convex core is the smallest convex subset in $\Gamma\backslash
\hh^{n+1}$, which can be obtained as follows.  The limit set
$\Lambda_\Gamma$ of the group and the discontinuity set
$\Omega_\Gamma$ are defined by
\begin{equation}
  \label{limitset}
 \Lambda_\Gamma:=\overline{\{\gamma(m)\in \mathbb B; \gamma\in \Gamma\}}
\cap \mathbb S^n\,, \quad \Omega_\Gamma:=\mathbb S^n\setminus \Lambda_\Gamma\end{equation}
where the closure is taken in the closed unit ball $\overline{\mathbb B}$
and $m\in \mathbb B$ is any point (the set $\Lambda_\Gamma$ does not
depend on the choice of $m$).  The group $\Gamma$ acts on the convex
hull of $\Lambda_\Gamma$ (with respect to hyperbolic geodesics) and
the convex core is the quotient space. 

An important quantity is the Hausdorff dimension of $\Lambda_\Gamma$
\begin{equation}
  \label{delta}
\delta:=\dim_{H}\Lambda_{\Gamma} < n
\end{equation} 
which in turn is, by Patterson \cite{Pa} and Sullivan \cite{Su}, the exponent of convergence of Poincar\'e series:
for any $m\in \mathbb B$,
\begin{equation}
\label{patsull}
\sum_{\gamma\in \Gamma} e^{-s d(m,\gamma m)}<\infty \iff s>\delta;
\end{equation}
we henceforth denote by $d(\cdot,\cdot)$ the distance function of the hyperbolic metric on $\mathbb B$.
Notice that the series~\eqref{patsull} is locally uniformly bounded in $m\in \mathbb B$.

The group $\Gamma$ acts properly discontinuously on $\Omega_\Gamma$ as
conformal transformations of the sphere and the quotient space
$\Gamma\backslash \Omega_\Gamma$ is a smooth compact manifold of
dimension $n$. The quotient
$$
M=\Gamma\backslash \hh^{n+1}
$$
is a smooth non-compact manifold equipped with the hyperbolic metric
$g$ induced by $g_{\hh^{n+1}}$, and it admits a smooth
compactification by setting $\bbar{M}=M\cup (\Gamma\backslash
\Omega_\Gamma)$, i.e. with $\plM=\Gamma\backslash \Omega_\Gamma$. Then $M$
is an asymptotically hyperbolic manifold in the sense of Section~\ref{s:ah}, of
constant sectional curvature $-1$.  We shall denote the covering map
by
\[
\pi: \mathbb B \cup \Omega_{\Gamma} \to \bbar{M}.
\]
We refer the reader to~\cite{Ni} for more details and properties of convex co-compact groups.

\subsection{Limiting measures in this setting}

In constant curvature, it turns out that the limiting measure
$\mu_\xi$ exists for all $\xi$ (rather than for Lebesgue almost every
$\xi$ as in Section~\ref{s:general.limiting}), and can be described as
a converging sum over the group. We give an expression below, which is
the same as the one obtained in \cite{g-n} when $\delta<n/2$.

For $\xi\in \mathbb S^n$, we let $\phi_\xi$ be the Busemann function%
\footnote{In Section~\ref{s:ah}, we used the coordinate $q\in\mathbb B$,
$p\in\mathbb S^n$ for certain charts near infinity of $M$, and the notation $\phi^{\mathbb B}_p(q)$ for 
the Busemann function on the ball. This was to avoid confusion with the coordinate $m,\xi$ on $M,\plM$. 
We keep in this appendix the notation $\phi_\xi(m)$ to match the notation of the general setting of the article.}
on the unit ball $\mathbb B$ defined by 
\[\phi_\xi(m)=\log\Big(\frac{1-|m|^2}{|m-\xi|^2}\Big).\]
The map $\Phi$ defined by 
\begin{equation}
  \label{psi}
\Phi: \mathbb B\x \mathbb S^n\to S^*\hh^{n+1}, \quad \Phi: (m,\xi)\mapsto (m,\pl_m\phi_\xi(m))
\end{equation}
gives a diffeomorphism between the unit cotangent bundle $S^*\hh^{n+1}$ and $\mathbb B\x \mathbb S^n$,  and satisfies
\[
\Phi^*d\mu_L=e^{n\phi_{\xi}(m)}\Vol_{\hh^{n+1}}(m)\wedge d\xi, \quad \textrm{ with }\,\, e^{n\phi_{\xi}(m)}
=\Big(\frac{1-|m|^2}{|m-\xi|^2}\Big)^n,
\]
if $d\mu_L$ is the Liouville measure (viewed as a volume form on the unit cotangent bundle) and $d\xi$ the canonical 
measure on $\mathbb S^n$. (This is a more general version of~\as{A8}
for the considered case.)
Any isometry $\gamma$ of $\hh^{n+1}$ acts on both spaces   by 
\[
\begin{gathered}
\gamma.(m,\nu)=(\gamma m, (d\gamma(m)\nu^*)^*), \textrm{ for }(m,\nu)\in S^*\hh^{n+1};\\ 
\quad \gamma.(m,\xi)=(\gamma m,\gamma\xi),\,\, 
\textrm{ for }(m,\xi)\in \mathbb B\x \mathbb S^n,  \end{gathered}
\] 
where $^*$ denotes the map identifying $T^*\hh^{n+1}$ with $T \hh^{n+1}$ through the metric. We have 
$\Phi(\gamma.(m,\xi))=\gamma.\Phi(m,\xi)$ and thus $\Phi$ descends to a map
$\Gamma\backslash (\hh^{n+1}\x \mathbb S^n)\to S^*(\Gamma\backslash \hh^{n+1})$, which we also denote by $\Phi$.

The limiting measure $\mu_\xi$ in the considered case is given by
%
%
\begin{lemm}\label{cccgroup}
Let $M=\Gamma\backslash \hh^{n+1}$ be a quotient of $\hh^{n+1}$ by a
convex co-compact group $\Gamma$ of isometries, let $\mc{F}$ be a
fundamental domain. Then the measure $\mu_{\pi(\xi)}$ of
\eqref{e:mu-xi-def} exists for all $\xi\in \Omega_\Gamma$ and is
described as a converging series by the following expression: if
$\xi\in \Omega_\Gamma\cap\bbar{\mc{F}}$ and $a\in C_0^\infty(S^*M)$,
then
\[
\int_{M} a\, d\mu_{\pi(\xi)}= \int_{\mc{F}}\sum_{\gamma\in \Gamma}a(m,d\phi_{\gamma \xi}(m))
e^{n(\phi_{\gamma\xi}(m)+\log |d\gamma(\xi)|)} \Vol_{\hh^{n+1}}(m)
\]
where $\phi_\xi(m)$ is the Busemann function
on $\mathbb B$ associated to $\xi\in \mathbb S^n$ and $|d\gamma(\xi)|$ is the
Euclidean norm of $d\gamma(\xi)$.
\end{lemm}
\begin{proof} 

We can view $a$ as a compactly supported function on the unit
cotangent bundle $S^*\mc{F}$ over a fundamental domain $\mc{F}\subset
\mathbb B$ and we extend $a$ by $0$ in $S^*\hh^{n+1}\backslash
S^*\mc{F}$ (the resulting function might not be smooth, but it does
not matter here). The flow $g^t$ on $S^*M$ is obtained by projecting
down the geodesic flow $\til{g}^t$ of the cover $S^*\hh^{n+1}$. Let
$\xi\in \Omega_\Gamma\cap\bbar{\mc{F}}$, then small neighbourhoods of
$\pi(\xi)$ in $M$ are isometric through $\pi$ to small neighbourhoods
of $\xi$ in the unit ball $\mathbb B$. By the construction of
the decomposition~\eqref{e:e-h-decomposition} for the asymptotically
hyperbolic case in Section~\ref{s:eisenstein.analytic}, the function
$E_h^0(\la,\pi(\xi);\pi(m))$ is equal to
$e^{(n/2+i\lambda/h)\phi_\xi(m)}$ for $m$ near $\xi$ ($\xi$ being fixed)
and thus $|b^0|^2=e^{n\phi_\xi(m)}$. One has
\[
\int_{M}a(g^{-t}(m,d\phi_\xi(m)))e^{n\phi_{\xi}(m)}\Vol_M(m)=
\int_{\mc{F}}\til{a}(\til{g}^{-t}\Phi(m,\xi))e^{n\phi_{\xi}(m)}\Vol_{\hh^{n+1}}(m)
\]
where $\til{a}(m,\nu):=\sum_{\gamma\in \Gamma}a(\gamma.(m,\nu))$ is
the lift to $S^*\hh^{n+1}$ of the function $a$ on $S^*M$ and
$\Vol_{\hh^{n+1}}(m)$ is the Riemannian measure on $\hh^{n+1}$.  Using
the map $\Phi$ of \eqref{psi}, one can define a map
$\til{g}^t_\xi:\mathbb B\to \mathbb B$ by
\[
\til{g}^t\Phi(m,\xi)=\Phi(\til{g}_\xi^t(m),\xi),
\]
this is a diffeomorphism which preserves the measure $e^{n\phi_\xi(m)} \Vol_{\hh^{n+1}}$. 
By~\cite[Lemma~4]{g-n},
we have $e^{n\phi_\xi(\gamma^{-1}m)}=e^{n\phi_{\gamma\xi}(m)}|d\gamma(\xi)|^{n}$,
but we also have $\gamma.\Phi(m,\xi)=\Phi(\gamma m,\gamma \xi)$. Let $U^+_\infty$
be defined in~\as{G4} and put $U:=\{m\mid (m,\pi(\xi))\in U^+_\infty\}$,
then $U$ lies in a small neighborhood of $\pi(\xi)$ in $M$.
We can identify 
$U$ with a small neighbourhood $\widetilde{U}$ of $\xi$ in $\mc{F}$ and we get
for $\tilde\mu_{\pi(\xi)}$ defined in~\eqref{e:mu-xi-tilde},
\begin{equation}
  \label{tfixed}
\begin{split}
\int_{S^*M} (a\circ g^{-t})\,d\tilde\mu_{\pi(\xi)}=&\int_{U}a(g^{-t}(m,d\phi_\xi(m)))e^{n\phi_{\xi}(m)}\Vol_M(m)\\
=&\int_{\widetilde{U}}\sum_{\gamma\in \Gamma}a(\gamma.\til{g}^{-t}\Phi(m,\xi))
e^{n\phi_{\xi}( m)}\Vol_{\hh^{n+1}}(m)\\
=&\sum_{\gamma\in \Gamma}\int_{\widetilde{U}}a(\Phi(\gamma \til{g}_\xi^{-t}m,\gamma \xi))
e^{n\phi_{\xi}(m)}\Vol_{\hh^{n+1}}(m)\\
=&\sum_{\gamma\in \Gamma}\int_{\gamma\til{g}_\xi^{-t}(\widetilde{U})} a(m,d\phi_{\gamma\xi}(m))
e^{n\phi_{\xi}(\gamma^{-1}m)}\Vol_{\hh^{n+1}}(m).
\end{split}
\end{equation}
We now observe that for all $\gamma\in \Gamma$,
$\lim_{t\to+\infty}\indic_{\gamma\til{g}_\xi^{-t}\widetilde{U}}=1$,
since $\widetilde U$ is a neighbourhood of $\xi$ in $\mathbb B$ containing all
points directly escaping to $\xi$. This achieves the proof by
recalling the definition~\eqref{e:mu-xi-def} of $\mu_{\pi(\xi)}$ and
taking the limit in \eqref{tfixed} and using the dominated convergence
theorem, as there exists $C,C'>0$ such that for all $m$ in the compact
set $\supp(a)$
\[
\sum_{\gamma\in \Gamma} e^{n\phi_{\xi}(\gamma^{-1}m)} = \sum_{\gamma\in \Gamma} 
\Big(\frac{1-|\gamma^{-1} m|}{|\gamma^{-1}m-\xi|^2}\Big)^n\leq C\sup_{m\in\supp(a)} e^{-n\, d(\gamma^{-1}m,0)}\leq C'
\]
by locally uniform (in $m$) convergence of Poincar\'e series \eqref{patsull} at $s=n$.
\end{proof}
%
%

\section{The escape rate}\label{s:escaperate}

Let us discuss the classical escape rate in some particular cases,
following the work of Bowen-Ruelle~\cite{BoRu}, Young~\cite{Yo}, and
Kifer~\cite{Ki}.

\subsection{Escape rate and the pressure of the unstable Jacobian}\label{ap:escape}
  \label{s:escaperate-1}

We consider $(M,g)$ a complete non-compact Riemannian manifold and say that
a compact set $K_0\subset S^*M$  is geodesically convex
if any geodesic trajectory in $S^*M$ which leaves $K_0$ never comes
back:
\begin{equation}\label{geodconvex}
\exists t_1, \exists t_0<t_1,\,\,  g^{t_0}(m,\nu)\in K_0 \textrm{ and } g^{t_1}(m,\nu)\in M\setminus
K_0\Longrightarrow \forall t\geq t_1,\,\, g^t(m,\nu)\in M\setminus K_0.
\end{equation}
A compact set $K_0\subset M$ is said geodesically convex if $\pi^{-1}(K_0)$ is geodesically convex 
where $\pi:S^*M\to M$ is the natural projection.
Let $K_0\subset S^*M$ be a geodesically convex compact set containing a neighborhood of the trapped set $K$.
The examples we
consider are $(M,g)$ which are hyperbolic or Euclidean near infinity,
and $K_0=S^*M\cap \{x\geq \eps_0\}$ with $x,\eps_0$ given in~\as{G2}. The
trapped set from Definition~\ref{d:trapped} can be written as
$$
K=\bigcap_{t\in\rr}g^t(K_0)=\bigcap_{j\in\zz}g^j(K_0)
$$
This is a compact maximal invariant set for the flow $g^t$. We
define the escape rate as in \cite{Yo,Ki} by
\[
Q:=\limsup_{t\to \infty}\frac{1}{t}\log \mu_L(\mathcal T(t)),
\]
with $\mu_L$ the Liouville measure and $\mathcal T(t)$ defined in~\eqref{e:T-t}.
Note that, since $K_0$ is geodesically convex, we have $\mathcal T(t_2)\subset \mathcal T(t_1)$
for $0\leq t_1\leq t_2$. The escape rate is clearly non-positive. 

In this section, we assume that $\mu_L(K)=0$ and write $Q$ in terms of
the topological pressure of the flow, under certain dynamical
assumptions.  More precisely, we assume that the trapped set $K$ is
\emph{uniformly partially hyperbolic}, in the following sense: there
exists $\varepsilon_f>0$ and a splitting of
$T(S^*M)$ over $K$ into continuous subbundles invariant under the flow
\[
T_{z}S^*M=E^{cs}_z\oplus E^u_z, \quad \forall z\in K
\]
such that the dimensions of $E^u$ and $E^{cs}$ are constant on $K$ and for all $\eps>0$, there is $t_0\in \rr$ such that 
\[\forall z\in K, \,\, \forall t\geq t_0,  
\left\{\begin{array}{ll}
\forall v\in E^{u}_z,\,\, |dg_z^t v|\geq e^{\eps_f t}|v|,\\
\forall v\in E^{cs}_z, \,\, |dg^t_zv|\leq  e^{\eps t}|v|.
\end{array}\right.\]
Let $J^u$ be the unstable Jacobian of the flow, defined by  
\[J^u(z):=-\pl_{t}(\det dg^t_z|_{E^u_z})|_{t=0}\]
where $dg^t: E^u_z\to E^u_{g^t_z}$ and
the determinant is defined using the Sasaki metric for choosing orthornormal bases 
in $E^u$.  If $\mu$ is a $g^t$-invariant measure on $K$, one has 
\[
\int_{K}J^u d\mu=-\int_{K}\sum_j \Lambda_j^+ d\mu
\]
where $\Lambda_j^+(z)$ are the positive Lyapunov exponents at a regular
point $z\in K$ counted with multiplicity (regular points are points
where the exponents are well defined, and this is set of full
$\mu$-measure by the Oseledec theorem). It is also direct to see that
$\int_{K}J^u d\mu=-\int_{K}\log \det (dg^1|_{E^u})d\mu$.

The \emph{topological pressure} of a continuous function $\varphi:K\to \rr$ with respect to the flow 
can be defined by the variational formula  
\begin{equation}
  \label{topopressure}
P(\varphi):=\sup_{\mu\in \mc{M}(K)}\bigg(h_{\mu}(g^1)+\int_K \varphi\, d\mu\bigg)
\end{equation}
where $\mc{M}(K)$ is the set of $g^t$-invariant Borel probability
measures and $h_\mu(g^1)$ is the measure theoretic entropy of the flow
at time $1$ with respect to $\mu$. In particular $P(0)$ is the
topological entropy of the flow.

A particular case of uniformly partially hyperbolic dynamics is when
$K$ is \emph{uniformly hyperbolic}, that is when there is a continuous
$g^t$-invariant splitting $E^{cs}=\rr H_p\oplus E^s$ into flow direction ($H_p$ is the
vector field generating the geodesic flow) and stable directions $E^s$
where for $t\geq t_0$
\[
\forall v\in E^{s}_z,\,\, |dg_z^t v|\leq e^{-\eps_f t}|v|.
\]
The flow is said to
be \emph{Axiom A} when the trapped set $K$ is a uniformly hyperbolic set such that 
the periodic orbits of $g^t$ on $K$ are dense
in $K$.

It is proved by Young \cite[Theorem~4]{Yo} that if $K$ is uniformly partially hyperbolic, then  
\begin{equation}
\label{youngformula}
Q=\lim_{t\to \infty}\frac{1}{t}\log \mu_L(\mathcal T(t))=P(J^u).
\end{equation}
In the Axiom A case, the same formula was essentially contained in the
work of Bowen-Ruelle (using the volume lemma \cite[Lemma~4.2 and~4.3]{BoRu}).
Moreover by \cite[Theorem~5]{BoRu}, if the incoming tail
$\Gamma_-$ (which is the union of stable manifolds over the trapped
set) has Liouville measure $0$, then $P(J^u)<0$. Thus we deduce
by~\eqref{e:mes0}
\[
\mu_L(K)=0 \textrm{ and }  g^t \textrm{ is Axiom A } \Longrightarrow P(J^u)<0.
\] 
Young~\cite[Theorem~4]{Yo} gives a lower bound $Q\geq P(-\sum_j \Lambda_j^+)$
which applies without any assumption on $K$ (but we are more
interested in an upper bound).

\subsection{Relation with fractal dimensions in particular cases} 
  \label{s:escaperate-2}

Assume first that the metric has constant curvature $-1$ in a small
neighbourhood of the trapped set $K$ (this includes the case of convex
co-compact hyperbolic quotients studied in Appendix~\ref{s:k-1}).
Then the geodesic flow on $S^*M$ is uniformly hyperbolic on $K$ and
has Lyapunov exponents $0$ (with multiplicity 1) and $\pm 1$ (each
with multiplicity $n$). Therefore, the maximal expansion rate
$\Lambda_{\max}$ from~\eqref{lamax} is equal to 1, one has $J^u(z)=-n$
for all $z\in K$, and (see for example~\cite[Theorem~4]{Fa})
\begin{equation}
P(J^u)=h_{\rm top}(K)-n=(\dim_{H}(K)-1)/2-n
\end{equation} 
where $h_{\rm top}$ is the topological entropy of the flow on $K$, and
$\dim_H(K)\in(0,n)$ is the Hausdorff dimension of $K$ (which is equal
to the Minkowski box dimension in this case).  For convex co-compact
hyperbolic quotients $\Gamma\backslash\hh^{n+1}$ (see
Section~\ref{s:k-1} for definition), one has by Sullivan~\cite{Su2}
\begin{equation}
  \label{deltadef}
\delta:=\dim_{H}(\Lambda_\Gamma)=h_{\rm top}(K)\end{equation} 
where $\Lambda_\Gamma$ is the limit set of the group $\Gamma$ defined in \eqref{limitset}.

If $g$ has negative pinched curvature near the trapped set, then one
still has upper and lower bounds on $P(J^u)$ in terms of $h_{\rm
top}(K)$ and the pinching constant. If the trapped set $K$ is
uniformly hyperbolic, it is also shown in~\cite{Fa} that
$\dim_H(K)\leq 1+2h_{\rm top}(K)/\Lambda_{\max}$. In dimension~$2$
there is an explicit relation between the Hausdorff dimension
$\dim_{H}(K)$ and pressures for Axiom~A cases: if
\[
a^{u}(z)=\lim_{t\to 0}\frac{1}{t}\log\|dg^t|_{E^u}\|>0, \quad a^{s}(z)=\lim_{t\to 0}\frac{1}{t}\log\|Dg^t|_{E^u}\|<0
\]
then Pesin--Sadovskaya~\cite{PeSa} show the following formula
\[
\dim_{H}(K)=1+t^u+t^s, \quad \textrm{ with } P(-t^u a^u)=P(-t^sa^s)=0.
\]
 
\section{Egorov's theorem until Ehrenfest time}
  \label{s:ehrenfest}

In this section, we prove Proposition~\ref{l:ehrenfest}, following the methods
of~\cite{bou-ro}, \cite[Section~5.2]{AnNon}, and~\cite[Theorem~11.12]{e-z}.
See also~\cite[Theorem~7.1]{gabriel}.
Without lack of generality, we assume that $t_0>0$.

\subsection{Estimating higher derivatives of the flow}

First of all, we need to estimate the derivatives of symbols under
propagation for long times. Consider the open set
$$
U_1=\{(m,\nu)\in T^*M\mid m\in U,\ 1-2\varepsilon_e<|\nu|_g<1+2\varepsilon_e\}.
$$
For each $k$, we fix a norm $\|\cdot\|_{C^k(\overline{U_1})}$ for the
space $C^k(\overline {U_1})$ of $k$ times differentiable functions on
$\overline{U_1}$. (The particular choice of the norm does not matter,
as long as it does not depend on $t$.)  The following estimate is an
analogue of~\cite[(5.6)]{AnNon}; we include the proof for the case of
manifolds for the reader's convenience.
%
%
\begin{lemm}
  \label{l:ehrenfest-basic}
Take $\Lambda_1>(1+2\varepsilon_e)\Lambda_{\max}$.  Then for each $k$,
there exists a constant $C(k)$ such that for each $a\in C_0^\infty(\overline{U_1})$
and each $t\in \mathbb R$,
\begin{equation}
  \label{e:ehrenfest-basic}
\|a\circ g^t\|_{C^k(\overline{U_1})}\leq C(k) e^{k\Lambda_1|t|}\|a\|_{C^k(\overline{U_1})}.
\end{equation}
\end{lemm}
\begin{proof}
Without loss of generality, we assume that $t>0$.
We first recall the formula for derivatives of the composition
$b\circ \psi$ of a function~$b\in C^\infty(\mathbb R^d)$ with a
map $\psi:\mathbb R^d\to \mathbb R^d$:
\begin{equation}
  \label{e:composition-formula}
\partial^\alpha(b\circ\psi)=\sum_{\alpha,j}c_{\alpha,j}(\partial_{j_1\dots j_m}b)\circ \psi\cdot
\prod_{l=1}^m \partial^{\alpha_l}\psi_{j_l},
\end{equation}
where $c_{\alpha,j}$ are constants, $j_1,\dots,j_m\in \{1,\dots,d\}$,
and $\alpha_1,\dots,\alpha_m$ are nonzero multiindices whose sum
equals $\alpha$. We see from~\eqref{e:composition-formula} that~\eqref{e:ehrenfest-basic}
is implied by the following estimate on the derivatives of the flow $g^t$
(required to hold in any coordinate system):
\begin{equation}
  \label{e:flow-estimate}
|\alpha|\leq k\Longrightarrow
\sup_{U_1\cap g^{-t}(U_1)}|\partial^\alpha g^t|\leq C_\alpha e^{|\alpha|\Lambda_1t}.
\end{equation}
The converse is also true, which can be seen by substituting
cooordinate functions in place of $a$ in~\eqref{e:ehrenfest-basic}.

To estimate higher derivatives of the flow, we will need several
definitions from differential geometry. For a vector field $X$ on
$\overline{U_1}$, define its pushforward $g^t_*X$ by
$$
X(a\circ g^t)=((g^t_*X) a)\circ g^t,\
a\in C^\infty(g^t(\overline{U_1})).
$$
Then $g^t_* X$ is a vector field on $g^t(\overline{U_1})$. In local coordinates,
we have
$$
(g_*^tX)^j=\sum_l (X^l\partial_l g^t_j)\circ g^{-t}.
$$
Note that since
$g^t=\exp(tH_p/2)$
and $g^t_* H_p=H_p$, we have
\begin{equation}
  \label{e:flow-dir}
\partial_t g^t_* X=-{1\over 2}[H_p,g^t_* X]=-{1\over 2}g^t_*[H_p,X].
\end{equation}
We fix a symmetric affine connection $\nabla$ on $T^*M$.
For vector fields $X$ and $Y$, consider
the differential operator $\nabla^2_{XY}$, acting on functions or on vector
fields, defined as follows: for a function $f$ and a vector field $Z$,
\begin{equation}
  \label{e:nabla-2}
\nabla^2_{XY}f=XYf-(\nabla_XY)f,\quad
\nabla^2_{XY}Z=\nabla_X\nabla_Y Z-\nabla_{\nabla_XY}Z.
\end{equation}
In local coordinates, we have (using Einstein's summation convention)
$$
\begin{gathered}
\nabla^2_{XY}f=X^iY^j(\partial^2_{ij}f-\Gamma_{ij}^l \partial_l f),\\
(\nabla^2_{XY}Z)^m=X^iY^j(\partial^2_{ij}Z^m+\Gamma^m_{j\alpha}\partial_iZ^\alpha
+\Gamma^m_{i\alpha}\partial_j Z^\alpha-\Gamma_{ij}^\alpha \partial_\alpha Z^m\\
+(\partial_i \Gamma_{j\alpha}^m+\Gamma_{i\beta}^m\Gamma_{j\alpha}^\beta
-\Gamma_{ij}^\beta\Gamma_{\alpha\beta}^m)Z^\alpha).
\end{gathered}
$$
Here $\Gamma_{ij}^l$ are the Christoffel symbols of the connection
$\nabla$. The advantage of $\nabla^2_{XY}$ over $XY$ is that the
coefficients of this differential operator at any point depend
(bilinearly) only on the values of $X$ and $Y$ at this point, but not
on their derivatives.

We now return to the proof of~\eqref{e:ehrenfest-basic}. The
estimate~\eqref{e:flow-estimate} for $k=1$ follows directly from the
definition~\eqref{e:lambda-max} of~$\Lambda_{\max}$.  It is then
enough to assume that~\eqref{e:flow-estimate} holds for some $k\geq 1$
and prove the estimate~\eqref{e:ehrenfest-basic} for $k+1$.  It
suffices to show that for any two vector fields $X,Y$ on $T^*M$ and
any $a\in C_0^\infty(U_1)$, we have the estimate
\begin{equation}
  \label{e:eb-internal}
\|XY(a\circ g^t)\|_{C^{k-1}(\overline{U_1})}\leq C e^{(k+1)\Lambda_1 t}\|a\|_{C^{k+1}(\overline{U_1})}.
\end{equation}
The left-hand side of~\eqref{e:eb-internal} is equal to
$\|(g^t_*Xg^t_*Y a)\circ g^t\|_{C^{k-1}(\overline U_1)}$.  We first
claim that
\begin{equation}
  \label{e:eb-internal-2}
\|(\nabla^2_{g^t_*Xg^t_*Y}a)\circ g^t\|_{C^{k-1}(\overline {U_1})}\leq
C e^{(k+1)\Lambda_1 t}\|a\|_{C^{k+1}(\overline{U_1})}.
\end{equation}
Indeed, in local coordinates
\begin{equation}
  \label{e:eb-internal-2.1}
(\nabla^2_{g^t_*Xg^t_*Y}a)\circ g^t=(X^\alpha \partial_\alpha g_i^t)(Y^\beta \partial_\beta g_j^t)
\big((\partial^2_{ij}a-\Gamma_{ij}^l \partial_l a)\circ g^t\big).
\end{equation}
We can now apply~\eqref{e:composition-formula} to get an expression
for any derivative of order no more than $k-1$ of~\eqref{e:eb-internal-2.1}.
The result will involve derivatives of orders $1,\dots,k$ of $g^t$, but
not its $k+1$'st derivative; therefore, we can apply~\eqref{e:flow-estimate}
to get~\eqref{e:eb-internal-2}.

Given~\eqref{e:eb-internal-2} and~\eqref{e:nabla-2}, it is enough to show
\begin{equation}
  \label{e:eb-internal-3}
\|((\nabla_{g^t_*X}g^t_*Y)a)\circ g^t\|_{C^{k-1}(\overline{U_1})}
\leq Ce^{(k+1)\Lambda_1 t}\|a\|_{C^k(\overline{U_1})}.
\end{equation}
The vector field $\nabla_{g^t_*X}g^t_*Y$ involves the second
derivatives of $g^t$, therefore the left-hand side
of~\eqref{e:eb-internal-3} depends on the $k+1$'st derivatives of
$g^t$ and we cannot apply~\eqref{e:flow-estimate} directly. We will
instead use the method of the proof of~\cite[Lemma~2.2]{bou-ro},
computing by~\eqref{e:flow-dir}
$$
\begin{gathered}
\partial_t(g_*^{-t}(\nabla_{g_*^tX}g^t_*Y))
={1\over 2}g_*^{-t}([H_p,\nabla_{g_*^tX}g^t_*Y]
-\nabla_{[H_p,g_*^tX]}g_*^tY
-\nabla_{g_*^tX}[H_p,g^t_*Y])
={1\over 2}g^{-t}_*Z_t,
\end{gathered}
$$
where $Z_t$ is the vector field given by
$Z_t= \nabla^2_{g_*^tXg_*^tY}H_p+R_\nabla(H_p,g_*^tX)(g_*^tY)$.
Here $R_\nabla$ is the curvature tensor of the connection $\nabla$.
Then
\begin{equation}
  \label{e:eb-internal-7}
\nabla_{g^t_*X}g^t_* Y=g^t_*(\nabla_XY)+{1\over 2}
\int_0^tg_*^{t-s}Z_s\,ds.
\end{equation}
We have
$$
\begin{gathered}
\|(g^t_*(\nabla_XY)a)\circ g^t\|_{C^{k-1}(\overline{U_1})}
=\|\nabla_XY(a\circ g^t)\|_{C^{k-1}(\overline {U_1})}\\
\leq C\|a\circ g^t\|_{C^k(\overline {U_1})}\leq Ce^{k\Lambda_1 t}
\|a\|_{C^k(\overline {U_1})}.
\end{gathered}
$$
It is then enough to handle the integral part of~\eqref{e:eb-internal-7}.
The field $Z_s$ depends quadratically on the first derivatives
of $g^s$, but does not depend on its higher derivatives;
therefore, writing an expression for $Z_s$ in local coordinates
similar to~\eqref{e:eb-internal-2.1}, we get for $a\in C_0^\infty(U_1)$,
$$
\|(Z_s a)\circ g^s\|_{C^{k-1}(\overline{U_1})}\leq Ce^{(k+1)\Lambda_1s}\|a\|_{C^k(\overline{U_1})}.
$$
Applying~\eqref{e:ehrenfest-basic} for the $C^k$ norm
(given by the induction hypothesis) and using the geodesic convexity of $U$,
we get
$$
\begin{gathered}
\int_0^t \|((g_*^{t-s}Z_s) a)\circ g^t\|_{C^{k-1}(\overline{U_1})}\,ds
=\int_0^t\|(Z_s(a\circ g^{t-s}))\circ g^s\|_{C^{k-1}(\overline{U_1})}\,ds\\\leq
C\int_0^t e^{(k+1)\Lambda_1 s}\|a\circ g^{t-s}\|_{C^k(\overline{U_1})}\,ds\leq
C\int_0^t e^{(k+1)\Lambda_1 s}e^{k\Lambda_1(t-s)}\|a\|_{C^k(\overline {U_1})}\,ds\\\leq
Ce^{(k+1)\Lambda_1t}\|a\|_{C^k(\overline{U_1})}
\end{gathered}
$$
and the proof is finished.
\end{proof}
%
%

\subsection{Proof of Proposition~\ref{l:ehrenfest}}

The proof of Proposition~\ref{l:ehrenfest} is based on repeatedly
applying the following corollary of
Lemma~\ref{l:ehrenfest-basic}. The functions $b^{(j)}$ below
will be the remainders in the formula for the commutator
$[h^2\Delta,A^{(j)}(t)]$, while the functions $c^{(j)}$ will be the errors
arising from multiplying our operators by $X_1$ and $X_2$.
%
%
\begin{prop}\label{l:ehrenfest-basic-2}
Take $\Lambda_1>(1+2\varepsilon_e)\Lambda_{\max}$.
Fix $t_0>0$ and let $\varphi\in C_0^\infty(U_1)$ satisfy $|\varphi|\leq 1$.
Assume that $a_0\in C^\infty(T^*M)$ and for each
$j\geq 0$, $b^{(j)}(t)\in C^\infty([0,t_0]\times T^*M)$,
and $c^{(j)}\in C^\infty(T^*M)$,
with support contained in some $j$-independent compact set.
For $j\geq 0$, define $a^{(j)}\in C^\infty([0,t_0]\times T^*M)$
inductively as the solutions to the equations
$$
\begin{gathered}
a^{(0)}(0)=a_0,\
a^{(j+1)}(0)=\varphi\cdot a^{(j)}(t_0)+c^{(j+1)};\\
\partial_t a^{(j)}(t)={1\over 2}H_p a^{(j)}(t)+b^{(j)}(t).
\end{gathered}
$$
Then for each $k$, and each $j$, we have (bearing in mind that each
$a^{(j)}$ is supported inside some $j$-independent compact set and
thus its $C^k$ norm is well-defined up to a constant)
$$
\begin{gathered}
\sup_{t\in [0,t_0]}\|a^{(j)}(t)\|_{C^k(T^*M)}\leq C(k) \big(e^{jk\Lambda_1t_0}\|a_0\|_{C^k}\\
+\max_{0\leq i\leq j}e^{(j-i)k\Lambda_1t_0}(\sup_{t\in[0,t_0]}\|b^{(i)}(t)\|_{C^k}+\|c^{(i)}\|_{C^k})\big),
\end{gathered}
$$
where $C(k)$ is a constant independent of $j$.
\end{prop}
\begin{proof}
We can write
$$
a^{(j)}(t)=a^{(j)}(0)\circ g^t+\int_0^t b^{(j)}(s)\circ g^{t-s}\,ds.
$$
Since $t_0$ is fixed, it is enough to estimate the derivatives
of $a^{(j)}(0)$. Define
$$
\varphi^{(j)}=\prod_{0\leq m<j} (\varphi\circ g^{mt_0});
$$
applying the Leibniz rule to $\varphi^{(j)}$, estimating each
nontrivial derivative of $\varphi\circ g^{mt_0}$ by
Lemma~\ref{l:ehrenfest-basic}, using that $|\varphi|\leq 1$ and
absorbing the (polynomial in $l$) number of different terms in the
Leibniz formula into the exponential by increasing $\Lambda_1$
slightly, we get $\|\varphi^{(j)}\|_{C^k}=\mathcal
O(e^{jk\Lambda_1t_0})$.  Now,
$$
\begin{gathered}
a^{(j)}(0)=\varphi^{(j)}\cdot(a_0\circ g^{jt_0})
+\sum_{i=0}^{j-1}\varphi^{(j-i)}\int_0^{t_0} b^{(i)}(s)\circ g^{(j-i)t_0-s}\,ds\\
+\sum_{i=1}^j \varphi^{(j-i)}\cdot (c^{(i)}\circ g^{(j-i)t_0}).
\end{gathered}
$$
Here we put $\varphi^{(0)}=1$.
We can now apply Lemma~\ref{l:ehrenfest-basic} again to get
the required estimate.
\end{proof}
%
%
We are now ready to prove Proposition~\ref{l:ehrenfest}. Fix a quantization
procedure $\Op_h$ on $M$; our symbols will be supported in a certain
compact set (in fact, no more than distance $t_0$ to the set $U$) and
we require that the corresponding operators be compactly supported.
Put $\Lambda_1=\Lambda'_0$.
Let $l$ satisfy~\eqref{e:ehrenfest-l}. We will construct the operators
$$
A^{(j)}_m(t)=\Op_h\bigg(\sum_{0\leq m'\leq m}a^{(j)}_{m'}(t)\bigg),\ 0\leq t\leq t_0,\
0\leq j\leq l,\
m\geq 0,
$$
Here the symbols $a^{(j)}_m$ will be supported in a fixed compact subset of $T^*M$
and satisfy the derivative bounds
\begin{equation}
  \label{e:ehrenfest-ders}
\sup_{t\in [0,t_0]}\|a^{(j)}_m(t)\|_{C^k}\leq C(k,m)h^{(1-2\rho_j)m-\rho_j k}.
\end{equation}
with the constants $C(k,m)$ independent on $j$ and $\rho_j$ defined
by~\eqref{e:rho-j}. The operators $A^{(j)}_m(t)$ will satisfy the
relations
\begin{equation}
  \label{e:ehrenfest-eqns}
\begin{gathered}
A^{(0)}_m(0)=A+\mathcal O(h^\infty)_{\Psi^{-\infty}},\\
A^{(j+1)}_m(0)=X_2 A^{(j)}_m(t_0)X_1+\Op_h(c^{(j)}_m)+
\mathcal O(h^\infty)_{\Psi^{-\infty}},\\
hD_t A^{(j)}_m(t)={1\over 2}[h^2\Delta,A^{(j)}_m(t)]+{h\over i}\Op_h(b^{(j)}_m(t))
+\mathcal O(h^\infty)_{\Psi^{-\infty}},
\end{gathered}
\end{equation}
where the symbols $b^{(j)}_m(t)$ and $c^{(j)}_m$ are supported in some
fixed compact set and satisfy bounds
\begin{equation}
  \label{e:ehrenfest-ders-2}
\sup_{t\in[0,t_0]}\|b^{(j)}_m(t)\|_{C^k},\|c^{(j)}_m\|_{C^k}\leq C(k,m)h^{(1-2\rho_j)(m+1)-\rho_j k},
\end{equation}
with the constants $C(k,m)$ again independent on $j$.

We construct the symbols $a^{(j)}_m$ iteratively, by requiring
that they solve the equations
$$
\begin{gathered}
a^{(0)}_m(0)=\delta_{m0}\cdot a_0,\
a^{(j+1)}_m(0)=\varphi a^{(j)}_m(t_0)-c^{(j)}_{m-1},\\
\partial_t a^{(j)}_m(t)={1\over 2}H_p a^{(j)}_m(t)-b^{(j)}_{m-1}(t).
\end{gathered}
$$
Here $A=\Op_h(a_0)+\mathcal O(h^\infty)_{\Psi^{-\infty}}$ and we put
$b^{(j)}_{-1}=c^{(j)}_{-1}=0$.  The function $\varphi\in C_0^\infty(U_1)$ is equal to
$\sigma(X_1)\sigma(X_2)\psi(|\nu|)$, where $\psi\in
C_0^\infty(1-2\varepsilon_e,1+2\varepsilon_e)$ is such that
$\psi(|\nu|)=1$ near $\WFh(A)$. We use the fact that the function
$|\nu|$ is invariant under the geodesic flow. The estimate~\eqref{e:ehrenfest-ders}
follows immediately from~\eqref{e:ehrenfest-ders-2} and Proposition~\ref{l:ehrenfest-basic-2}.
As for the equations~\eqref{e:ehrenfest-eqns}
and the bounds~\eqref{e:ehrenfest-ders-2}, they follow from~\eqref{e:ehrenfest-ders} and the
following commutator formula:
$$
[h^2\Delta,\Op_h(a)]={h\over i}\Op_h(H_pa)+\Op_h(b)+\mathcal O(h^\infty)_{\Psi^{-\infty}},\
b=\mathcal O(h^{2-2\rho}\|a\|_{S_\rho})_{S_\rho},
$$
true for any $\rho<1/2$ and any $a\in S^{\comp}_\rho$.

Now, consider the asymptotic sums
$$
a^{(j)}(t)\sim \sum_{m\geq 0} a^{(j)}_m(t)
$$
and define the operators $A^{(j)}(t)=\Op_h(a^{(j)}(t))$.
By~\eqref{e:ehrenfest-eqns}, these operators satisfy
$$
\begin{gathered}
A^{(0)}(0)=A+\mathcal O(h^\infty)_{\Psi^{-\infty}},\
A^{(j+1)}(0)=X_2 A^{(j)}(t_0)X_1+\mathcal O(h^\infty)_{\Psi^{-\infty}},\\
hD_t A^{(j)}(t)={1\over 2}[h^2\Delta,A^{(j)}(t)]+\mathcal O(h^\infty)_{\Psi^{-\infty}}.
\end{gathered}
$$
We then have
$$
(X_2U(t_0))^lA(U(-t_0)X_1)^l=A^{(l)}(0)+\mathcal O(h^\infty)_{L^2\to L^2}.
$$
It remains to recall that $a^{(l)}(0)\in S^{\comp}_{\rho_l}$ uniformly
in $l$.  The principal symbol and microlocal vanishing statements
follow directly from the procedure used to construct the symbols
$a^{(j)}_m$.

\section{Proof of quantum ergodicity in the semiclassical setting}
  \label{s:qe}

In this section, we illustrate how our methods yield a proof of the following
integrated quantum ergodicity statement in the semiclassical setting:
%
%
\begin{theo}
  \label{t:qe}
Let $(M,g)$ be a compact Riemannian manifold of dimension $d$ and
assume that the geodesic flow $g^t$ on $M$ is ergodic with respect to
the Liouville measure $\mu_L$ on the unit cotangent bundle $S^*M$. For
each $h>0$, let $(e_j)_{j\in \mathbb N}$ be an orthonormal basis of
eigenfunctions of $h^2\Delta$ with eigenvalues $\lambda_j^2$. Then for
each semiclassical pseudodifferential operator $A\in\Psi^0(M)$, we
have
\begin{equation}
  \label{e:qe-int-2}
h^{d-1}\sum_{\lambda_j\in [1,1+h]} \bigg|\langle A e_j,e_j\rangle_{L^2(M)}
-{1\over \mu_L(S^*M)}\int_{S^*M} \sigma(A)\,d\mu_L\bigg|\to 0\text{ as }h\to 0.
\end{equation}
\end{theo}
%
%
A more general version of Theorem~\ref{t:qe} was proved
in~\cite{h-m-r}, in particular relying on the result
of~\cite{d-g,PeRo} on $o(h)$ remainders for the Weyl law when the
closed geodesics form a set of measure zero.  The purpose of this
Appendix is to provide a shorter proof. Theorem~\ref{t:qe} is
formulated here for the semiclassical Laplacian for simplicity of
notation, but it applies to any self-adjoint semiclassical
pseudodifferential operator $P(h)$ with compact resolvent on a compact
manifold, if the Hamiltonian flow of the principal symbol $p$ of
$P(h)$ has no fixed points and is ergodic on the energy surface
$p^{-1}(0)$ and we take eigenvalues in the interval $[0,h]$.

The key component of our proof is the following estimate:
%
%
\begin{lemm}
  \label{l:qe-key}
Let $M$ be as in Theorem~\ref{t:qe}. Then for each $A\in\Psi^0(M)$, we have
\begin{equation}
  \label{e:qe-key}
h^{d-1}\sum_{\lambda_j\in [1,1+h]}\|Ae_j\|_{L^2(M)}^2\leq (C\|\sigma(A)\|_{L^2(S^*M)}
+\mathcal O(h))^2.
\end{equation}
Here $\|\sigma(A)\|_{L^2(S^*M)}$ is the $L^2$ norm of the restriction
of $\sigma(A)$ to $S^*M$ with respect to the Liouville measure.  The
constant in $\mathcal O(h)$ depends on $A$, but the constant $C$ does
not.
\end{lemm}
\begin{proof}
Assume first that $A$ is compactly microlocalized. We can rewrite the
left-hand side of~\eqref{e:qe-key} as the square of the
Hilbert--Schmidt norm of $h^{(d-1)/2}A\Pi_{[1,1+h]}$, where
$\Pi_{[1,1+h]}=\indic_{[1,(1+h)^2]}(h^2\Delta)$ is a spectral projector. It
can then be estimated using the local theory of semiclassical Fourier
integral operators, by~\eqref{e:h-s-estimate-1} (applied
to the adjoint of the operator in interest).

To handle the case of a general $A$, it remains to note that if
$\WFh(A)\cap S^*M=\emptyset$, then the left-hand side
of~\eqref{e:qe-key} is $\mathcal O(h^\infty)$, as each $Ae_j$ is
$\mathcal O(h^\infty)$ by the elliptic estimate
(Proposition~\ref{l:elliptic}; see also the proof of Proposition~\ref{l:avg-bdd}).
\end{proof}
%
%
Putting $A$ equal to the identity in~\eqref{e:qe-key}, we get the following
upper Weyl bound:
\begin{equation}
  \label{e:upper-weyl}
\#\{j\mid \lambda_j\in [1,1+h]\}\leq Ch^{1-d}.
\end{equation}
We can now prove Theorem~\ref{t:qe}. Take $A\in\Psi^0(M)$; by
subtracting a multiple of the identity operator and applying the
ellipticity estimate, we may assume that $A$ is compactly
microlocalized and
\begin{equation}
  \label{e:avg-0}
\int_{S^*M}\sigma(A)\,d\mu_L=0.
\end{equation}
Define the quantum average
$$
\langle A\rangle_T={1\over T}\int_0^T U(t)AU(-t)\,dt.
$$
Here $U(t)=e^{ith\Delta/2}$ is the semiclassical Schr\"odinger propagator.
By Egorov's theorem (Proposition~\ref{l:egorov}), for any fixed $T$ the
operator $\langle A\rangle_T$ lies in $\Psi^0$, modulo an $\mathcal O(h^\infty)_{L^2\to L^2}$
remainder, and its principal symbol is
$$
\sigma(\langle A\rangle_T)=\langle \sigma(A)\rangle_T:={1\over T}\int_0^T \sigma(A)\circ g^t\,dt.
$$
Note that for each $j$, we have $U(t)e_j=e^{it\lambda_j/(2h)}$ and
thus $\langle\langle A\rangle_T e_j,e_j\rangle=\langle A
e_j,e_j\rangle$.  Using Cauchy--Schwarz inequality in $j$ and the
bounds~\eqref{e:qe-key} and~\eqref{e:upper-weyl}, we get
$$
\begin{gathered}
h^{d-1}\sum_{\lambda_j\in[1,1+h]}|\langle Ae_j,e_j\rangle|
=h^{d-1}\sum_{\lambda_j\in [1,1+h]}|\langle \langle A\rangle_Te_j,e_j\rangle|
\\\leq h^{d-1}\sum_{\lambda_j\in [1,1+h]}\|\langle A\rangle_T e_j\|_{L^2}
\leq C\bigg(h^{d-1}\sum_{\lambda_j\in [1,1+h]}\|\langle A\rangle_T e_j\|_{L^2}^2\bigg)^{1/2}
\\\leq C\|\langle\sigma(A)\rangle_T\|_{L^2(S^*M)}+\mathcal O_T(h).
\end{gathered}
$$
However, by~\eqref{e:avg-0} and the von Neumann ergodic
theorem~\cite[Theorem~15.1]{e-z}, we have $\|\langle
\sigma(A)\rangle_T\|_{L^2(S^*M)}\to 0$ as $T\to \infty$. Therefore,
for each $\varepsilon>0$ we can choose $T$ large enough so that the
left-hand side of~\eqref{e:qe-int-2} is bounded by
$\varepsilon/2+\mathcal O(h)$.  Then for $h$ small enough, it is
bounded by $\varepsilon$; since the latter was chosen arbitrarily
small, we get~\eqref{e:qe-int-2}.

\smallsection{Acknowledgements}
We would like to thank Viviane Baladi, Dima Jakobson, Frederic Naud,
St\'ephane Nonnenmacher, Steve Zelditch, and Maciej Zworski for useful
discussions and providing references on the subject. We additionally
thank St\'ephane Nonnenmacher for explaining some estimates on higher
derivatives of the flow (forming the basis for
Lemma~\ref{l:ehrenfest-basic}), and the anonymous referees for their
interest in this work and suggesting many improvements.
S.D. would also like to thank the DMA of
Ecole Normale Sup\'erieure where part of this work was done. S.D. was
partially supported by NSF grant DMS-0654436. C.G. is supported by ANR
grant ANR-09-JCJC-0099-01.


\def\arXiv#1{\href{http://arxiv.org/abs/#1}{arXiv:#1}}

\end{document}